\documentclass[11pt]{article}

\usepackage[margin=1in]{geometry}
\usepackage{amsmath,amssymb,amsthm,mathtools}
\usepackage{mathrsfs}
\usepackage{enumitem}
\usepackage[colorlinks=true,linkcolor=blue,citecolor=blue,urlcolor=blue]{hyperref}
\usepackage{xcolor}
\usepackage[nameinlink,noabbrev,capitalize]{cleveref}
\usepackage{algorithm}
\usepackage{algpseudocode}

\usepackage{graphicx}
\usepackage{subcaption} 

\usepackage{color-edits}
\addauthor{zx}{orange}

\theoremstyle{definition}
\newtheorem{definition}{Definition}[section]
 
\newtheorem{condition}[definition]{Condition}
\theoremstyle{plain}
\newtheorem{theorem}[definition]{Theorem}
\newtheorem{lemma}[definition]{Lemma}
\newtheorem{proposition}[definition]{Proposition}
\newtheorem{corollary}[definition]{Corollary}
\newtheorem{question}{Question}
\newtheorem{problem}{Problem}

\theoremstyle{remark}
\newtheorem{remark}[definition]{Remark}
\newtheorem{example}[definition]{Example}

\newcommand{\R}{\mathbb{R}}
\newcommand{\X}{\mathcal{X}}
\newcommand{\ip}[2]{\left\langle #1,\,#2\right\rangle}

\newcommand{\sign}{\operatorname{sign}}
\newcommand{\dist}{\operatorname{dist}}

\newcommand{\calS}{\mathcal{S}}
\newcommand{\calL}{\mathcal{L}}
\newcommand{\DeltaLS}{\Delta_{\mathrm{LS}}}

\title{Function-free Optimization via Comparison Oracles}
\author{Katya Scheinberg\thanks{H. Milton School of Industrial and Systems Engineering, Georgia Institute of Technology, Atlanta, Georgia, USA.   \href{mailto:katya.scheinberg@isye.gatech.edu}{katya.scheinberg@isye.gatech.edu}} \and Zikai Xiong\thanks{H. Milton School of Industrial and Systems Engineering, Georgia Institute of Technology, Atlanta, Georgia, USA.  \href{mailto:zxiong84@gatech.edu}{zxiong84@gatech.edu}}}
\date{May 2026}

\begin{document}
\maketitle

\begin{abstract}
In this work, we study optimization specified only through a comparison oracle: given two points, it reports which one is preferred. We call it function-free optimization because we do not assume access to, nor the existence of, a canonical application-given objective function. Instead, our goal is to find a most-preferred feasible point, which we call an optimal solution. This model arises in preference- and ranking-based settings, where the objective values and derivatives are unavailable, meaningless, or non-identifiable. Even if a representative function exists for the preference relation, it may be nonsmooth, nonconvex, or even discontinuous. We develop an analytical and algorithmic framework based on the geometry of preference level sets, which remains well-defined from comparisons alone. We introduce a new optimality measure, the level-set optimality gap, defined as the distance from the preference level set to the optimal solutions, and also the regularity radius, which plays the role of a stationarity certificate. Under regularity of the preference relation in a $d$-dimensional Euclidean space, we propose a method for estimating normal directions to accuracy $\epsilon$ using $O(d \, \log(d/\epsilon))$ comparisons, nearly matching a lower bound of $\Omega(d\, \log(1/\epsilon))$. Under convexity and regularity of the preference relation and a local growth condition on the regularity radius, the resulting normal direction descent method reaches an $\epsilon$ level-set optimality gap using at most $\widetilde O(dD^2/\epsilon^2)$ comparisons, over $O(D^2/\epsilon^2)$ normal direction estimation steps. Here $D$ is the distance from the initial point to the optimal set. This number of normal direction estimation steps matches the lower bound of $\Omega(D^2/\epsilon^2)$ for normal direction span-based methods. Since prior knowledge in practical applications of function-free optimization is usually very limited, we also develop adaptive schemes for both estimating the normal direction and solving the optimization problem. These adaptive schemes match the fixed-parameter complexity bounds up to logarithmic factors.
\end{abstract}




\section{Introduction}
\label{sec:introduction} 
Comparison- and preference-based feedback is often the natural form of information in applications where objective values are unavailable or not clearly defined. In search and ranking systems, user behavior can reveal relative preferences rather than utility values \cite{yue2009interactive,yue2012karmed}. Related preference-based optimization models also arise in recommender systems and A/B testing, where a latent objective is queried through pairwise comparisons \cite{gonzalez2017preferential}. In human-in-the-loop personalization and clinical tuning, comparisons between alternatives are more reliable than numerical scoring \cite{tucker2020preference,zhao2021optimization}. In preference-based reinforcement learning and reinforcement learning from human feedback, signals are commonly collected as preferences over trajectory segments or rankings of model outputs rather than through direct access to a reward function \cite{furnkranz2012preference,christiano2017deep,ouyang2022training}. These examples suggest that, in many applications, the observable object is not a numerical objective value but a preference relation between feasible alternatives.

Comparing the objective values of two (or more) points, as component of an algorithm, has long been used in derivative-free optimization. However, most algorithms
regard this procedure   as a way of using function values, or zeroth-order oracles, rather than as a separate (weaker) oracle. For example, classical direct-search methods compare trial points (in terms of their objective values) with the best point found so far \cite{hooke1961direct}, and the Nelder--Mead method is likewise order-based \cite{nelder1965simplex,lagarias1998convergence}. These methods tend to have poor or no convergence guarantees. 
More recently, several algorithms  that use function comparisons and yet have some complexity guarantees (under certain structural assumptions on the objective function) have been proposed for zeroth-order optimization \cite{jamieson2012query,bergou2020stochastic,golovin2020gradientless} and gradient estimation \cite{cai2022one,tao2026gradient}. 
Paper \cite{jamieson2012query} is the first in this line of work which we are aware of that explicitly presents the operation of comparing two points as a so-called ``comparison oracle,''
thus explicitly avoiding the need to compute function values.

 ``Comparison oracle,'' namely an oracle that returns the relative order of the function values at two queried points have appeared in a variety of recent papers. 
 Comparison oracle based optimization is also studied under the name of dueling optimization \cite{saha2021dueling,saha2025dueling,ren2026riemannian}. Related notions sub-zeroth order oracles \cite{karabag2021smooth} and order oracles \cite{lobanov2024acceleration} are other names of comparison oracles. Beyond pairwise comparisons, ranking oracles can also be obtained from multiple pairwise comparisons \cite{tang2024zeroth}.

Both the classical and the recent methods treat comparisons primarily as a weak interface to a latent scalar objective function. Specifically,  comparison oracle is generated by value differences of an underlying function. More importantly the guarantees are stated in terms of properties of that function, such as function-value gaps, regret, or gradient norms and the complexity results, when they exists, rely on the properties of the function.  This creates a gap with the underlying applications. In preference-driven problems, there may not exist an application-given scalar objective function. Even when an objective function representation exists, it may be non-identifiable from comparisons, nonunique, or only a surrogate introduced for modeling convenience. The authors of \cite{golovin2020gradientless}  observe the fact that multiple functions can result in identical comparison oracles and state their results in terms of the "best" strongly convex function that "fits" the oracle.

In summary, prior work aims to optimize a scalar function through weaker oracles. In contrast, our central goal is  to formulate an optimization problem when the preference relation itself is the basic object defining optimality. The main question motivating this paper is: 
\begin{question}\label{question1}
How should we formulate and analyze optimization using only a preference relation, without choosing any application-given scalar objective function representative?
\end{question}
This question is not merely a matter of notation. 
Figure~\ref{figs:levelsets} shows three functions that are dramatically different.  One is convex and smooth except for the optimal solution,  another is nonconvex, and another is discontinuous. They have different graphs but their level sets have the same shape (spheres centered at the origin).  
They induce the same preference relation: $x$ is preferred to $y$ if and only if $\|x\|_2\le \|y\|_2$.
Hence, these three functions are indistinguishable from the perspective of comparisons, and running any comparison based algorithm on these three different functions will observe the same comparison feedbacks and yield exactly the same trajectories of iterates. 
Therefore, manually choosing one representative function introduces artificial information that is invisible to comparisons. To answer Question \ref{question1}, we need assumptions, algorithms, and guarantees expressed directly in terms of the preference relation.

\begin{figure}[h]
    \begin{center}
        \includegraphics[width=\linewidth]{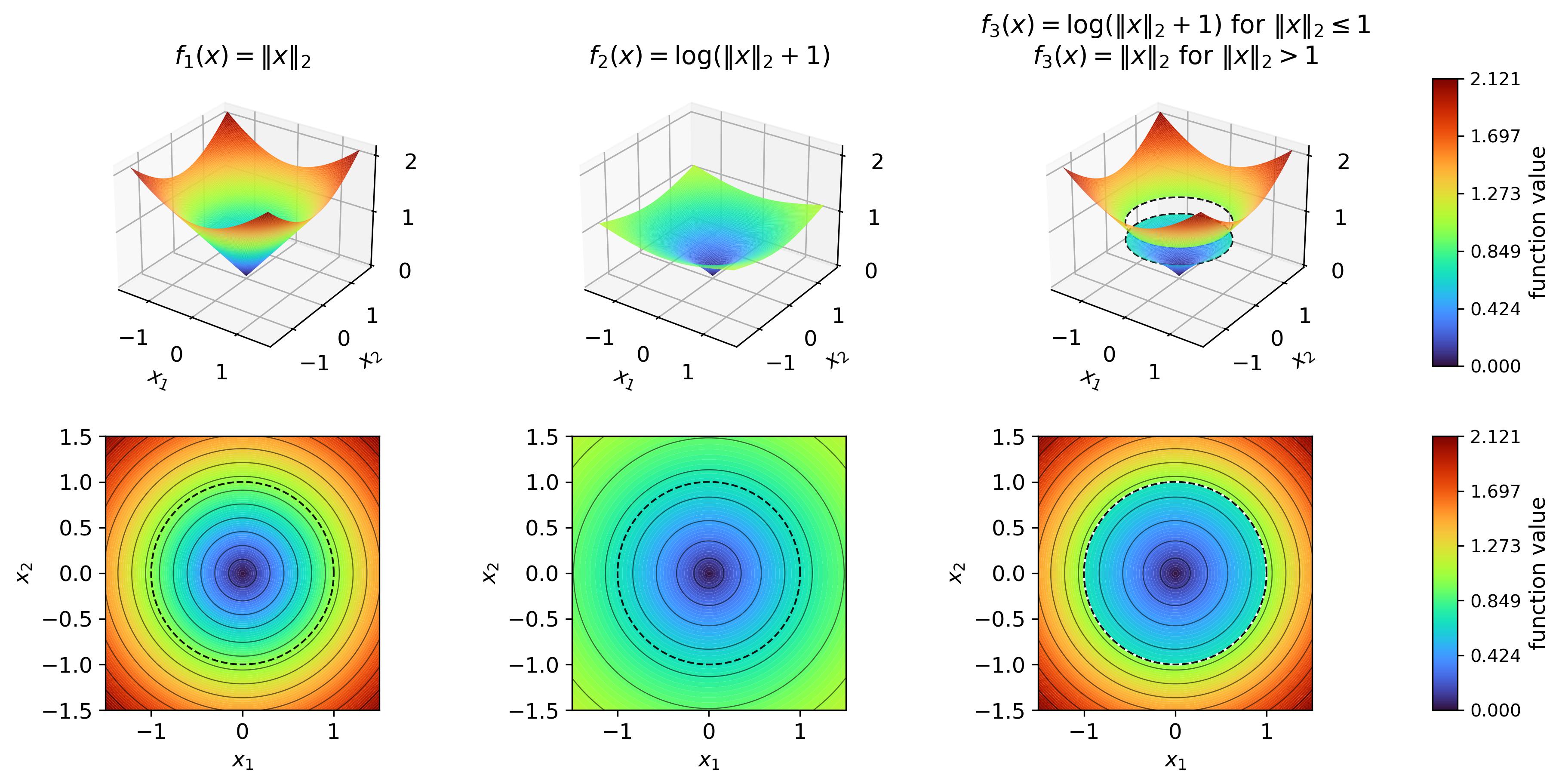}
    \end{center} 
    \caption{Graphs and level sets of three functions on $x=(x_1,x_2)\in\mathbb{R}^2$: $f_1(x):=\|x\|_2$, $f_2(x):=\log(\|x\|_2+1)$, and $f_3(x):=\log(\|x\|_2+1)$ if $\|x\|_2\le 1$, and $f_3(x):=\|x\|_2$ otherwise.}
    \label{figs:levelsets}
\end{figure}

Our answer to Question \ref{question1} is a \emph{function-free optimization} framework, which is also our first main contribution.  
We model the available information by a binary relation $\preccurlyeq$ on a $d$-dimensional Euclidean space $\X$, interpreted as a preference relation:
\begin{equation}\label{eq:preference-relation}
x\preccurlyeq y \quad \text{means that $x$ is preferred to $y$; equivalently, $x$ is no worse than $y$.}
\end{equation}  
Given a feasible set $\mathcal C \subseteq \X$, the optimization task is to identify a most-preferred feasible point under $\preccurlyeq$:
\begin{equation}
\label{eq:main-optimization-problem}
\text{find }x \in \mathcal{C}  \ , \quad \text{s.t.} \ \ 
    x\preccurlyeq y\ \ \text{for all} \ \  y\in\mathcal{C}.
\end{equation} 
This is the main optimization problem studied in this paper.  We call $x^\star\in\mathcal{C}$ an optimal solution if it is no worse than every feasible point.  Equivalently, the set of all optimal solutions is
\begin{equation}
\label{eq:optimal-set}
\X^\star\;:=\;\{x\in\mathcal{C}:\ x\preccurlyeq y \text{ for all } y\in\mathcal{C}\}.
\end{equation}   

To measure approximate optimality in this framework, we propose new measures defined purely through the preference relation.  The main near-optimality measure in this paper is the \emph{level-set optimality gap}: intuitively, the preference level set through $x$ consists of all points that are indifferent to $x$, and the level-set optimality gap measures the distance from this preference level set to the optimal set $\X^\star$.  This quantity is defined entirely from the preference relation and the Euclidean geometry.  We also introduce the \emph{regularity radius}, a local geometric scale describing how regularly the preference level set behaves near the current point.  A small regularity radius is a function-free analogue of near-stationarity, and under a local growth condition it can also imply near-optimality.  The formal definitions of preference level sets, the level-set optimality gap, the regularity radius, and the assumptions on the preference relation are given in Section~\ref{sec:co-foundations}.  Importantly, assumptions such as plateau-freeness, convexity, and regularity are imposed directly on the preference relation.

Once having the function-free optimization framework, the next question is algorithmic: how can one solve the problem without gradients or function values?  In smooth optimization, the gradient is critical because it indicates how the function changes and which direction improves the objective.  In function-free optimization, a particular gradient is not available and it is not intrinsic to the preference relation. The replacement used in this paper is the \emph{normal direction} to the current preference level set. Roughly speaking, the normal direction is the direction in which one leaves the current preference level fastest. Thus, moving in the opposite direction usually leads to better points. Therefore, in this paper, we are interested in estimating normal directions and using normal directions to design algorithms.

Throughout the paper, both the normal direction estimation problem and the optimization problem are guided by two principles: optimality and adaptivity.  Optimality matters because we seek worst-case complexity guarantees that cannot be substantially improved in general. It is a valuable benchmark for understanding the fundamental limits and remaining potential of an algorithm.
Adaptivity matters because, in applications with only ordinal feedback, the relevant local preference parameters are rarely known in advance, so the algorithm should not rely on parameters that need knowledge of them. Therefore, we seek comparison based methods that are near-optimal in the worst case while also adjusting automatically to unknown local geometry.

This leads to the remaining two questions studied in this paper:
\begin{question}\label{question2}
Can normal directions be estimated from comparisons as efficiently as possible, while remaining adaptive to local preference relation geometry?
\end{question}
\begin{question}\label{question3}
Once normal directions can be estimated, can they be used to design optimization methods that are likewise optimal and adaptive?
\end{question}
 
Our answers to Questions \ref{question2} and \ref{question3} form our second and third main contributions of this paper, below we give a high-level summary.

Our answer to Question \ref{question2} is a near-optimal adaptive normal-direction estimation method that uses only pairwise comparisons.  The key idea is an orthogonal-complement construction. Instead of estimating the normal direction directly, the method constructs $d-1$ approximately orthogonal directions to the normal direction, and then the one-dimensional orthogonal complement approximates the normal direction.
The method estimates the normal direction to accuracy $\epsilon$ using
$O\!\left(d\log\frac{d}{\epsilon}\right)$  pairwise comparisons, provided the comparison radius is chosen small enough.  We also prove a minimax lower bound of order $\Omega\!\left(d\log\frac{1}{\epsilon}\right)$ comparisons for normal-direction estimation.  Hence the fixed-radius estimator is nearly optimal.
We then give an adaptive version for convex preference relations.  This version uses line search to choose the comparison radius automatically and it achieves the same leading dependence on $d$ and $\epsilon$, with only logarithmic factor overhead.  

To answer Question \ref{question3}, we study mainly in the class of \emph{normal-direction span-based methods} for convex preference relations, which parallels the first-order methods in classical convex optimization.  In this class, we propose the normal direction descent method (NDD), and establish a two-case guarantee in terms of the level-set optimality gap and the regularity radius. To ensure smaller level-set optimality gap, the method needs to use more comparisons, while getting a smaller regularity radius only requires choosing a smaller comparison radius when estimating normal directions.
Under a local growth condition of the regularity radius, this yields an $\epsilon$ level-set optimality gap guarantee of consuming at most  $\widetilde O\!\left(\frac{dD^2}{\epsilon^2}\right)$ pairwise comparisons over \(O\left(\tfrac{D^2}{\epsilon^2}\right)\) normal-direction estimation steps.  Here $D$ denotes the distance from the initial point to the optimal solutions, and $\widetilde O (\cdot)$ hides only logarithmic factors of $d$ and $\frac{D}{\epsilon}$. This normal direction complexity is optimal as we prove a minimax lower bound of  $\Omega\!\left(\frac{D^2}{\epsilon^2}\right)$ for all normal-direction span-based methods. We also develop an adaptive method, namely adaNDD. It does not require prior knowledge of $D$ or the local growth parameters, and it matches the rates of NDD up to a logarithmic factor.

Before  concluding this section we want to emphasize an important distinction between two settings induced by comparison oracles. The setting we consider here is based on a noiseless oracle. In the literature,  comparison oracles are often allowed to be noisy, meaning that they may return the wrong order with some probability. Among them, when the oracle gives the correct answer with  probability  $\frac{1}{2}$ plus a constant positive term, which we call "advantage",  repeated independent queries can simulate a noiseless oracle with high probability \cite{saha2021dueling}. 
 A more general model assumes that the probability of giving the correct answer depends on the function-value difference between the two points \cite{jamieson2012query,saha2025dueling}.  The more general model however uses function value in the oracle definition, thus the setting is inherently tied to an objective function and is no longer function-free. Such oracles are thus beyond the study in this paper. 

\subsection{Literature review}\label{subsec:otherwork}

We now summarize the key aspects of the existing relevant literature mentioned above.

Comparison-based rules have a long history in derivative-free optimization. Early direct-search methods iterate by comparing trial solutions with the best point found so far \cite{hooke1961direct}, the Nelder--Mead method  orders simplex vertices by objective value and uses this ordering to choose reflection, expansion, contraction, and shrink steps \cite{nelder1965simplex}.  Modern direct-search and pattern-search methods formalize polling and acceptance rules based on sampled objective values \cite{torczon1997convergence,kolda2003optimization,mads}. However,  Nelder--Mead method has only limited general convergence theory and can converge to non stationary points \cite{lagarias1998convergence,mckinnon1998convergence}, while  direct-search methods that enjoy reasonable complexity guarantees rely on so-called 
{\em  sufficient-decrease conditions} which required knowledge of the difference of  function values and not just how they compare  \cite{vicente2013worst,kolda2003optimization}.  While purely comparison-based variants of direct search methods exist they have much weaker (exponential time) worst-case guarantees.

Some recent zeroth-order methods  use comparison-based update rules and manage to develop complexity guarantees under structural assumptions on the objective function. Bergou et al. \cite{bergou2020stochastic} propose the stochastic three-point method, which samples a direction $s$ and moves to the best point among $x$, $x+\alpha s$, and $x-\alpha s$. They give complexity guarantees for smooth objectives under various convexity assumptions. Golovin et al. \cite{golovin2020gradientless} propose the GradientLess Descent method for smooth strongly convex objective. Each iteration of the method searches around the current point at different radii and iteratively updates to the best point among the search points.  Golovin et al. \cite{golovin2020gradientless} realizes that since the method is purely comparison-based, it also applied to monotone transformations of the objective and thus their complexity results are stated in terms of the "best" strongly convex function that can be obtained under such transformation.

A more recent line of work treats pairwise comparisons as an explicit oracle for optimization. Jamieson et al. \cite{jamieson2012query} give lower bounds for derivative-free optimization with comparison oracles and propose a comparison-based randomized coordinate-descent method for smooth strongly convex objectives, where comparisons are used to perform line search along coordinate directions. Lobanov et al. \cite{lobanov2024acceleration} develop a comparison-based accelerated coordinate-descent method for smooth strongly convex objectives.   Noisy comparison oracles have also been studied. If the oracle returns the correct comparison with a fixed advantage over random guessing, then repeated independent queries can simulate a noiseless comparison oracle with high probability, while the number of queries scales quadratically with the inverse of the advantage over random guessing \cite{jamieson2012query,saha2021dueling}.  
Saha et al. \cite{saha2021dueling} study smooth convex and strongly convex optimization under such fixed-advantage noisy comparisons, and call this setting {\em dueling} convex optimization. Their method uses comparisons of randomly perturbed points  to estimate normalized descent directions and then applies normalized gradient descent.
This was later extended to noisy oracles whose advantage depends on the function-value difference \cite{saha2025dueling}. These work treat comparisons as a weak interface to the objective function, so the convergence guarantees are established in terms of either function-value gap or gradient norms, and the algorithm needs prior knowledge of parameters of the underlying objective (such as smoothness parameter $L$).

Another related line of work focuses on estimating normalized gradients for smooth functions. Cai et al. \cite{cai2022one} assume sparse gradients and use one-bit compressed sensing to estimate a normalized gradient from pairwise comparisons. A related estimator is later used by Chen et al. \cite{chen2025compo} for nonconvex smooth optimization in preference alignment. Karabag et al. \cite{karabag2021smooth} maintain a cone of plausible directions and shrink this cone using half-space cuts extracted from comparison feedback. This method achieves logarithmic dependence on the estimation error $\epsilon$, but its cost grows quadratically with the dimension. They also use a variant of the ellipsoid method to solve smooth convex optimization problems. Zhang and Li \cite{zhang2024comparisons} propose a two-phase method: first finding a dominant gradient component with the largest magnitude, and then estimating the ratio between this component and each of the other components. This improves the dimension dependence, but requires the comparison radius to be in a restrictive range. Very recently, while preparing the first version of this manuscript, we found a concurrent work \cite{tao2026gradient}, which further improves the method of \cite{zhang2024comparisons} by first rotating the basis toward a reference direction that is close to the normalized gradient. This loosens the requirement on the comparison radius, although the guarantee is no longer deterministic. Tao et al. \cite{tao2026gradient} also give a lower bound and show that their method is nearly optimal in terms of the comparison complexity. When the preference relation is realized by a smooth objective, the normalized gradient is identical to the normal direction. Our normal direction estimation method is deterministic and also nearly optimal, and furthermore for convex preference relations, it is adaptive to the local geometry. A more detailed comparison of these methods and our method will be given in Section \ref{sec:normal-estimation}.

Comparison feedback has also been extended to other settings. Saha et al. \cite{saha2024faster} study batched and multiway comparison oracle feedbacks and show that they can improve convergence rates. Tang et al. \cite{tang2024zeroth} use ranking oracles, which return an ordering over several queried points and can be simulated by multiple pairwise comparisons, to estimate descent directions and develop a stochastic gradient descent type method. Very recently, Ren et al. \cite{ren2026riemannian} study optimization on Riemannian manifolds using comparison oracles.  El Bakkali et al. \cite{bakkali2026nonsmooth} study nonsmooth nonconvex optimization with noisy comparison oracles
where the advantage of noisy oracles depends on the function-value difference via a link function whose expression is explicitly known.

\subsection{Organization and notation}

\paragraph{Notation.} 

In this paper, we let $(\X,\langle\cdot,\cdot\rangle)$ be a $d$-dimensional real Euclidean vector space, and let $\|x\|:=\sqrt{\langle x,x\rangle}$ be the induced norm. For a set $A\subseteq\X$, we write $\operatorname{int}(A)$, $\overline A$, and $\partial A$ for its interior, closure, and boundary, respectively. Let $\dist(x,A):=\inf_{y\in A}\|x-y\|$ denote the Euclidean distance from $x$ to $A$, and let $\dist(A_1,A_2) := \inf_{x\in A_1, y\in A_2} \|x-y\|$ denote the distance between two sets $A_1$ and $A_2$ in $\X$. Let $B(x,r)$ denote the closed Euclidean ball centered at $x$ with radius $r$. Let $\mathbb S^{d-1}$ denote the unit sphere in $\X$.  
For set $\mathcal{C}$, we use $\operatorname{diam}(\mathcal{C}):=\sup_{x,y\in\mathcal{C}}\|x-y\|$ to denote its diameter.
When $\mathcal C$ is closed and convex, $\Pi_{\mathcal C}(x)$ denotes the Euclidean projection of $x$ onto $\mathcal C$.  For nonzero vectors $u$ and $v$, we define $\angle(u,v):=\arccos\!\left(\frac{\langle u,v\rangle}{\|u\|\,\|v\|}\right)
\in[0,\pi]$.
The notation $a=O(b)$ means $a\le Cb$ for some absolute constant $C>0$, and $a=\Omega(b)$ means $a\ge cb$ for some absolute constant $c>0$.   Similarly, $a\lesssim b$ means $a\le Cb$ for an absolute constant $C>0$.  The notation $a=o(b)$ means $a/b\to0$ in the stated limiting regime.  We use $\widetilde O(\cdot)$ to suppresse absolute constants and logarithmic factors.

\paragraph{Outline of the paper.} 
The rest of the paper is organized as follows. 
Sections \ref{sec:co-foundations}, \ref{sec:normal-estimation}, and \ref{sec:ngd-fixed-radius} correspond to the three main contributions of the paper, answering Questions \ref{question1}, \ref{question2}, and \ref{question3}, respectively.
In Section~\ref{sec:co-foundations}, we propose the function-free optimization framework. And we also give definitions of the ingredients of the framework, including common assumptions, and optimality and stationarity criteria. We also explain the relation to comparison-based optimization when a scalar objective function exists.
In Section~\ref{sec:normal-estimation}, we study normal direction estimation using pairwise comparisons.  We present the fixed comparison radius normal direction estimator, the adaptive version for convex preference relations, and the lower bound showing near-optimality of these two methods. We also discuss how normal directions directly enable cutting-plane methods.
In Section~\ref{sec:ngd-fixed-radius}, we propose and analyze the normal direction descent method NDD, and its adaptive version adaNDD. We also prove lower bounds for the general class of normal-direction span-based methods and show the near-optimality of NDD and adaNDD. Finally, in Section~\ref{sec:conclusion}, we summarize the contributions and discuss future directions.

\section{Function-free optimization via preferences and comparisons}
\label{sec:co-foundations}

In this section, we introduce the function-free optimization framework. 

\subsection{Preference relation and comparison feedback}
\label{subsec:co-definition}

Recall that the preference relation $\preccurlyeq$ is defined in \eqref{eq:preference-relation}, on a $d$-dimensional Euclidean space $(\X,\langle\cdot,\cdot\rangle)$ (e.g., $\X=\R^d$), equipped with its induced norm $\|\cdot\|$. 
Throughout this paper, we assume the preference relation $\preccurlyeq$ is complete and transitive, defined as follows.
\begin{definition}[Complete and transitive preferences]
\label{def:total-preorder}
The preference relation $\preccurlyeq$ is a complete and transitive preference relation if
\begin{enumerate}[label=(\roman*),leftmargin=2.2em]
\item (\emph{Completeness}) for all $x,y\in\X$, either $x\preccurlyeq y$ or $y\preccurlyeq x$ (or both);
\item (\emph{Transitivity}) for all $x,y,z\in\X$, if $x\preccurlyeq y$ and $y\preccurlyeq z$, then $x\preccurlyeq z$.
\end{enumerate}
\end{definition}

Given $\preccurlyeq$, we define the strict preference relation $\prec$ and the indifference relation $\sim$ by
\[
x\prec y \quad\Longleftrightarrow\quad (x\preccurlyeq y)\ \text{and not}\ (y\preccurlyeq x),
\qquad
x\sim y \quad\Longleftrightarrow\quad (x\preccurlyeq y)\ \text{and}\ (y\preccurlyeq x).
\]
Intuitively, $x\prec y$ means that $x$ is strictly preferred to $y$, while $x\sim y$ means that $x$ and $y$ are tied (indifferent).

Given a feasible set $\mathcal C\subseteq\X$, the goal is to find a point $x^\star\in\mathcal C$ such that $x^\star\preccurlyeq y$ for all $y\in\mathcal C$. We have defined this problem in \eqref{eq:main-optimization-problem} and its solution set, denoted by $\X^\star$, has been defined in \eqref{eq:optimal-set}.
Throughout this paper we assume having access to efficient implementation of Euclidean projections onto $\mathcal{C}$. 
 
Since we do not have numerical objective values, some typical oracles used by classical optimization algorithms (e.g., function values or gradients) are unavailable.
Instead, we access the preference relation $\preccurlyeq$ through pairwise comparisons.
A comparison between two points \(x,y\in\X\) reveals which of the three  cases holds:
\(x\prec y\), \(x\sim y\), or \(y\prec x\). 
\begin{definition}[Comparison oracle]
\label{def:comparison-oracle} 
A comparison oracle is a mapping $\mathscr{C}_{\mathrm{cmp}}: \X \times \X \to \{-1, 0, +1\}$ such that for any $x, y \in \X$:
$\mathscr{C}_{\mathrm{cmp}}(x,y)=-1$ if $y\prec x$, $\mathscr{C}_{\mathrm{cmp}}(x,y)=0$ if $y\sim x$, and $\mathscr{C}_{\mathrm{cmp}}(x,y)=+1$ if $x\prec y$.
\end{definition}
\noindent
Throughout the paper, one use of the comparison oracle is simply called a comparison and the comparison complexity of an algorithm is the total number of comparisons it uses.

Before presenting algorithms for \eqref{eq:main-optimization-problem}, we introduce several basic geometric concepts defined directly from the preference relation $\preccurlyeq$.
These concepts will be used later to design algorithms and to analyze their performance.

\subsection{Preference sublevel sets and level sets}
\label{subsec:preference-levelset}
 
Using the preference relation $\preccurlyeq$, we define the \emph{(preference) sublevel set} and \emph{(preference) level set} at any point $x\in\X$ as follows.

\begin{definition}[Preference sublevel and level sets]
\label{def:level-sets}
For any $x\in\X$, the preference sublevel set at $x$ and the preference level set at $x$ are defined as
\begin{equation}
\label{eq:functionfree-levelsets}
\calS_x\;:=\;\{y\in\X:\ y\preccurlyeq x\},
\qquad
\calL_x\;:=\;\{y\in\X:\ y\sim x\}.
\end{equation}
\end{definition}

We call them ``preference sublevel sets'' and ``preference level sets'' to distinguish them from the usual function value based sublevel sets and level sets.
For simplicity, we will omit the word ``preference'' when there is no risk of confusion.
The sets $\calS_x$ and $\calL_x$ are intrinsic to the preference relation $\preccurlyeq$: they are defined purely from comparisons and do not require any numerical objective function.

\begin{lemma}
\label{lem:sublevel-nesting}
Assume the preference relation $\preccurlyeq$ is complete and transitive.
Then for any $x,y\in\X$, $y\preccurlyeq x$ if and only if $ \calS_y\subseteq \calS_x$. 
In particular, $y\sim x$ if and only if $\calS_y=\calS_x$ (equivalently, $\calL_y=\calL_x$).
Moreover, the set of optimal solutions satisfies $
\X^\star = \bigcap_{x\in\mathcal{C}} \bigl(\calS_x\cap \mathcal{C}\bigr)$.
\end{lemma}

\begin{proof}
If $y\preccurlyeq x$ and $z\in\calS_y$, then $z\preccurlyeq y$; by transitivity, $z\preccurlyeq x$, so $z\in\calS_x$ and $\calS_y\subseteq \calS_x$.
Conversely, if $\calS_y\subseteq \calS_x$, then $y\in\calS_y$ (since $y\preccurlyeq y$), hence $y\in\calS_x$ and $y\preccurlyeq x$.
The equivalence for $y\sim x$ follows by applying the first statement in both directions.

Finally, $z\in\bigcap_{x\in\mathcal{C}}(\calS_x\cap\mathcal{C})$ if and only if $z\in\mathcal{C}$ and $z\preccurlyeq x$ for all $x\in\mathcal{C}$,
which is exactly the definition of $z\in\X^\star$ in \eqref{eq:optimal-set}.
\end{proof}

\paragraph{Preference properties via sublevel-set geometry.}
Below we introduce several geometric regularity properties of the preference relation $\preccurlyeq$ through the sublevel sets $\calS_x$ and level sets $\calL_x$.

\begin{definition}[Plateau-free and convex preferences]
\label{def:well-posed}
We say the preference relation $\preccurlyeq$ is:
\begin{itemize}
\item \emph{plateau-free} if for every $x\in\mathcal{X}\setminus \X^\star$, ties occur only on the boundary of the sublevel set at $x$, i.e., $\partial \calS_x=\calL_x$;   
\item \emph{convex} if for every $x\in\mathcal{X}$, the sublevel set $\calS_x$ is convex.
\end{itemize}
\end{definition}

These properties are stated directly in terms of sublevel sets and level sets, which are essentially equivalent to the preference relation $\preccurlyeq$ (Lemma \ref{lem:sublevel-nesting}) and do not require any numerical objective function.
They also match common modeling assumptions in preference data:
plateau-free preferences rule out regions of exact indifference with nontrivial volume (ties occur on a boundary rather than on a flat plateau), and convexity captures the idea that convex combinations of points that are no worse than $x$ are still no worse than $x$. 
Furthermore, with plateau-freeness, the level set $\calL_x$ is exactly the boundary of the sublevel set $\calS_x$. As $\calL_x$ is contained within $\calS_x$ according to its definition, the sublevel set $\calS_x$ is closed.
In the rest of the paper, we work primarily in the plateau-free setting. 
Convexity will only be invoked when deriving global guarantees for optimization algorithms.

\paragraph{Regularity and normal direction.}
Furthermore, we define the following regularity condition on the preference relation $\preccurlyeq$:
\begin{definition}[Regularity of preference relation]
\label{def:regularity} 
For $x \in \partial \mathcal{S}_x$ and $r \ge 0$, the preference relation $\preccurlyeq$ is $r$-\emph{regular} at $x$, if there exists  $n \in \mathbb{S}^{d-1}$, such that $B(x-r n, r)\subseteq \calS_x$ and $B(x+r n, r)\subseteq \overline{\X\setminus \calS_x}$.
\end{definition}
Here $B(z,r)$ in \ref{def:regularity} denotes the closed Euclidean ball centered at $z$ with radius $r$, i.e., $B(z,r):=\{y\in\X:\|y-z\|\le r\}$.
The condition in \ref{def:regularity}  says that, at scale $r$,  $\calS_x$ contains an interior ball of radius $r$ tangent at $x$ and, simultaneously, its complement contains an exterior ball of radius $r$ tangent at $x$ along direction $n$. The unit direction of $n$, denoted by $n_x := \frac{n}{\|n\|}$ (which is equal to $n$), is the unique outward unit normal direction to $\partial\calS_x$ at $x$. We call $n_x$ the \emph{normal vector} at $x$ for short.

To quantify how ``curved'' the boundary $\partial\calS_x$ is near $x$, we introduce the \emph{regularity radius} $r_x$ of $x$ and the curvature $\kappa_x$ of $x$. 
\begin{definition}[Regularity radius and curvature]  
\label{def:curvature}
Assume the preference relation $\preccurlyeq$ is $r$-regular at $x$ for $r\ge 0$, and let $n_x$ be the outward unit normal direction to $\partial\calS_x$ at $x$. The regularity radius $r_x$ of $x$ is defined as follows:
\begin{equation}
\label{eq:rollingball}
r_x \;:=\; \sup\Bigl\{ r\ge 0:\  B(x-r n_x, r)\subseteq \calS_x \ \text{and}\  B(x+r n_x, r)\subseteq \overline{\X\setminus \calS_x}\Bigr\}\in(0,\infty] \ ,
\end{equation}
and the curvature $\kappa_x$ of $x$ is defined as $\kappa_x := \frac{1}{r_x}$, with the convention $1/\infty=0$ and $1/0 = \infty$.
\end{definition}  
Thus $r_x$ is the largest radius at which $\partial\calS_x$ can be locally supported by two tangent balls. 
Figure \ref{figs:levelset} shows an illustration of the normal direction $n_x$ and the regularity radius $r_x$. 
Large $r_x$ corresponds to a locally ``flat'' boundary, while small $r_x$ corresponds to a highly curved boundary.

\begin{figure}[htbp]
    \centering
    \begin{subfigure}{0.34\textwidth}
        \centering
        \includegraphics[width=1\textwidth]{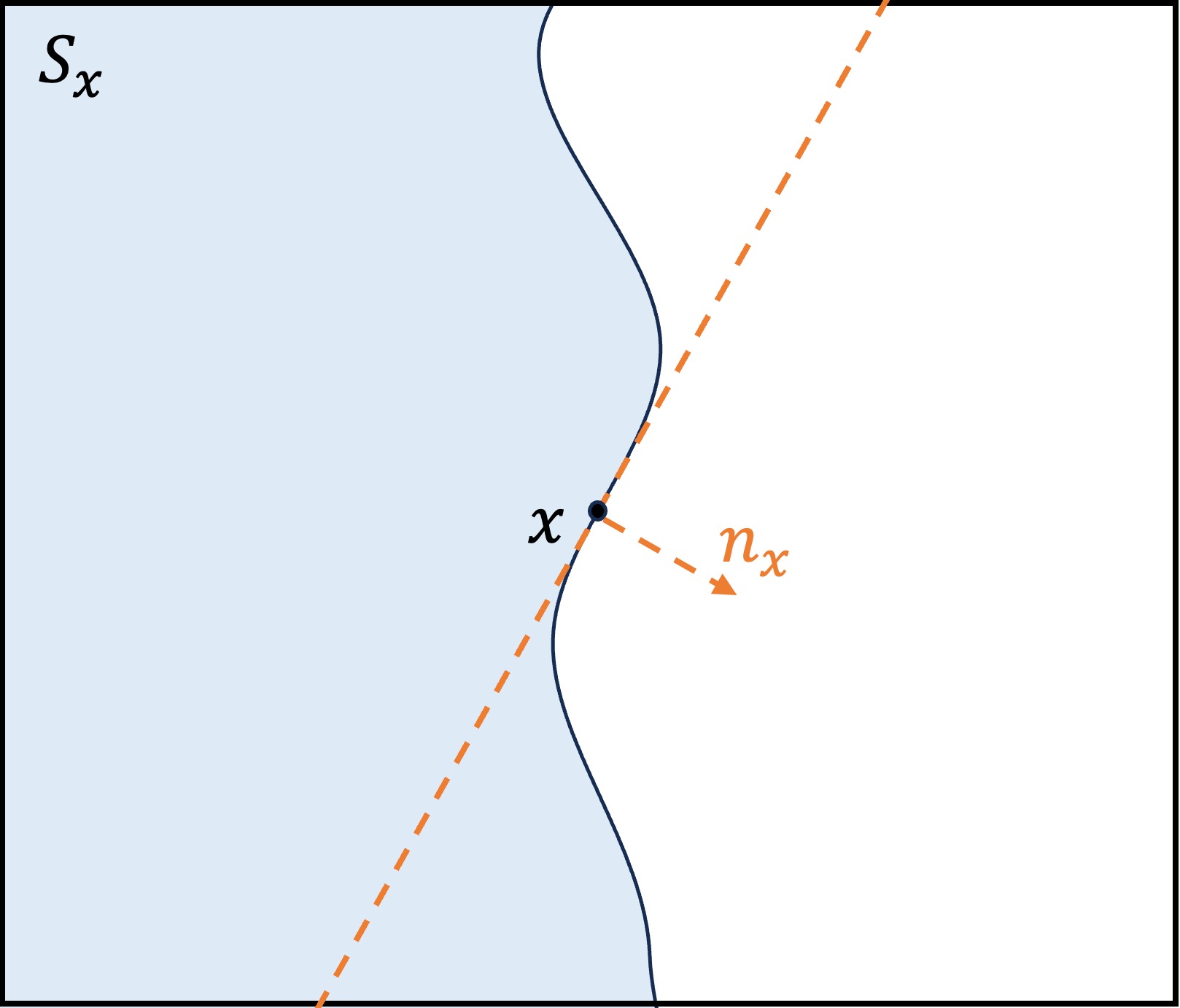}
        \caption{Outward unit normal direction $n_x$ to $\partial\calS_x$ at $x$}
        \label{fig:normaldirection}
    \end{subfigure}
    \qquad
    \begin{subfigure}{0.34\textwidth}
        \centering
        \includegraphics[width=1\textwidth]{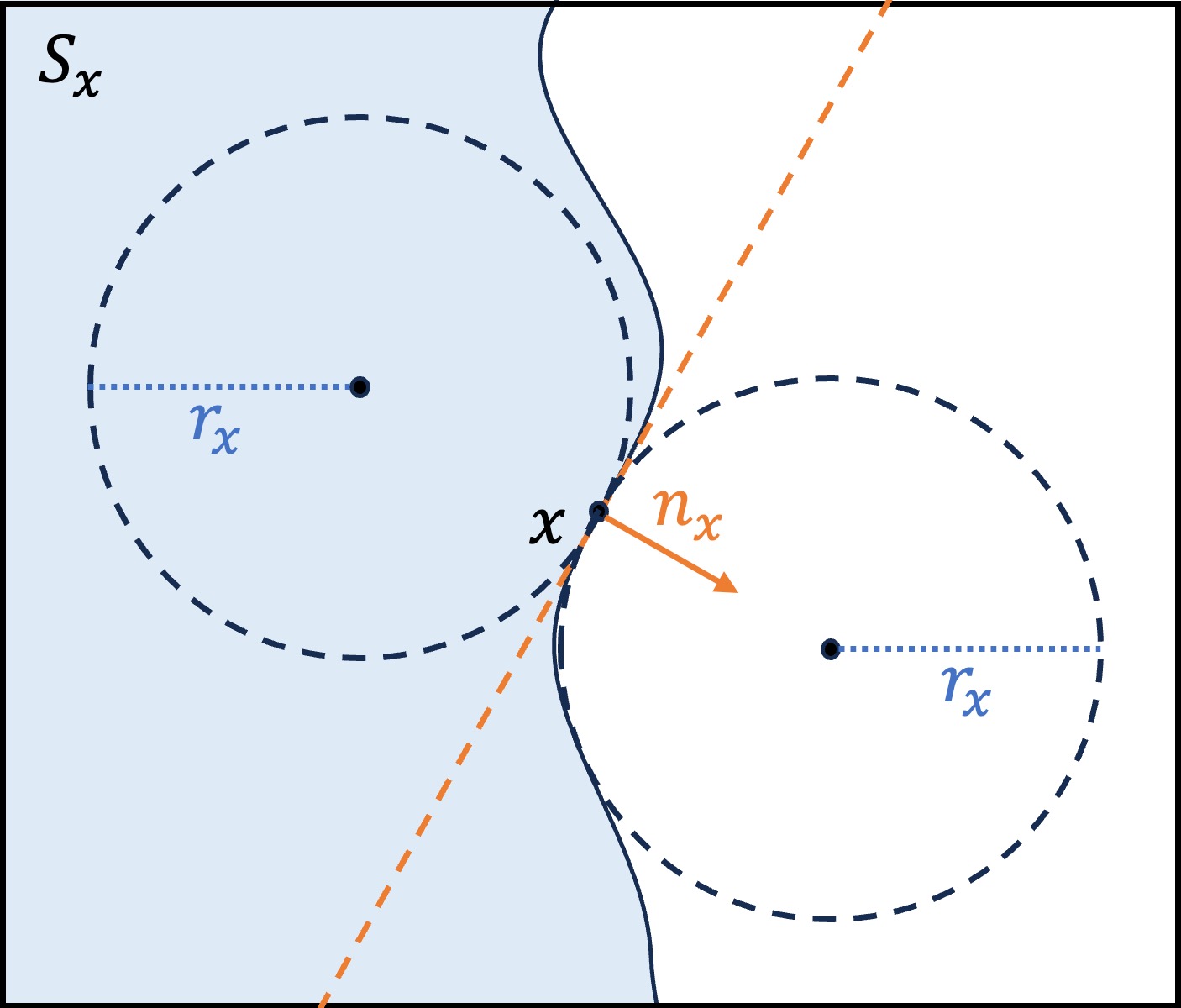}
        \caption{Two-sided tangent balls, and regularity radius $r_x$ at $x$}
        \label{fig:rollingball}
    \end{subfigure}
    \caption{Illustration of the outward unit normal direction $n_x$ and the regularity radius $r_x$.}
    \label{figs:levelset}
\end{figure}

It should be noted that, the regularity radius and curvature are local properties for $x$. Intuitively, the smaller the sublevel set $\mathcal{S}_x$ is, the smaller the largest incribing ball may be, and thus the smaller the regularity radius $r_x$ and the larger the curvature $\kappa_x$ for a boundary point $x$ can achieve. 
Furthermore, when the preference is convex, the sublevel set $\mathcal{S}_x$ is convex, and thus no matter how large $r$ is, $B(x+r n_x, r)\subseteq \overline{\X\setminus \calS_x}$, so the two-sided tangent-ball condition is equivalent to the ``one-sided'' tangent-ball condition  $r_x:= \sup_r\{r\ge 0:B(x-r n_x, r)\subseteq \calS_x\}$.   
In geometric terminology, radii defined through tangent balls are closely related to the notion of reach for sets with positive reach \cite{federer1969geometric}. In that language, $r_x$ plays the role of a local reach parameter.

\subsection{Relation with comparison-based optimization}
\label{subsec:order-equivalence}

The framework above takes the preference relation $\preccurlyeq$ (and the comparison oracle $\mathscr{C}_{\mathrm{cmp}}$) as the primitive description of the problem.
Much of the comparison-based optimization literature, however, starts from an objective function $f$ and then assumes that the algorithm can only access $f$ through comparisons \cite{zhang2024comparisons,tao2026gradient,karabag2021smooth,bergou2020stochastic}.
This subsection clarifies the relationship: whenever an objective $f$ realizes the preference relation, comparison-based optimization on $f$ is exactly the same as function-free optimization under~$\preccurlyeq$.
Moreover, the function-free viewpoint shows that some objective functions are actually invariant: strictly increasing reparameterizations of $f$ generate the same comparisons and hence the same optimization problem.

Let us say a function $f:\X\to\R$ realizes the preference relation $\preccurlyeq$ if for all $x,y\in\X$, $x\preccurlyeq y$ if and only if $f(x)\le f(y).$
If such a representative function $f$ exists, then the function-free optimization problem \eqref{eq:main-optimization-problem} is identical to the usual constrained minimization problem
\[
\min_{x\in\mathcal{C}} f(x),
\qquad\text{and hence}\qquad
\X^\star=\arg\min_{x\in\mathcal{C}} f(x).
\]  
Whenever $f$ realizes $\preccurlyeq$, comparisons can be encoded numerically by
\begin{equation}
\label{eq:Cf-def}
\mathscr{C}_f(x,y)\;:=\;\sign\!\bigl(f(y)-f(x)\bigr)\in\{-1,0,+1\}.
\end{equation}
In this case, $\mathscr{C}_f$ is consistent with Definition~\ref{def:comparison-oracle}: for all $x,y\in\X$, $\mathscr{C}_f(x,y)=-1$ iff $y\prec x$, $\mathscr{C}_f(x,y)=0$ iff $y\sim x$, and $\mathscr{C}_f(x,y)=+1$ iff $x\prec y$.
Thus, one may view the oracle $\mathscr{C}_{\mathrm{cmp}}$ as being instantiated by $\mathscr{C}_f$ when a representative function exists.

However, the same preference relation can have many different representative functions. We say two functions $f,g:\X\to\R$ are preference--equivalent on $\X$ if they realize the same preference relation: $f(x)\le f(y)$ if and only if $g(x)\le g(y)$ for all $x,y\in\X$. Equivalently, $\mathscr{C}_f(x,y)=\mathscr{C}_g(x,y)$ for all $x,y\in\X$. 
The next lemma shows that preference--equivalence is exactly composition with a strictly increasing scalar map.

\begin{lemma}\label{lem:phi-global}
If $f$ and $g$ are preference--equivalent on $\X$, then there exists a unique strictly increasing map
\(
\phi:\mathrm{range}(f)\to\mathrm{range}(g)
\)
such that $g=\phi\circ f$.
Conversely, for any strictly increasing $\phi:\R\to\R$, the functions $f$ and $\phi\circ f$ are preference--equivalent on~$\X$.
\end{lemma}

\begin{proof}
Define $\phi(t):=g(x)$ for any $x\in\X$ with $f(x)=t$.
This is well-defined because if $f(x)=f(y)$, then both $f(x)\le f(y)$ and $f(y)\le f(x)$ hold; by preference--equivalence,
$g(x)\le g(y)$ and $g(y)\le g(x)$, hence $g(x)=g(y)$.

If $t_1<t_2$ and $f(x_i)=t_i$, then $f(x_1)\le f(x_2)$ but not $f(x_2)\le f(x_1)$.
By preference--equivalence, $g(x_1)\le g(x_2)$ but not $g(x_2)\le g(x_1)$, hence $g(x_1)<g(x_2)$.
Therefore $\phi(t_1)<\phi(t_2)$ and $\phi$ is strictly increasing.
Uniqueness is immediate from $g=\phi\circ f$.
The converse holds because strictly increasing maps preserve order.
\end{proof}

The three functions in Figure~\ref{figs:levelsets} are preference--equivalent. Although looking very different, they induce the same preference relation and hence the same optimization problem from the perspective of comparisons.

\begin{proposition}[Invariance under preference--equivalent representative functions]
\label{prop:invariance-representatives}
If two objective functions $f$ and $g$ induce the same preference relation,
then any comparison-based algorithm will behave in exactly the same way on $f$ and on $g$:
it will receive the same oracle replies, generate the same sequence of queried points and iterates (up to its own randomness).
\end{proposition}
Therefore, for these algorithms that depends only on comparison information, no matter deterministic or randomized, any guarantee can only depend on the induced preference relation (and the sets $\calS_x,\calL_x$), not on the numerical scaling of a particular representative.  For example, the smoothness and convexity of the objective functions rely on the particular numerical scaling, but the smoothness and convexity of the preference relation are invariant under strictly increasing reparameterizations.

It should be noted that the comparison oracles considered in this paper (Definition~\ref{def:comparison-oracle} and \eqref{eq:Cf-def}) always return the correct comparison result. Exact comparison is common in comparison-based derivative-free optimization \cite{hooke1961direct,nelder1965simplex,golovin2020gradientless,bergou2020stochastic}. In addition, sometime people assume the comparison oracle may make errors with some probability \cite{jamieson2012query,saha2021dueling,saha2025dueling}.
In those cases, repeating comparing the same pair and taking a majority vote can generate the correct comparison result with high probability \cite{jamieson2012query}. If the probability of returning the right answer is at least $\tfrac{1}{2}+\nu$ for some $\nu>0$, then \cite{saha2021dueling} shows that there exists an adaptive sampling strategy so that with fixed given $\delta \in (0,1)$, it does not need prior knowledge of $\nu$ and can recover the correct comparison result with failure probability at most $\delta$ using at most $O(\nu^{-2}\log(1/\delta))$ repeated comparisons.

\medskip
\noindent
\textbf{Relation between regularity radius and smoothness.}
Although there might exist multiple functions realizing a preference relation, if one of them is smooth, then a relation between the regularity radius and the gradient norm can be established, as shown in the next lemma. 
\begin{lemma} 
\label{lem:regularityradius-hess-grad}
Suppose the preference relation $\preccurlyeq$ is realized by a differentiable function $f:\X\to\R$, and assume $f$ is $L$-smooth, i.e.,
\begin{equation}
\label{eq:L-smooth-grad-lip}
\|\nabla f(x)-\nabla f(y)\|\ \le\ L\|x-y\|
\qquad\text{for all }x,y\in\X.
\end{equation}
Fix any $x\in\X $ with $\nabla f(x)\neq 0$, then the normal direction $n_x$ at $x$ is given by $\hat{n}\;:=\;\frac{\nabla f(x)}{\|\nabla f(x)\|}$. The regularity radius $r_x$ and curvature $\kappa_x$ satisfy
\begin{equation}
\label{eq:rx-lowerbound-smooth}
r_x\ \ge\ \frac{\|\nabla f(x)\|}{L},
\qquad\text{and hence}\qquad
\kappa_x\ \le\ \frac{L}{\|\nabla f(x)\|} \ .
\end{equation}
\end{lemma}

\begin{proof}
If $L=0$, then $\nabla f$ is constant and the induced preference sublevel sets are halfspaces. In this case
$r_x=\infty$ and $\kappa_x=0$, so \eqref{eq:rx-lowerbound-smooth} holds trivially.
Assume below that $L>0$ and set $r:=\frac{\|\nabla f(x)\|}{L}.$ We will then show there exists two-sided tangent balls, $B(x-r \hat{n},r)$ and $B(x+r \hat{n},r)$, of radius $r$ at $x$ along direction $\hat{n}$, which implies $\hat{n}$ is the normal direction, $r_x\ge r$ and hence \eqref{eq:rx-lowerbound-smooth}.

Take any $y\in B(x-r \hat{n},r)$. Expanding $\|y-(x-r \hat{n})\|^2\le r^2$ gives $\|y-x\|^2+2r\langle \hat{n},y-x\rangle\le 0$, \text{so} $\langle \hat{n},y-x\rangle\le -\frac{\|y-x\|^2}{2r}$. Multiplying by $\|\nabla f(x)\|$ yields
\[
\langle \nabla f(x),y-x\rangle
=\|\nabla f(x)\|\langle \hat{n},y-x\rangle
\le -\frac{\|\nabla f(x)\|}{2r}\|y-x\|^2
=-\frac{L}{2}\|y-x\|^2.
\]
Note that smoothness \eqref{eq:L-smooth-grad-lip} implies that $f(y) \le f(x)+\langle \nabla f(x),y-x\rangle+\frac{L}{2}\|y-x\|^2,$ so  $f(y)\le f(x)$.
Therefore,
$B(x-r \hat{n},r)\subseteq \calS_x$.

Similarly, for any $y\in B(x+r \hat{n},r)$, expanding $\|y-(x+r \hat{n})\|^2\le r^2$ gives $\|y-x\|^2-2r\langle \hat{n},y-x\rangle\le 0$, so $\langle \hat{n},y-x\rangle\ge \frac{\|y-x\|^2}{2r}$.
Therefore
\[
\langle \nabla f(x),y-x\rangle
=\|\nabla f(x)\|\langle \hat{n},y-x\rangle
\ge \frac{\|\nabla f(x)\|}{2r}\|y-x\|^2
=\frac{L}{2}\|y-x\|^2.
\]
Similarly, the smoothness condition \eqref{eq:L-smooth-grad-lip} implies that $f(y) \ge f(x)+\langle \nabla f(x),y-x\rangle-\frac{L}{2}\|y-x\|^2,$ so $f(y)\ge f(x)$. Therefore, $B(x+r \hat{n},r)\subseteq \overline{\X\setminus \calS_x}$.

By Definitions~\ref{def:regularity} and \ref{def:curvature}, $\hat{n}$ is the normal direction $n_x$ at $x$, and $r_x\ge r$ and $\kappa_x\le 1/r$.
\end{proof}

\subsection{Optimality and stationarity criteria}\label{subsec:optimality-stationarity}
\paragraph{Level-set optimality gap.}
In function-based optimization, the quality of a solution is often measured by function suboptimality (e.g., $f(x)-f^\star$).
In function-free optimization, we do not have access to numerical objective values, so we instead use comparison-invariant geometric measures.
Two examples are the distance to optimal solutions, $\dist(x,\X^\star)$, and the level-set optimality gap, $\DeltaLS(x)$, defined as
\[ 
\DeltaLS(x):=\dist\bigl(\X^\star,\calL_x\bigr).
\]
Since $x\in \calL_x$, we always have $\DeltaLS(x)\le \dist(x,\X^\star)$, so $\DeltaLS(x)$ is never larger than the point-to-optimal-set distance and can therefore be easier to satisfy.
Moreover, because points in the same level set $\calL_x$ are regarded as indifferent under the preference relation, if $\DeltaLS(x)$ is small, then there exists a point that is indifferent to $x$ and close to the optimal set.
In this sense, $x$ is as good as a point near the optimal set from the perspective of the preference relation.
Therefore, $\DeltaLS(x)$ is more intrinsic to the preference relation $\preccurlyeq$ than $\dist(x,\X^\star)$.
For this reason, most of our optimization results are stated in terms of $\DeltaLS(\cdot)$.
 
\paragraph{Regularity radius.}
Although not always reliable, the regularity radius $r_x$ can also serve as a measure of near-stationarity in some cases.
The regularity radius $r_x$ is also intrinsic to the preference relation.
At a high level, small $r_x$ indicates that the local geometry of the sublevel set near $x$ is sharp, so small perturbations may  lead to noticeable changes in the local ordering.
This is consistent with the intuition that, in a near-optimal or near-stationary region, a small perturbation of the underlying preference may flip the local ordering of two points.

If a smooth function realizes the preference relation, then small regularity radius implies small gradient norm, which is widely regarded as a measure of near-stationarity.
Indeed, Lemma~\ref{lem:regularityradius-hess-grad} shows that if $\preccurlyeq$ is realized by an $L$-smooth differentiable function $f$, then
$
r_x \ge \frac{\|\nabla f(x)\|}{L},
$
and hence $\|\nabla f(x)\|\le L r_x$.
Moreover, small regularity radius is in general a stronger certificate than small gradient norm, since the converse need not hold.
For $f(t)=t^3$ on $\R$, the point $t=0$ is stationary, but the induced preference sublevel set $\calS_0$ is $(-\infty,0]$, so $r_0=\infty$.
Furthermore, it will be shown in Section~\ref{subsec:regularity-vs-distance} that a local growth condition widely holds for regularity radius. For example, it holds when the preference relation is realized by several commonly used classes of functions, including convex quadratic functions, strongly convex functions, and smooth functions satisfying the PL inequality and quadratic growth. 
If such a local growth condition holds, then small regularity radius directly implies small level-set optimality gap $\DeltaLS(\cdot)$.

Throughout the paper we use $\DeltaLS(\cdot)$ as the primary measure of near-optimality, while small $r_x$ appears as a secondary certificate.  This viewpoint is also reflected in our algorithmic guarantees: the method will either drive the level-set optimality gap to be small, or return a point with small regularity radius. Achieving the former usually requires higher comparison complexity, while achieving the latter typically only requires setting certain parameters to be sufficiently small.

\section{Normal direction estimation from comparisons}
\label{sec:normal-estimation}

This section studies how to estimate the normal direction $n_x$ at a point $x$ using only pairwise comparisons. 
Throughout this section we assume the preference relation is plateau-free (Definition~\ref{def:well-posed}) and $r_x$-regular at $x$ (Definition~\ref{def:regularity}) with regularity radius $r_x>0$,
so the normal direction $n_x$ exists and is unique.
We first propose a normal direction estimation method based on comparing $x$ and $x+h u$ with user-chosen $h>0$ and $u\in\mathbb{S}^{d-1}$, for $x$ at which the preference relation is $r_x$-regular with regularity radius $r_x>0$.
In \ref{subsec:convex-lineseach}, assuming in addition that the sublevel set $\calS_x$ is convex, we propose a certified comparison scheme using line search to update $h$, yielding a parameter-free normal direction estimator.
Minimax lower bounds are given in Section~\ref{sec:normal-lower-bound}.

\subsection{Halfspace certificates of normal direction via comparisons}

Before presenting the algorithms, we first show how to extract halfspace certificates of normal direction $n_x$ from comparisons, which is the key building block for our normal direction estimation methods.
 
\begin{lemma}
\label{lem:curv-margin-OSNE}
Fix $x\in\mathcal{X}\setminus\X^\star$ and assume the preference relation is plateau-free and $r_x$-regular at $x$ with regularity radius $r_x > 0$.
Fix any $h>0$ and set $\tau:=\frac{h}{2r_x}$.
Then for every $u\in\mathbb{S}^{d-1}$,
\begin{enumerate}[label=(\roman*),leftmargin=2.2em]
\item if $x \preccurlyeq x+hu$, then $\langle n_x,u\rangle\ge -\tau$;
\item if $x+hu\preccurlyeq x$, then $\langle n_x,u\rangle\le \tau$.
\end{enumerate}
\end{lemma}

\begin{proof} 
We prove both claims by contradiction.

\smallskip
\noindent\textit{(i)}
Assume $\langle n_x,u\rangle< -\tau=-\frac{h}{2r_x}$.
Then
\[
\|(x+hu)-(x-r_x n_x)\|^2
=
\|hu+r_x n_x\|^2
=
h^2+r_x^2+2hr_x\langle u,n_x\rangle
<
r_x^2,
\]
so $x+hu\in \operatorname{int}B(x-r_x n_x,r_x)\subseteq \operatorname{int}(\calS_x)$.
Since $x\notin\X^\star$, plateau-freeness implies that interior points of $\calS_x$ cannot tie with $x$; hence $x+hu\prec x$.
In particular, $x\preccurlyeq x+hu$ is false, proving (i).

\smallskip
\noindent\textit{(ii)}
Assume $\langle n_x,u\rangle>\tau=\frac{h}{2r_x}$.
Then
\[
\|(x+hu)-(x+r_x n_x)\|^2
=
\|hu-r_x n_x\|^2
=
h^2+r_x^2-2hr_x\langle u,n_x\rangle
<
r_x^2,
\]
so $x+hu\in \operatorname{int}B(x+r_x n_x,r_x)\subseteq \operatorname{int}(\calS_x^c)$.
Therefore $x+hu\not\preccurlyeq x$, proving (ii).
\end{proof}

Lemma \ref{lem:curv-margin-OSNE} says that, the comparison between $x$ and $x+hu$ provides a halfspace certificate of the normal direction $n_x$ with margin $\tau=\frac{h}{2r_x}$. The left two plots of Figure \ref{figs:nonconvex-margin} shows an illustration.
Intuitively, if $x+hu \in S_x$, then the angle between $u$ and $n_x$ cannot be much smaller than $90^\circ$, and if $x+hu\notin S_x$, then the angle between $u$ and $n_x$ cannot be much larger than $90^\circ$.
The flatter the sublevel set is (the smaller $h$ is and the larger $r_x$ is), the more accurate the certificate is. 

Furthermore, if assuming in addition that $\calS_x$ is convex, then the halfspace certificate can be strengthened with no margin.

\begin{lemma}
\label{lem:convex-supporting-halfspace}
Let the preference relation be $r_x$-regular at $x$ with regularity radius $r_x>0$, and assume in addition that $\calS_x$ is convex.
Then $\calS_x$ is contained in the supporting halfspace: $\calS_x
\subseteq
\{z\in\mathcal{X}:\langle n_x,z-x\rangle\le 0\}$.
Furthermore, for any $h>0$ and $u\in\mathbb{S}^{d-1}$, if $
x+hu\preccurlyeq x$, then $\langle n_x,u\rangle\le 0.$
\end{lemma}

\begin{proof}
Since $\calS_x$ is convex, its tangent halfspace at $x$ is a supporting halfspace. 
Now fix $h>0$ and $u\in\mathbb{S}^{d-1}$.
If $x+hu\preccurlyeq x$, then $x+hu\in\calS_x$, and $\mathcal{S}_x \subseteq \{z\in\mathcal{X}:\langle n_x,z-x\rangle\le 0\}$ implies
\[
\langle n_x,hu\rangle=\langle n_x,(x+hu)-x\rangle\le 0,
\]
which is  the desired conclusion.
\end{proof}
 
In the convex setting, Lemma~\ref{lem:convex-supporting-halfspace} provides an exact one-sided certificate that we will exploit in \ref{subsec:convex-lineseach} by adaptively shrinking $h$ until comparisons become fully reliable. The right plot of Figure \ref{figs:nonconvex-margin} shows an illustration.

\begin{figure}[t]
    \begin{subfigure}{0.32\textwidth}
        \centering
        \includegraphics[width=1\textwidth]{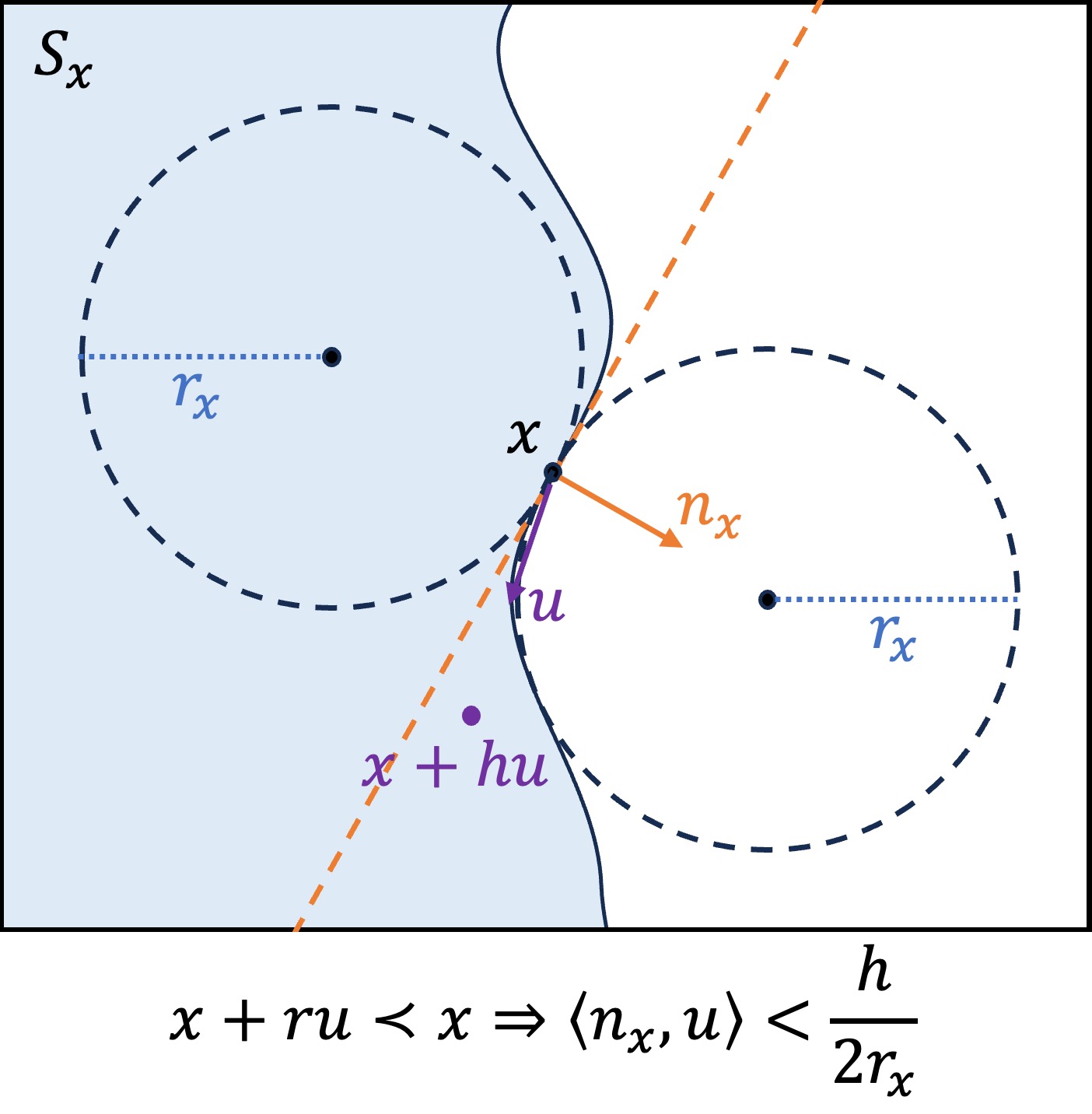} 
    \end{subfigure}
    \hfill
    \begin{subfigure}{0.32\textwidth}
        \centering
        \includegraphics[width=1\textwidth]{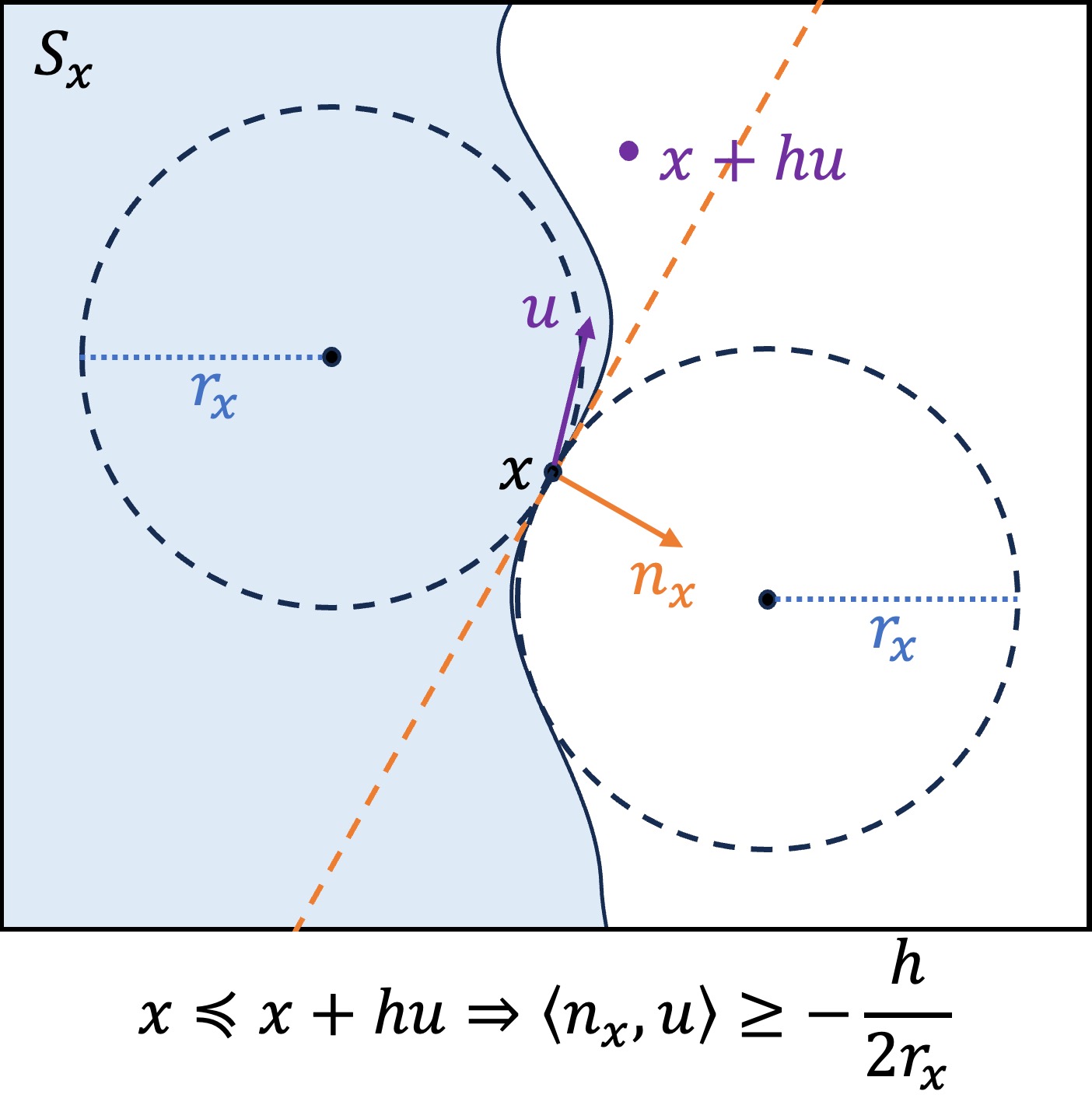} 
    \end{subfigure}
    \hfill
        \begin{subfigure}{0.32\textwidth}
        \centering
        \includegraphics[width=1\textwidth]{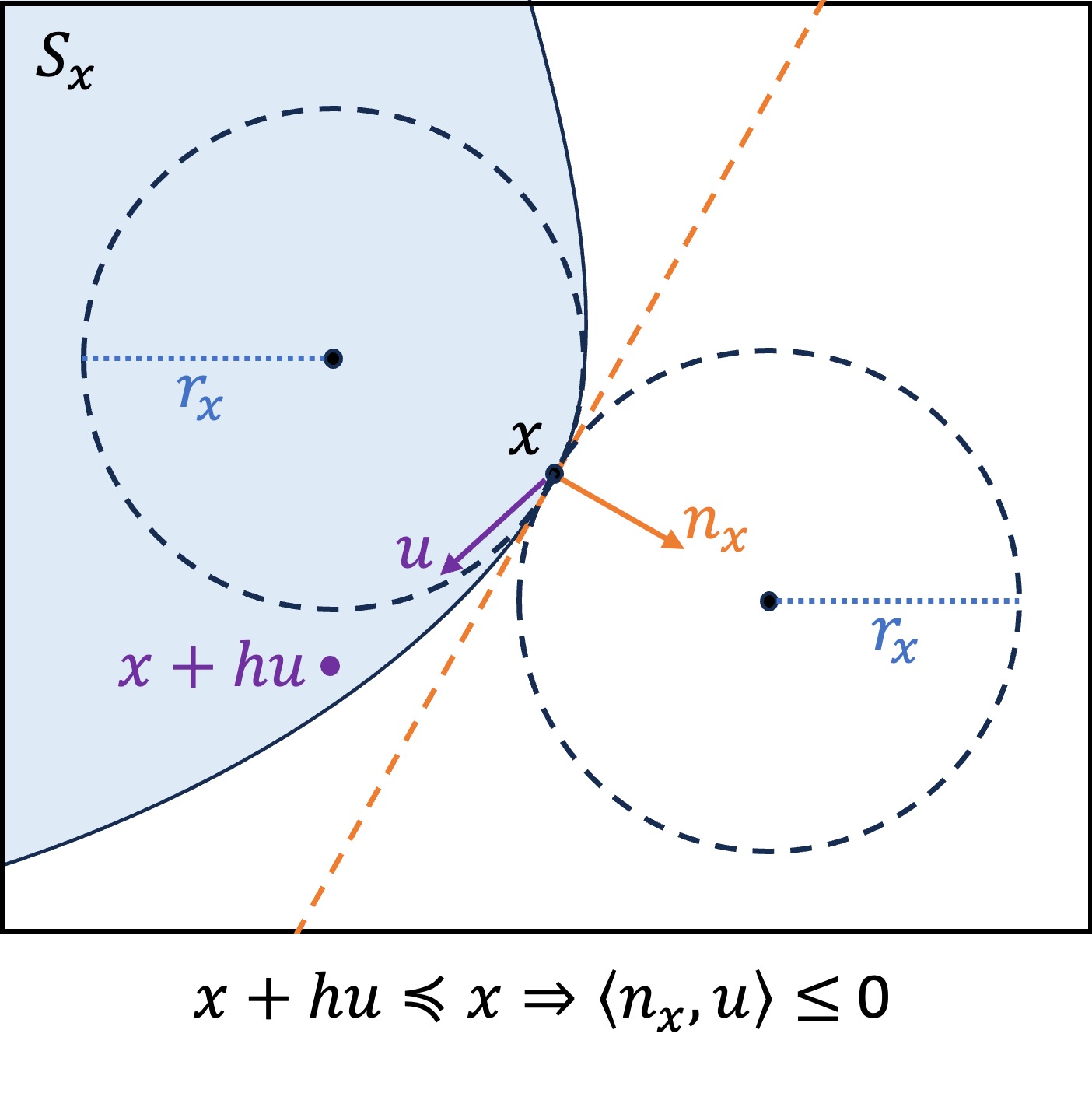} 
    \end{subfigure}
    \caption{Illustration examples of Lemma~\ref{lem:curv-margin-OSNE} for nonconvex sublevel sets $\calS_x$ (left two plots) and Lemma~\ref{lem:convex-supporting-halfspace} for convex sublevel sets $\calS_x$ (right one plot).}
    \label{figs:nonconvex-margin}
\end{figure}

\subsection{Normal direction estimation with fixed comparison radius}\label{subsec:normal-estimation}

Our normal direction estimation method has a two-layer structure. The inner layer (Algorithm~\ref{alg:planar-bisect}) performs a bisection search on the unit circle in a given two-dimensional plane, while the outer layer (Algorithm~\ref{alg:OSNE}) uses this planar bisection search repeatedly to build an orthonormal family of $d-1$ approximately tangential directions. Rather than searching for $n_x$ directly, we construct directions $y_1,\dots,y_{d-1}$ satisfying $|\langle n_x,y_i\rangle|$ small. The remaining one-dimensional orthogonal complement then approximates the normal direction line $\{\pm n_x\}$, and one final comparison fixes the outward orientation. In this subsection, we work with the scheme of fixed comparison radius $h$, while the adaptive comparison radius case will be presented in Section~\ref{subsec:convex-lineseach}.

Algorithm~\ref{alg:planar-bisect} solves the planar subproblem in $\mathrm{span}\{a,b\}$ and returns an approximately tangential direction. Algorithm~\ref{alg:OSNE} calls this planar routine $d-1$ times in carefully chosen planes so that the resulting directions remain orthonormal. Theorem~\ref{thm:OSNE} then gives the overall estimation error and comparison complexity in terms of the comparison radius $h$ and the planar depth $T$.

\begin{algorithm}[htbp]
\caption{Planar bisection for an approximately tangential direction}
\label{alg:planar-bisect}
\begin{algorithmic}[1]
\State \textbf{Input:} base point $x$, comparison radius $h>0$, orthonormal $a,b\in\mathbb{S}^{d-1}$, steps $T\in\mathbb{N}_{++}$.
\State Sample $\theta\sim\mathrm{Unif}[0,2\pi)$ and set
$\bar a\gets \cos\theta\,a+\sin\theta\,b$,
$\bar b\gets -\sin\theta\,a+\cos\theta\,b$. 

\State If using fixed comparison radius, compare $x$ with $x + h\bar a$ and $x - h\bar a$. \label{line:pb-compare-a}
\State If using adaptive comparison radius, use Algorithm~\ref{alg:linesearch} to update $h$, and compare $x$ with $x + h\bar a$ and $x - h\bar a$. \label{line:pb-compare-a-line-search}
\State
\textbf{If} {$(x+h\bar a\preccurlyeq x\ \textbf{and}\ x-h\bar a\preccurlyeq x)$ \textbf{or} $(x\preccurlyeq x+h\bar a\ \textbf{and}\ x\preccurlyeq x-h\bar a)$} \textbf{return} $\bar a$.  \label{line:pb-return-a}

\State \textbf{If} $x+h\bar a\preccurlyeq x$ \textbf{set} $(u_+,u_-)\gets(-\bar a,\bar a)$; \textbf{else set} $(u_+,u_-)\gets(\bar a,-\bar a)$.\label{line:setinitialu}
\State If using fixed comparison radius, compare $x$ with $x+h\bar b$. \label{line:pb-compare-b}
\State If using adaptive comparison radius, use Algorithm~\ref{alg:linesearch} to update $h$ and compare $x$ with $x+h\bar b$. \label{line:pb-compare-b-line-search}
\State \textbf{If} $x+h\bar b\preccurlyeq x$ \textbf{update} $u_-\gets \bar b$; \textbf{else update} $u_+\gets \bar b$.\label{line:updateu} 
\For{$t=1$ \textbf{to} $T$}
    \State $u_M\gets \dfrac{u_++u_-}{\|u_++u_-\|}$.
    \State If using fixed comparison radius, compare $x$ with $x+h u_M$.\label{line:pb-compare-um}
    \State If using adaptive comparison radius, use Algorithm~\ref{alg:linesearch} to update $h$, and compare $x$ with $x+h u_M$.\label{line:pb-compare-um-line-search}
    \State \textbf{If} $x+h u_M\preccurlyeq x$ \textbf{update} $u_-\gets u_M$; \textbf{else update} $u_+\gets u_M$. 
\EndFor
\State \Return $u_M$.
\end{algorithmic}
\end{algorithm}

Algorithm \ref{alg:planar-bisect} is essentially a bisection search on the unit circle in $\mathrm{span}\{a,b\}$. Lines \ref{line:setinitialu} and \ref{line:updateu} initialize $u_+$ and $u_-$ to be two orthogonal directions that are separated by the normal direction $n_x$ (i.e., $x+hu_-\preccurlyeq x\preccurlyeq x+hu_+$), and the loop performs a bisection search by comparing $x$ with $x+h u_M$, where $u_M$ is the normalized midpoint of $u_+$ and $u_-$. The output is the last midpoint $u_M$. Lemma \ref{lem:planar-bisect} below shows that $u_M$ is an approximately tangential direction of $n_x$.

\begin{lemma}
\label{lem:planar-bisect}
Let the preference relation be plateau-free and $r_x$-regular at $x$ with regularity radius $r_x>0$,  and $d\ge 2$.
Run Algorithm~\ref{alg:planar-bisect} with fixed comparison radius  $h > 0$, and let the algorithm return $y$.
Then the output satisfies
\begin{equation}
\label{eq:planar-bisect-bound}
|\langle n_x,y\rangle|\ \le\ \frac{h}{2r_x}\;+\;\frac{\pi}{2^{T+2}}.
\end{equation}
\end{lemma}

\begin{proof}
Set $\tau:=h/(2r_x)$. Recall that Lemma~\ref{lem:curv-margin-OSNE} shows for any $u\in\mathbb{S}^{d-1}$,
 if $
x+hu\preccurlyeq x$ then $\langle n_x,u\rangle\le \tau$, and if $
x\preccurlyeq x+hu$ then $\langle n_x,u\rangle\ge -\tau$.

Algorithm~\ref{alg:planar-bisect} works in the plane $\mathrm{span}\{a,b\}$ (the random rotation only changes the in-plane basis).
If the algorithm returns $\bar a$, then either
(i) $x\pm h\bar a\preccurlyeq x$, or (ii) $x\preccurlyeq x\pm h\bar a$.
In case (i), applying Lemma~\ref{lem:curv-margin-OSNE} to $u=\bar a$ and $u=-\bar a$ yields
$\langle n_x,\bar a\rangle\le \tau$ and $-\langle n_x,\bar a\rangle\le \tau$, hence $|\langle n_x,\bar a\rangle|\le \tau$.
Case (ii) is analogous.
Thus the desired bound holds.
 
Otherwise, the algorithm constructs a quadrant bracket $(u_-,u_+)$ satisfying
$x+hu_-\preccurlyeq x$, $x\preccurlyeq x+hu_+$, and  $\angle(u_-,u_+)\le \frac{\pi}{2}$.
Using Lemma~\ref{lem:curv-margin-OSNE}, $\langle n_x,u_-\rangle\le \tau$, $\langle n_x,u_+\rangle\ge -\tau$.
Let $I$ denote the shorter circular arc on $\mathbb{S}^{d-1}\cap \mathrm{span}\{a,b\}$ connecting $u_-$ to $u_+$.
Since map $u\mapsto \langle n_x,u\rangle$ is continuous on $I$,  we have $u^\star\in I$ and $|\langle n_x,u^\star\rangle|\le \tau$.
At each bisection step, the algorithm sets $u_M=\frac{u_-+u_+}{\|u_-+u_+\|}$, which is the geodesic midpoint of $I$, so the bracket arc length halves.
Moreover, updating $u_-$ when $x+hu_M\preccurlyeq x$ and updating $u_+$ otherwise preserves the two implications $x+hu_-\preccurlyeq x$ and $x\preccurlyeq x+hu_+$, and hence preserves the existence of some $u^\star$ in the current bracket arc with $|\langle n_x,u^\star\rangle|\le \tau$.

After $T$ bisections, the bracket arc length is at most $\frac{\pi}{2^{T+1}}$,
so the returned midpoint $y$ satisfies $\angle(y,u^\star)\le \frac{\pi}{2^{T+2}}$ for some $u^\star$ in the final arc with
$|\langle n_x,u^\star\rangle|\le \tau$.
Finally, since $\|y\|=\|u^\star\|=1$ and $\|y-u^\star\|\le \angle(y,u^\star)$, we obtain 
\[
|\langle n_x,y\rangle|
\le
|\langle n_x,u^\star\rangle| + |\langle n_x,y-u^\star\rangle|
\le
\tau + \|y-u^\star\|
\le
\tau + \angle(y,u^\star)
\le
\tau + \frac{\pi}{2^{T+2}},
\]
which is \eqref{eq:planar-bisect-bound}.
\end{proof}

We now lift the planar bisection method to $d$ dimensions by repeatedly computing orthogonal directions $y_i$ of $n_x$ in successive two-dimensional planes. Once such a direction is computed, it is then rotated in a $2$-dimensional subspace to ensure orthonormality, so that the final normal direction estimate can be read off without solving any linear system.

\begin{algorithm}[htbp]
\caption{Normal direction estimation via orthogonal-subspace construction}
\label{alg:OSNE}
\begin{algorithmic}[1]
\State Input: point $x$, dimension $d$, comparison radius $h>0$, planar bisection steps $T\in\mathbb{N}_{++}$.
\State Sample a Haar-uniform orthonormal frame $(q_1,\dots,q_d)$ in $\X$.
\State Set $v_1\gets q_1$.
\For{$i=1$ \textbf{to} $d-1$}
    \State Run Algorithm~\ref{alg:planar-bisect} at $x$ with $(a,b)=(v_i,q_{i+1})$ and $T$ steps, and obtain $y_i$.  
    \State Set $v_{i+1}\gets -\langle y_i, q_{i+1}\rangle v_i + \langle y_i, v_i\rangle q_{i+1}$.
\EndFor
\State If using fixed comparison radius, compare $x$ with $x+h v_d$.\label{line:compare-v-d}
\State If using adaptive comparison radius, use Algorithm~\ref{alg:linesearch} to update $h$, and compare $x$ with $x+h v_d$.\label{line:compare-v-d-line-search}
\State \textbf{If} $x+h v_d\prec x$ \textbf{return} $\widehat n:=-v_d$ \textbf{else} \textbf{return} $\widehat n:=v_d$.
\end{algorithmic}
\end{algorithm}

\begin{lemma}
\label{lem:OSNE-invariants}
In Algorithm~\ref{alg:OSNE}, for each $i\in\{1,\dots,d-1\}$ the pair $(v_i,q_{i+1})$ passed to Algorithm~\ref{alg:planar-bisect} is orthonormal, and  $y_1,\dots,y_{i-1},v_i$ are orthonormal.
\end{lemma}

\begin{proof} 
We prove by induction that for each $i\in\{1,\dots,d-1\}$ the set $\{y_1,\dots,y_{i-1},v_i\}$ is orthonormal and satisfies
\[
\mathrm{span}\{y_1,\dots,y_{i-1},v_i\}\subseteq \mathrm{span}\{q_1,\dots,q_i\}.
\]
For $i=1$, this holds since $v_1=q_1$.
Let $y_i$ be the unit vector returned by that planar call and define $c_i:=\langle y_i,v_i\rangle$ and $s_i:=\langle y_i,q_{i+1}\rangle$ so that $y_i=c_i v_i+s_i q_{i+1}$ with $c_i^2+s_i^2=1$.
The update sets $v_{i+1}:=-s_i v_i+c_i q_{i+1}$, which is unit and satisfies $\langle v_{i+1},y_i\rangle=0$.
Moreover, $v_{i+1}$ is orthogonal to all previous $y_j$ because it lies in $\mathrm{span}\{v_i,q_{i+1}\}$ and both $v_i$ and $q_{i+1}$ are orthogonal to $\{y_1,\dots,y_{i-1}\}$.
Thus $\{y_1,\dots,y_i,v_{i+1}\}$ is orthonormal.
Finally, since both $v_i$ and $q_{i+1}$ lie in $\mathrm{span}\{q_1,\dots,q_{i+1}\}$, so does $v_{i+1}$, completing the induction. Taking $i=d-1$ shows that $\{y_1,\dots,y_{d-1},v_d\}$ is an orthonormal basis of $\X$.

Since $q_{i+1}$ is orthogonal to $\mathrm{span}\{q_1,\dots,q_i\}$, it is also orthogonal to $v_i$ and all previous $y_j$.
Because $\|q_{i+1}\|=\|v_i\|=1$, the pair $(v_i,q_{i+1})$ is orthonormal.
\end{proof}

Now we are ready for showing the accuracy guarantee and the number of pairwise comparisons used by Algorithm~\ref{alg:OSNE}.

\begin{theorem}[Normal direction estimation] 
\label{thm:OSNE}
Assume the preference relation is plateau-free and $r_x$-regular at $x$ with regularity radius $r_x>0$, and $d\ge 2$. In the regime of fixed comparison radius $h > 0$,  the output $\widehat n$ of Algorithm~\ref{alg:OSNE} satisfies
\begin{equation}
\label{eq:OSNE-error-bound}
\|n_x-\widehat n\|
\ \le\
2\sqrt{d-1}\left(\frac{h}{r_x}+\frac{\pi}{2^{T+1}}\right).
\end{equation}
and Algorithm~\ref{alg:OSNE} uses at most $(d-1)(T+3)+1$ pairwise comparisons.
\end{theorem}

\begin{proof} 
If the right-hand side of \eqref{eq:OSNE-error-bound} is at least $2$, then the claim is trivial because both $n_x$ and $\widehat n$ are unit vectors. Hence later we only consider the case $h<r_x$.

Let $\tau:= h/(2 r_x)$ and $\gamma:=\tau+\pi/2^{T+2}$.
Lemma~\ref{lem:planar-bisect} implies that for each $i\in\{1,\dots,d-1\}$, $|\langle n_x,y_i\rangle|\le \gamma$ almost surely. 
Using Lemma~\ref{lem:OSNE-invariants}, we can write
$n_x=\alpha_0 v_d+\sum_{i=1}^{d-1}\beta_i y_i$
where $\alpha_0:=\langle n_x,v_d\rangle$ and $\beta_i:=\langle n_x,y_i\rangle$.
Then $\sum_{i=1}^{d-1}\beta_i^2\le (d-1)\gamma^2$.

If $\sqrt{d-1}\,\gamma\ge \tfrac12$, then the right-hand side of \eqref{eq:OSNE-error-bound} is at least $2$, while $\|n_x-\widehat n\|\le 2$ for any unit vectors; hence \eqref{eq:OSNE-error-bound} holds trivially.

Assume now that $\sqrt{d-1}\,\gamma<\tfrac12$.
Then $\sum_{i=1}^{d-1}\beta_i^2<\tfrac14$ and therefore
$|\alpha_0|=\sqrt{1-\sum_{i=1}^{d-1}\beta_i^2}>\sqrt{3}/2$.
In particular $|\alpha_0|>\tau$ since $\tau\le \gamma<\tfrac12$.
We claim that the final orientation step returns $\widehat n=\mathrm{sign}(\alpha_0)\,v_d$.
Indeed, if $\alpha_0>\tau$ then $x+h v_d\preccurlyeq x$ cannot occur (by Lemma~\ref{lem:curv-margin-OSNE}(ii)), so the algorithm cannot output $-v_d$.
Similarly, if $\alpha_0<-\tau$ then $x\preccurlyeq x+h v_d$ cannot occur (by Lemma~\ref{lem:curv-margin-OSNE}(i)), so the algorithm cannot output $v_d$.
Thus $\widehat n=\mathrm{sign}(\alpha_0)\,v_d$ and $\langle n_x,\widehat n\rangle=|\alpha_0|$.
Finally,
\[
\|n_x-\widehat n\|^2
=
(1-|\alpha_0|)^2+\sum_{i=1}^{d-1}\beta_i^2
\le
2\sum_{i=1}^{d-1}\beta_i^2
\le
2(d-1)\gamma^2,
\]
so $\|n_x-\widehat n\|\le 2\sqrt{d-1}\,\gamma\le 4\sqrt{d-1}\,\gamma$, proving \eqref{eq:OSNE-error-bound}.

Finally, each call to Algorithm~\ref{alg:planar-bisect} uses at most $T+3$ comparisons, and the final orientation step uses one more. Overall, at most $(d-1)(T+3)+1$ comparisons are used.
\end{proof}
 
The bound \eqref{eq:OSNE-error-bound} separates the two sources of error:
$2\sqrt{d-1}\frac{h}{r_x}$ is the local-geometry-induced bias from finite-radius comparisons (Lemma~\ref{lem:curv-margin-OSNE}),
and $2\sqrt{d-1} \frac{\pi}{2^{T+1}}$ is the algorithmic approximation error from terminating planar bisection after $T$ steps. Intuitively, the smaller $r_x$ is, the faster the local geometry can change, leading to a larger bias term in the error.

\begin{corollary}
\label{cor:OSNE-epsilon}
Under the assumptions of Theorem~\ref{thm:OSNE}, for any $\epsilon\in(0,2]$, let
$h \in (0, \frac{\epsilon r_x}{8\sqrt{d-1}}]$, and $T = \left\lceil\log_2\!\Big(\frac{8\sqrt{d-1}\pi}{\epsilon}\Big)\right\rceil - 2$.
Then $\|n_x-\widehat n\| \le \epsilon$, and Algorithm~\ref{alg:OSNE} uses at most
\begin{equation}\label{eq:OSNE-comparisons-bound}
N_\epsilon \;=\;
(d-1)\left\lceil\log_2\left(\frac{51\sqrt{d-1}}{\epsilon}\right)\right\rceil + 1
\end{equation}
pairwise comparisons.
\end{corollary}
\begin{proof}
Substituting $h \le \frac{\epsilon r_x}{8\sqrt{d-1}}$ and $T = \lceil\log_2(\frac{8\sqrt{d-1}\pi}{\epsilon})\rceil - 2$ into \eqref{eq:OSNE-error-bound} yields $\|n_x-\widehat n\| \le \epsilon$. Note that for $\epsilon \in (0,2]$, $h > 0$ and $T \in \mathbb{N}_{++}$ so they are valid parameters. The number of comparisons is at most
$$
(d-1)(T+3)+1 = (d-1)\left(\left\lceil\log_2\left(\frac{8\sqrt{d-1}\pi}{\epsilon}\right)\right\rceil +1 \right)+1 \le (d-1)\left\lceil\log_2\left(\frac{51\sqrt{d-1}}{\epsilon}\right)\right\rceil   + 1 ,
$$
where the last inequality uses $\log_2(8\pi)  +1 =  \log_2(16\pi) < \log_2(51)$. This finishes the proof.
\end{proof}

Corollary~\ref{cor:OSNE-epsilon} shows that, once the fixed comparison radius satisfies
$h\lesssim \epsilon r_x/\sqrt{d}$, Algorithm~\ref{alg:OSNE} estimates the normal direction to accuracy $\epsilon$ using $O(d\log(d/\epsilon))$ comparisons. In Section~\ref{sec:normal-lower-bound}, we will prove a minimax lower bound of order $\Omega(d\log(1/\epsilon))$ for normal direction estimation. Hence the above complexity is nearly optimal. Making the range of proper comparison radii as large as possible is also preferable, as it would be easier to choose a suitable $h$ in practice, and comparisons between very close points are less robust practically.

It should also be noted that the case $d=1$ should be separately analyzed, because Algorithm~\ref{alg:OSNE} reduces to a single comparison between $x$ and $x+h v_1$, where $v_1\in\mathbb S^0$, and Algorithm~\ref{alg:planar-bisect} is never used.

\begin{proposition}\label{prop:OSNE-1d} 
    When $d=1$, let $v_1\in\mathbb S^0$ be the unit direction used by Algorithm~\ref{alg:OSNE}. The algorithm reduces to a single comparison between $x$ and $x+h v_1$.  Still, suppose the preference relation is plateau-free and $r_x$-regular at $x$ with regularity radius $r_x>0$. The output $\widehat n$ is $-v_1$ if $x+h v_1\prec x$ and $v_1$ otherwise. In this case, $\widehat n=n_x$ if $h<2r_x$. The whole procedure uses only one comparison.
\end{proposition}

\begin{proof}
    In one dimension, $\mathbb S^0=\{\pm e\}$, so $n_x,\widehat n\in\{\pm e\}$.
    Applying Lemma~\ref{lem:curv-margin-OSNE} with $u=e$ and $u=-e$ shows that if $h < 2 r_x$, then $x+he\preccurlyeq x$ implies $\langle n_x,e\rangle \le \tfrac{h}{2r_x} < 1$, so $n_x = -e$; and $x\prec x+he$ implies $\langle n_x,e\rangle \ge -\tfrac{h}{2r_x} > -1$, so $n_x = e$.
\end{proof}

When a smooth function $f$ realizes the preference relation, the normal direction coincides with the normalized gradient. There are existing works in estimating the normalized gradient using comparison oracles. The most common approach is to use the random-direction sign estimator of the form $\mathscr{C}_f(x+hu,x) \,u$ where $\mathscr{C}_f$ is the comparison oracle of $f$ defined in \eqref{eq:Cf-def}. This estimator's expectation is proportional to the normalized gradient when $h$ is sufficiently small. See, for example, \cite{saha2021dueling,saha2025dueling,ren2026riemannian,cai2022one}, for this idea and its usage in sparse problems, Riemannian optimization and noisy oracle settings. However, the cost of simply taking the average of repetitions increases polynomially with $d/\epsilon$ for a target estimation error $\epsilon$.
A second approach, proposed in \cite{karabag2021smooth}, estimates the normalized gradient by maintaining a cone of plausible directions and shrinking this cone using half-space cuts extracted from comparison feedback. However, the cost grows quadratically with the dimension $d$ (see the proof of Theorems 1 and 2 in \cite{karabag2021smooth}). 

A third approach, studied in \cite{zhang2024comparisons}, considers $L$-smooth functions and proposes a two-stage method: first finding a dominant gradient component that is the largest in magnitude, and then estimating the ratio of this component to each of the other components of the gradient vector. After translating its notation to ours, its method gets an $\epsilon$-accurate normalized gradient using $O(d\,\log(d/\epsilon))$ comparisons, but requires the comparison radius to satisfy
$h\lesssim \epsilon\|\nabla f(x)\|/(L d^{3/2})$. Since Lemma~\ref{lem:regularityradius-hess-grad} gives
$r_x\ge \|\nabla f(x)\|/L$, this means
$h\lesssim \epsilon r_x/d^{3/2}$, which is more restrictive by an additional factor of $d$ than the admissible scale in Corollary~\ref{cor:OSNE-epsilon}. When preparing the first manuscript of this work, we realized a recent concurrent work \cite{tao2026gradient} proposes to first rotate the basis towards a reference direction (which is similar to the estimator of the first approach) and then continue with \cite{zhang2024comparisons}'s approach. This improves the required comparison radius to order $\epsilon\|\nabla f(x)\|/(L\sqrt d)$, but the guarantee is only with constant success probability because the reference-direction construction is randomized.

All of the above normalized gradient estimation methods assume the existence of an underlying smooth objective function and require function-dependent quantities such as a smoothness constant $L$ and a lower bound on $\|\nabla f(x)\|$. Our estimator is formulated directly in the function-free preference geometry. The next subsection addresses the remaining drawback of the need to pre-tune $h$ for convex preference by introducing an adaptive strategy.

\subsection{Normal direction estimation with adaptive comparison radius}
\label{subsec:convex-lineseach}
 
In the normal direction estimation method in Section \ref{subsec:normal-estimation}, the error bound of Theorem \ref{thm:OSNE} contains a bias term of the form $2\sqrt{d-1}\frac{h}{r_x}$, in which $r_x$ is the regularity radius of $x$. The regularity radius may vary substantially from point to point. As discussed in Section~\ref{subsec:optimality-stationarity} (and later in Section \ref{subsec:regularity-vs-distance}), a small regularity radius may itself signal proximity to an optimal or near-stationary region.  
Consequently, when normal direction estimation is used inside an optimization algorithm, it may be unrealistic to fix a single comparison radius along the entire trajectory.  
 
Therefore, in this subsection we develop a normal direction estimation method  that can automatically adjust the comparison radius using line search. According to Lemma \ref{lem:convex-supporting-halfspace}, if additionally assuming convexity of the current sublevel set $\calS_x$, we may decrease the comparison radius until comparisons yield an exact halfspace certificate of $ n_x$.
With exact halfspace certificates, the normal direction estimation error can be controlled purely by the bisection depth $T$, even when $r_x$ is unknown.
The price of line search appears only through a logarithmic dependence on $1/r_x$.

The method follows the same outer orthogonal-subspace construction (Algorithm~\ref{alg:OSNE}) and the same planar bisection method
(Algorithm~\ref{alg:planar-bisect}), but we implement their comparison steps through an additional subroutine (Algorithm \ref{alg:linesearch}) to potentially adjust the comparison radius $h$.
The comparison radius $h$ is stored in the global memory and shared across all calls to Algorithm~\ref{alg:OSNE}, Algorithm~\ref{alg:planar-bisect}, and Algorithm~\ref{alg:linesearch}, and it is initialized once to $h_0$ and may be decreased over time.

\begin{algorithm}[htbp]
\caption{Comparison with line search}
\label{alg:linesearch}
\begin{algorithmic}[1]
\State Input: base point $x$, direction $u\in\mathbb{S}^{d-1}$, and shared comparison radius $h$ initialized once to $h_0$. 
\While{{$x\prec x+h u$ \textbf{and} $x\prec x-h u$}}
  \State Halve the comparison radius $h\gets h/2$  
\EndWhile
\State
\textbf{Return} the updated $h$ and whether $x\prec x+h u$ and $x\prec x-h u$. 
\end{algorithmic}
\end{algorithm}

Before presenting the main theorem, we introduce three useful lemmas. 
The first lemma quantifies how many radius halvings are needed before termination of Algorithm~\ref{alg:linesearch}.

\begin{lemma}
\label{lem:linesearch}
Assume the preference relation is plateau-free and $r_x$-regular at $x$ with $r_x>0$.
Fix $u\in\mathbb{S}^{d-1}$ with $\langle n_x,u\rangle\neq 0$.
Consider one call of Algorithm~\ref{alg:linesearch} entered with comparison radius $h_{\mathrm{in}}>0$.
Then the while-loop terminates after at most
\begin{equation}
\label{eq:certsign-halvings}
\left\lceil \log_2\!\Big(\frac{h_{\mathrm{in}}}{r_x |\langle n_x,u\rangle|}\Big)\right\rceil_{+}
\end{equation}
halvings. 
\end{lemma}

\begin{proof}
Let $
t_\star
:=\left\lceil \log_2\!\left(\frac{h_{\mathrm{in}}}{r_x |\langle n_x,u\rangle|}\right)\right\rceil_{+},$ and $h_\star:=\frac{h_{\mathrm{in}}}{2^{t_\star}}$.
Then $h_\star\le r_x |\langle n_x,u\rangle|$, hence $h_\star\le r_x$.
Define the inward direction $u_{\mathrm{in}}:=-\sign(\langle n_x,u\rangle)u$ so that $\langle n_x,u_{\mathrm{in}}\rangle=-|\langle n_x,u\rangle|.$
Since $h_\star\le r_x |\langle n_x,u\rangle|$ and $|\langle n_x,u\rangle|>0$, we have $\langle n_x,u_{\mathrm{in}}\rangle=-|\langle n_x,u\rangle|<-\frac{h_\star}{2r_x}.$
Therefore,
\[
\|(x+h_\star u_{\mathrm{in}})-(x-r_x n_x)\|^2
=
\|h_\star u_{\mathrm{in}}+r_x n_x\|^2
=
h_\star^2+r_x^2+2h_\star r_x\langle u_{\mathrm{in}},n_x\rangle
<
r_x^2,
\]
so $x+h_\star u_{\mathrm{in}}\in \operatorname{int}B(x-r_x n_x,r_x)\subseteq \operatorname{int}(\calS_x)$.
By plateau-freeness, interior points cannot tie with $x$, hence $x+h_\star u_{\mathrm{in}}\prec x$.

In particular, $x\prec x+h_\star u_{\mathrm{in}}$ is false, so at radius $h_\star$ the while-condition
\(
x\prec x+hu \ \text{and}\ x\prec x-hu
\)
cannot hold. Thus the loop terminates after at most $t_\star$ halvings, which yields \eqref{eq:certsign-halvings}. 
\end{proof}

The second lemma is an anti-concentration bound for Haar directions 
which will be used to control how many halving steps are needed with high probability.

\begin{lemma} 
\label{lem:haar-smallball-simple}
Let $q$ be uniform on $\mathbb{S}^{d-1}$ and fix any $v\in\mathbb{S}^{d-1}$.
Then for every $\rho\in[0,1]$,
\[
\mathbb{P}\big(|\langle v,q\rangle|\le \rho\big)\ \le\ \sqrt{\frac{2d}{\pi}}\,\rho.
\]
\end{lemma}

This result could be derived by a corollary of the cosine-measure bound of independent random vectors uniformly distributed on the unit sphere, as stated in Lemma B.2 of \cite{gratton2015direct}.
 
\begin{proof}
The cosine measure bound in Lemma~B.2 of \cite{gratton2015direct} (letting $D=\{q\}$ and $m=1$) states that for any $\tau\in[0,\sqrt{d}]$, the cosine measure $\mathrm{cm}(D,v)$ (which essentially is $\langle v,q\rangle$) satisfies   
\[
\mathbb{P}\!\left(\langle v,q\rangle\ge \frac{\tau}{\sqrt{d}}\right)
\ \ge\ 1-\left(\frac12+\frac{\tau}{\sqrt{2\pi}}\right)
\ =\ \frac12-\frac{\tau}{\sqrt{2\pi}}.
\]
Letting $\tau=\rho\sqrt{d}$ yields $
\mathbb{P}\big(\langle v,q\rangle \ge \rho\big)\ \ge\ \frac12-\rho\sqrt{\frac{d}{2\pi}}$ for $\rho \in [0,1]$.
By symmetry $\langle v,q\rangle$ has the same distribution as $-\langle v,q\rangle$, hence
$\mathbb{P}(\langle v,q\rangle \le -\rho)=\mathbb{P}(\langle v,q\rangle \ge \rho)$.
Therefore,
\[
\mathbb{P}\big(|\langle v,q\rangle|\le \rho\big)
=1-\mathbb{P}(\langle v,q\rangle \ge \rho)-\mathbb{P}(\langle v,q\rangle \le -\rho)
\le 1-2\left(\frac12-\rho\sqrt{\frac{d}{2\pi}}\right)
=\sqrt{\frac{2d}{\pi}}\,\rho.
\]This finishes the proof.
\end{proof}

The third lemma states that once every queried direction comes with the correct sign of $\langle n_x,u\rangle$, planar bisection becomes exact
and the local-geometry-induced bias term disappears from the planar error bound (compared with Lemma \ref{lem:planar-bisect}).

\begin{lemma}
\label{lem:planar-bisect-cvx}
Assume $\calS_x$ is closed and convex, and the preference relation is plateau-free and $r_x$-regular at $x$ with $r_x>0$.
Run Algorithm~\ref{alg:planar-bisect} to depth $T$,
implementing Algorithm~\ref{alg:linesearch} before all comparisons.
In the event that Algorithm~\ref{alg:linesearch} always terminates, the returned unit vector $y$ almost surely satisfies
\begin{equation}
\label{eq:planar-bisect-cert-bound}
|\langle n_x,y\rangle|\ \le\ \frac{\pi}{2^{T+2}}.
\end{equation} 
\end{lemma}

\begin{proof} 

Fix any direction $u\in\mathbb{S}^{d-1}$ queried by Algorithm~\ref{alg:planar-bisect}.
Since Algorithm~\ref{alg:linesearch} terminates, at the final radius $h$ we have at least one of $x+hu\preccurlyeq x$ or $x-hu\preccurlyeq x$.
Using the comparison outcomes already produced in the final check of the while-condition in Algorithm~\ref{alg:linesearch}, we obtain a one-sided sign certificate.
If $\langle n_x,u\rangle>0$ then Lemma~\ref{lem:convex-supporting-halfspace} implies $x+hu\notin \calS_x$ and hence $x\prec x+hu$.
If $\langle n_x,u\rangle<0$, then Lemma~\ref{lem:convex-supporting-halfspace} applied to $-u$ implies
$x-hu\notin\calS_x$, hence $x\prec x-hu$.
Since the line-search loop has terminated, by completeness, $x+hu\preccurlyeq x$.

Consequently, Algorithm~\ref{alg:planar-bisect} performs an exact bisection on the sign of the continuous function
$u\mapsto \langle n_x,u\rangle$ restricted to the unit circle in the plane $\mathrm{span}\{a,b\}$.
If $\mathrm{proj}_{\mathrm{span}\{a,b\}}(n_x)=0$ (although later we will prove it is with zero probability), then $\langle n_x,y\rangle=0$ for every returned $y$ and the claim is trivial.
Otherwise, there exists a tangential direction $u^\star\in \mathrm{span}\{a,b\}\cap\mathbb{S}^{d-1}$ with
$\langle n_x,u^\star\rangle=0$.

The initialization step using the orthonormal pair $(\bar a,\bar b)$ produces a bracket $(u_-,u_+)$ lying in a quadrant, hence
$\angle(u_-,u_+)\le \pi/2$, and let $u^\star$ lie on the shorter arc between them. Each bisection step replaces an endpoint by the geodesic midpoint, halving the bracket arc length, while $u^\star$. After $T$ bisections, the final bracket length is at most $\frac{\pi}{2^{T+1}}$, so the returned midpoint $y$ satisfies
$\angle(y,u^\star)\le \frac{\pi}{2^{T+2}}$.
Finally,
\[
|\langle n_x,y\rangle|
=
|\langle n_x,y-u^\star\rangle|
\le
\|y-u^\star\|
\le
\angle(y,u^\star)
\le
\frac{\pi}{2^{T+2}},
\]
which is \eqref{eq:planar-bisect-cert-bound}.
\end{proof}

With Lemmas \ref{lem:linesearch}, \ref{lem:haar-smallball-simple} and \ref{lem:planar-bisect-cvx}, now we are ready to state the main complexity bound result for normal direction estimation with adaptive comparison radius.

\begin{theorem} [Adaptive normal direction estimation]
\label{thm:OSNE-convex-oracle}
Assume $d\ge2$ and fix $x\in\X$.
Suppose the preference relation is plateau-free and $r_x$-regular at $x$ with regularity radius $r_x>0$, and assume moreover that $\calS_x$ is convex.  
Fix $\epsilon\in(0,\sqrt{2}]$ and define the bisection depth
\begin{equation}
\label{eq:T-epsilon}
T_\epsilon
\;:=\;
\left\lceil \log_2\!\Big(\frac{\pi\sqrt{d-1}}{2\epsilon}\Big)\right\rceil.
\end{equation}
Run the outer construction of Algorithm~\ref{alg:OSNE} with depth $T_\epsilon$, initializing the shared radius to $h_0$. In the regime with adaptive comparison radius, the algorithm terminates almost surely and returns $\widehat n$ satisfying
  $\|\widehat n-n_x\|\le \epsilon$.
  Moreover, for every $\delta\in(0,1)$, with probability at least $1-\delta$,
    \begin{equation}
  \label{eq:Nhpp}
  N_\epsilon
  \le 
  2(d-1)\left\lceil \log_2\!\left(\frac{7\sqrt{d-1}}{\epsilon}\right)\right\rceil+ 
  2\left\lceil \log_2\!\left(\frac{h_0\,d^3}{r_x \epsilon}\right)\right\rceil_{+} + 16
  \;+\;4\left\lceil \log_2\!\left(\frac{1}{\delta}\right)\right\rceil.
  \end{equation}
\end{theorem}

\begin{proof}
Let $T$ denote $T_\epsilon$ in this proof. Recall that the comparison radius $h$ is shared globally, initialized at $h_0$, and is updated only by Algorithm~\ref{alg:linesearch}.

We first prove the accuracy guarantee $\|\widehat n-n_x\| \le \epsilon$.  
By Lemma~\ref{lem:OSNE-invariants}, the vectors $\{y_1,\dots,y_{d-1},v_d\}$ form an orthonormal basis. 
Write
\[
n_x=\alpha_0\,v_d+\sum_{i=1}^{d-1}\beta_i\,y_i,
\qquad\text{where}\quad
\beta_i:=\langle n_x,y_i\rangle.
\]
Lemma~\ref{lem:planar-bisect-cvx} implies that each call of Algorithm~\ref{alg:planar-bisect} returns a unit vector $y_i$ satisfying $|\langle n_x,y_i\rangle| \le \frac{\pi}{2^{T+2}}$, so
 $\sqrt{\sum_{i=1}^{d-1}\beta_i^2} \le \sqrt{d-1}\cdot \frac{\pi}{2^{T+2}}.$
The last step of Algorithm~\ref{alg:OSNE} chooses $\widehat n\in\{\pm v_d\}$ so that $\langle n_x,\widehat n\rangle=|\alpha_0|$.
Therefore,
\[
\|n_x-\widehat n\|
=
\sqrt{(1-|\alpha_0|)^2+\|\beta\|^2}
\le
\sqrt{2}\|\beta\|\le \sqrt{2d-2}\cdot \frac{\pi}{2^{T+2}} \ .
\] 
By \eqref{eq:T-epsilon}, $2^T\ge \frac{\pi\sqrt{d-1}}{2\epsilon}$, so the right-hand side is at most $\epsilon$. Furthermore, note that as $\epsilon \in(0,\sqrt{2}]$ and $d \ge 2$, $\left\lceil \log_2\!\left(\tfrac{\pi\sqrt{d-1}}{2\epsilon}\right)\right\rceil \ge \left\lceil \log_2\!\left(\tfrac{\pi}{2\sqrt{2}}\right)\right\rceil \ge 1$, so $T_\epsilon$ defined in \eqref{eq:T-epsilon} is in $\mathbb{N}_{++}$, and the bisection depth is valid.
This proves that $\|n_x-\hat{n}\| \le \epsilon$.

With the accuracy guarantee established, the rest of the proof counts the number of comparisons.
Algorithm~\ref{alg:OSNE} makes $d-1$ calls of Algorithm~\ref{alg:planar-bisect} and one final orientation comparison, and each call of Algorithm~\ref{alg:planar-bisect} uses at most $T_\epsilon+3$ comparisons (in lines \ref{line:pb-compare-a-line-search}, \ref{line:pb-compare-b-line-search} and \ref{line:pb-compare-um-line-search}). Note that these comparisons can be directly read off from the final while-condition evaluation of Algorithm~\ref{alg:linesearch} without additional comparisons. 
Therefore, later we consider only the number of calls to Algorithm~\ref{alg:linesearch} and the number of halvings inside it. 
For each call of Algorithm~\ref{alg:planar-bisect}, there are at most $T_\epsilon+2$ calls of Algorithm~\ref{alg:linesearch}. Overall, at most $Q:=(d-1)(T_\epsilon+2)+1$ calls of Algorithm~\ref{alg:linesearch} are used in Algorithm~\ref{alg:OSNE}.

Then let $H$ be the total number of assignments $h\gets h/2$ made across all calls.
A call with $H_{\mathrm{call}}$ halvings checks the condition $H_{\mathrm{call}}+1$ times, making $2H_{\mathrm{call}}+2$ comparisons.
Summing over calls gives
\begin{equation}
\label{eq:N=2(Q+H)}
N_\epsilon \;\le \; 2Q+2H = 2(d-1)(T_\epsilon+2)+2 + 2H.
\end{equation}
It remains to control $H$.
 
Let $\mathcal{U}$ be the set of all directions $u$ ever passed to Algorithm~\ref{alg:linesearch}.
Lemma~\ref{lem:linesearch} implies that  
\begin{equation}
\label{eq:H-vs-eta-min}
H\ \le\ \left\lceil \log_2\!\Big(\frac{ h_0}{r_x\cdot \min_{u\in\mathcal{U}}|\langle n_x,u\rangle|}\Big)\right\rceil_{+}.
\end{equation}
Therefore, in order to bound $H$, we need to first lower bound $\min_{u\in\mathcal{U}}|\langle n_x,u\rangle|$ with high probability.

Fix the $i$-th call of Algorithm \ref{alg:planar-bisect}, the corresponding plane is $\mathrm{span}\{v_i,q_{i+1}\}$.
After the random rotation in line~2 of Algorithm~\ref{alg:planar-bisect}, the algorithm uses an orthonormal basis, denoted by $(\bar a_i,\bar b_i)$, of this plane. Consider the equally spaced grid on the unit circle in this plane:
\begin{equation}\label{eq:calGi}
\mathcal{G}_i
:=
\left\{
\cos\left(\tfrac{k\pi}{2^{T+1}}\right)\,\bar a_i+\sin\left(\tfrac{k\pi}{2^{T+1}}\right)\,\bar b_i
:\ k=0,1,\dots,2^{T+2}-1
\right\}.
\end{equation}
Up to depth $T$, Algorithm~\ref{alg:planar-bisect} only queries endpoints and dyadic midpoints of a quadrant arc, hence every direction queried in Algorithm \ref{alg:planar-bisect} lies in $\mathcal{G}_i$. Therefore
\begin{equation}\label{eq:eta-min-grid0}
\min_{u\in \mathcal{U}}|\langle n_x,u\rangle|
\ \ge\
\min_{i=1,\dots,d-1}\min_{u\in\mathcal{G}_i}|\langle n_x,u\rangle|.
\end{equation}
Next we will derive the lower bound for each $ \min_{u\in\mathcal{G}_i}|\langle n_x,u\rangle|$.

Let $\hat{u}_i$ denote $\mathrm{proj}_{\mathrm{span}\{v_i,q_{i+1}\}}(n_x)$ and let $u_i^\star$ be any unit direction in $\mathrm{span}\{v_i,q_{i+1}\}$ satisfying $\langle n_x,u_i^\star\rangle=0$
(equivalently, $u_i^\star$ is orthogonal to $\hat{u}_i$).
Define the smallest angular separation between the line of $u_i^\star$ and the grid endpoints as follows:
\begin{equation}\label{eq:Deltai}
    \Delta_i
:=
\min_{g\in\mathcal{G}_i}\min\{\angle(g,u_i^\star),\ \angle(g,-u_i^\star)\}. 
\end{equation}
Then for any queried direction $u\in\mathcal{G}_i$,
\begin{equation}\label{eq:eta-lower-bound}
\begin{aligned}
    & 
|\langle n_x,u\rangle| = |\langle \hat{u}_i,u\rangle + \langle n_x - \hat{u}_i,u\rangle| =  |\langle \hat{u}_i,u\rangle| = \|\hat{u}_i\| \cdot \|u\| \cdot |\cos(\angle(\hat{u}_i,u))| \\
= \ &  \|\hat{u}_i\| \cdot \|u\| \cdot |\sin(\angle(u_i^\star,u))| \ge  \|\hat{u}_i\| \cdot \|u\| \cdot \sin(\Delta_i)  = \|\hat{u}_i\|  \cdot \sin(\Delta_i) 
\end{aligned}
\end{equation}
where the second equality is because $n_x - \hat{u}_i$ is orthogonal to the plane of $u$, the fourth equality is because $\hat{u}_i$, $u_i^\star$ and $u$ are in the same plane and $\langle u_i^\star, \hat{u}_i \rangle = 0$, and the inequality is because of the definition of $\Delta_i$ in \eqref{eq:Deltai}. 

Furthermore, since  $\Delta_i\le \pi/2$, $ \sin(\Delta_i) \ge \frac{2}{\pi}\Delta_i$. In addition, since $v_i$ and $q_{i+1}$ are orthonormal, $\hat{u}_i = \mathrm{proj}_{\mathrm{span}\{v_i,q_{i+1}\}}(n_x) = \langle n_x,v_i\rangle v_i + \langle n_x,q_{i+1}\rangle q_{i+1}$, we have $\|\hat{u}_i\| \ge |\langle n_x,q_{i+1}\rangle|$. Therefore,  substituting these two lower bounds back to \eqref{eq:eta-lower-bound} gives a lower bound for each $ \min_{u\in\mathcal{G}_i}|\langle n_x,u\rangle|$:
\begin{equation}\label{eq:eta-lower-bound2} 
\min_{u\in\mathcal{G}_i}|\langle n_x,u\rangle|  \ge |\langle n_x,q_{i+1}\rangle|\cdot \frac{2}{\pi}\Delta_i
\end{equation}
Finally, using \eqref{eq:eta-min-grid0} and take the minimum over $i=1,\dots,d-1$ yields
\begin{equation}
\label{eq:eta-min-grid}
\min_{u\in\mathcal{U}}|\langle n_x,u\rangle|
\ \ge\
\frac{2}{\pi}\,
\underbrace{\Big(\min_{2\le j\le d}|\langle n_x,q_j\rangle|\Big)}_{=:A}\,
\underbrace{\Big(\min_{1\le i\le d-1}\Delta_i\Big)}_{=:B}.
\end{equation}
Then we use $A$ and $B$ to denote the two product terms in the right-hand side of \eqref{eq:eta-min-grid} for brevity.

To further lower bound $\min_{u\in\mathcal{U}}|\langle n_x,u\rangle|$, we now bound $A$ and $B$ from below with high probability.
By Lemma~\ref{lem:haar-smallball-simple} and a union bound,
\begin{equation}\label{eq:boundA}
\mathbb{P}(A\le \alpha)\ \le\ (d-1)\sqrt{\frac{2d}{\pi}}\,\alpha.
\end{equation}
For $B$, fix $i$ and condition on the input orthonormal vectors $v_i$ and $q_{i+1}$, then the random in-plane rotation is controlled by the random variable $\theta\sim \mathrm{Unif}[0,2\pi)$.
That rotation makes $\mathcal{G}_i$ a uniform random rotation of a fixed equally spaced grid of spacing $\frac{\pi}{2^{T+1}}$.
Equivalently, the position of $u_i^\star$ (and $-u_i^\star$) inside its nearest length-$\frac{\pi}{2^{T+1}}$ interval is uniform on $[0,\frac{\pi}{2^{T+1}})$.
Hence for any $\beta\in[0,\tfrac12]$,
\begin{equation}\label{eq:boundB}
     \mathbb{P}\left(B\le \beta \cdot \frac{\pi}{2^{T+1}}\right) \le 
     \sum_{i=1}^{d-1}\mathbb{P}\left(\Delta_i\le \beta \cdot \frac{\pi}{2^{T+1}}\right) \le \sum_{i=1}^{d-1} 2\beta = 2(d-1)\beta.
\end{equation}
With \eqref{eq:boundA} and \eqref{eq:boundB}, choose $
\alpha:=\frac{\delta}{2(d-1)}\sqrt{\frac{\pi}{2d}}$ and $\beta:=\frac{\delta}{4(d-1)}$, then by a union bound $\mathbb{P}(A\ge \alpha,\ B\ge \beta \cdot \tfrac{\pi}{2^{T+1}})\ge 1-\delta$, and on this event \eqref{eq:eta-min-grid} gives
\[
\min_{u\in\mathcal{U}}|\langle n_x,u\rangle|
\ \ge\
\frac{2}{\pi}\cdot \frac{\delta^2}{8(d-1)^2}\sqrt{\frac{\pi}{2d}}\cdot \frac{\pi}{2^{T+1}}
\]

Plugging this into \eqref{eq:H-vs-eta-min} and using $2^T\le \frac{\pi\sqrt{d-1}}{\epsilon}$ and $(d-1)^{5/2}\sqrt{d}\le d^3$ (for $d\ge 2$),
we obtain
\begin{equation}
\label{eq:H-hpp-final}
H
\ \le\
\left\lceil
\log_2\!\Big(
\frac{32\sqrt{2\pi}\, h_0\,d^3}{r_x\,\epsilon\,\delta^2}
\Big)
\right\rceil_{+}
\end{equation} with probability at least $1-\delta$. 

Finally, substituting \eqref{eq:H-hpp-final} and $T_\epsilon
= \left\lceil \log_2\!\left(\frac{\pi\sqrt{d-1}}{2\epsilon}\right)\right\rceil$ (by \eqref{eq:T-epsilon}) into \eqref{eq:N=2(Q+H)}, we have
\begin{equation}
\label{eq:N-hpp-final}
\begin{aligned}
    N_\epsilon  & \le  2(d-1)\left(\left\lceil \log_2\!\left(\frac{\pi\sqrt{d-1}}{2\epsilon}\right)\right\rceil+2\right)+2 + 2
\left\lceil
\log_2\!\Big(
\frac{32\sqrt{2\pi}\, h_0\,d^3}{r_x\, \epsilon\,\delta^2}
\Big) 
\right\rceil_{+} \\
& \le   2(d-1)\left\lceil \log_2\!\left(\frac{7\sqrt{d-1}}{\epsilon}\right)\right\rceil+16 + 2
\left\lceil
\log_2\!\Big(
\frac{h_0\,d^3}{r_x\,\epsilon\,\delta^2}
\Big) 
\right\rceil_{+} 
\end{aligned}
\end{equation}
where the second inequality is because $4\cdot \tfrac{\pi}{2}\le 7$ and $\log_2(32\sqrt{2\pi})\approx 6.33 \le 7$.
  Using the ceiling inequality $
  \left\lceil \log_2\!\Big(\frac{C}{\delta^2}\Big)\right\rceil_{+}
  \le
  \left\lceil \log_2(C)\right\rceil_{+}
  +2\left\lceil \log_2\!\Big(\frac{1}{\delta}\Big)\right\rceil$,
  with $C=\frac{h_0\,d^3}{r_x\,\epsilon}$, 
  we obtain \eqref{eq:Nhpp}. 
\end{proof}

Theorem~\ref{thm:OSNE-convex-oracle} shows that the estimator is adaptive to the unknown local geometry.
For any target accuracy $\epsilon$, the procedure does not require any prior knowledge of the regularity radius $r_x$ or the curvature $\kappa_x$.
Instead, it starts from an arbitrary initial comparison radius $h_0$ and automatically shrinks the radius until the comparisons become informative at the scale required by the local geometry at $x$. 

The dependence in the complexity bound \eqref{eq:Nhpp} is also transparent. A larger initial $h_0$ may cause more line-search halving steps.
Likewise, a smaller regularity radius $r_x$ corresponds to sharper local geometry and hence requires more comparisons.
The good news is that all of these dependencies on $h_0$ and $r_x$ are only logarithmic.

The randomness comes from the internal randomness of the estimator in generating the Haar directions and uniform distribution. 
Its effect is also logarithmic through the term $4\left\lceil \log_2(1/\delta)\right\rceil$, which means the total number of comparisons has an exponentially decaying tail and a finite expectation of the same order as the bound in \eqref{eq:Nhpp}. 

When the dimension $d$ is large, the leading contribution in \eqref{eq:Nhpp} is the term $2(d-1)\left\lceil \log_2\!\left(\frac{7\sqrt{d-1}}{\epsilon}\right)\right\rceil$. 
In this case, the adaptive estimator is about a factor of two more expensive than the fixed-comparison-radius estimator, whose complexity is given by \eqref{eq:OSNE-comparisons-bound} in Corollary \ref{cor:OSNE-epsilon}.
This extra factor is the price of adaptivity: during line search, each check requires comparisons in both the $+u$ and $-u$ directions, whereas the fixed-radius scheme does not need two-sided comparisons.

Finally, it should also be noted that the case $d=1$ should be separately analyzed. In this case Algorithm~\ref{alg:planar-bisect} is never used.  See the proposition below. 

\begin{proposition}
\label{prop:OSNE-1d-convex}
Assume $d=1$, in the same setting as Theorem~\ref{thm:OSNE-convex-oracle}, the output $\widehat n$ is exactly $n_x$, and the algorithm terminates after at most $2+2\bigl\lceil \log_2(h_0/r_x)\bigr\rceil_{+}$
comparisons.
\end{proposition}

\begin{proof}
In one dimension, $v_1\in\mathbb S^0$ and $n_x\in\{\pm v_1\}$, so $\langle n_x,v_1\rangle\in\{\pm 1\}$ and in particular
$|\langle n_x,v_1\rangle|=1$. 
Since the while-condition of Algorithm~\ref{alg:linesearch} is false at termination either $x+h v_1\preccurlyeq x $ or $x-h v_1\preccurlyeq x$.

If $x+h v_1\preccurlyeq x$, then by Lemma~\ref{lem:convex-supporting-halfspace} we have
$\langle n_x,v_1\rangle\le 0$, which forces $\langle n_x,v_1\rangle=-1$,
i.e.\ $n_x=-v_1$,
Otherwise $x+h v_1\not\preccurlyeq x$, so necessarily $x-h v_1\preccurlyeq x$.
Applying Lemma~\ref{lem:convex-supporting-halfspace} to $-v_1$ yields
$\langle n_x,-v_1\rangle\le 0$, i.e.\ $\langle n_x,v_1\rangle\ge 0$.
Thus $\langle n_x,v_1\rangle=1$ and hence $n_x=v_1$.
Therefore $\widehat n=n_x$.

For the comparison count, Lemma~\ref{lem:linesearch} gives that a single call to Algorithm~\ref{alg:linesearch} performs at most
$\left\lceil \log_2(h_0/(r_x|\langle n_x,v_1\rangle|))\right\rceil_{+}
=
\left\lceil \log_2(h_0/r_x)\right\rceil_{+}$
radius halvings.
Each while-check uses two comparisons, so the total number of comparisons is at most $2+2\lceil \log_2(h_0/r_x)\rceil_{+}$.
\end{proof}

\subsection{Lower bound of normal direction estimation methods}
\label{sec:normal-lower-bound}

This subsection shows that for the task of estimating the normal direction, no algorithm can guarantee worst-case error $\epsilon$ within $o(d\cdot\log(1/\epsilon))$ comparisons without additional assumptions. 
Consequently, both normal direction estimation methods (with fixed or adaptive comparison radii) proposed in this section are nearly optimal.

We establish the bound on preference relations realized by linear functions $\langle c, x\rangle$ for $c \in \mathbb{S}^{d-1}$. They are plateau-free, convex and regular at all $x\in\X$, and $c$ is the normal direction. Theorem \ref{thm:minimax} below shows that in this class of relations, there always exist two different normal directions that are hard to distinguish with a limited number of comparisons.

\begin{definition}[Algorithms with \(t\) comparisons]
\label{def:adaptive-comparison-algorithm}
An algorithm with \(t\) comparisons proceeds sequentially and has access only to its past comparison history and its own private randomness.
We represent all private randomness by a random seed $\xi$ sampled before the run.
At round $i=1,\dots,t$, after observing the previous comparison history
\[
\mathcal T_{i-1}:=\bigl((p_1,q_1,\sigma_1),\dots,(p_{i-1},q_{i-1},\sigma_{i-1})\bigr),
\]
where $\sigma_j\in\{-1,0,+1\}$ is the comparison oracle response to the pair $(p_j,q_j)$, the algorithm chooses the next pair
$(p_i,q_i)\in\X\times\X$ as a measurable function of $(\mathcal T_{i-1},\xi)$.
After $t$ rounds, its output is a measurable function of $(\mathcal T_t,\xi)$.
\end{definition}

\begin{theorem}[Lower bound of normal direction estimation]
\label{thm:minimax}
Fix $d\ge 2$ and $t\in\mathbb N_{++}$.
For each $c\in\mathbb S^{d-1}$, let $\preccurlyeq_c$ denote the preference relation realized by
$f_c(x):=\langle c,x\rangle$.
For any algorithm with \(t\) comparisons in the sense of Definition~\ref{def:adaptive-comparison-algorithm}, and for every realization of its private random seed $\xi$, there exist
$c,\hat c\in\mathbb S^{d-1}$ such that
\begin{equation}
\label{eq:minimax-angle}
\angle(c,\hat c)
\ > \ 2^{-(t+1)/(d-1)},
\end{equation}
and, under the same $\xi$, the algorithm receives the same comparison transcript under $\preccurlyeq_c$ and $\preccurlyeq_{\hat c}$.
\end{theorem}

\begin{proof}
Fix an arbitrary realization of the private random seed $\xi$.
Conditional on this seed, the algorithm's next comparison pair is a deterministic function of the previous comparison history.

For the linear preference relation $\preccurlyeq_c$, a comparison between $p$ and $q$ returns the sign of
$\langle c,q-p\rangle$. 
If $p=q$, the comparison gives no information.
Let $D_0:=\mathbb S^{d-1}$ and let $\operatorname{area}(\cdot)$ denote surface area on $\mathbb S^{d-1}$.
We adversarially construct a comparison history.
Suppose that after $i-1$ replies, the set of directions consistent with the comparison history is a measurable set
$D_{i-1}\subseteq\mathbb S^{d-1}$.
Since the seed $\xi$ is fixed, this comparison history uniquely determines the next comparison pair $(p_i,q_i)$.
If $p_i=q_i$, set $D_i:=D_{i-1}$ and continue.
Otherwise, set $u_i:=(q_i-p_i)/\|q_i-p_i\|$.
The two possible nonzero replies correspond to
\[
D_{i-1}^+ := D_{i-1}\cap\{c:\ \langle c,u_i\rangle>0\},
\qquad
D_{i-1}^- := D_{i-1}\cap\{c:\ \langle c,u_i\rangle<0\}.
\]
The boundary has surface area zero.
Hence at least one of $D_{i-1}^+$ and $D_{i-1}^-$ has area at least
$\frac12\operatorname{area}(D_{i-1})$.
The adversary chooses the corresponding reply and sets $D_i$ to that larger cell.
Iterating yields
\begin{equation}
\label{eq:version-area}
\operatorname{area}(D_t)
\ \ge\
2^{-t}\operatorname{area}(\mathbb S^{d-1}).
\end{equation}

For $\alpha\in[0,\pi]$ and $v\in\mathbb S^{d-1}$, define
$\mathrm{Cap}(v,\alpha):=\{w\in\mathbb S^{d-1}:\angle(w,v)\le\alpha\}$. And we let $\alpha_t\in(0,\pi)$ be the root of
\begin{equation}
\label{eq:alpha-def}
\frac{\operatorname{area}(\mathrm{Cap}(v,\alpha_t))}{\operatorname{area}(\mathbb S^{d-1})}
=
2^{-(t+1)} ,
\end{equation}
which is independent of $v$ by rotational symmetry. 
Choose any $c\in D_t$.
By \eqref{eq:version-area}, $D_t$ has area at least twice the area of $\mathrm{Cap}(c,\alpha_t)$.
Therefore $D_t\setminus \mathrm{Cap}(c,\alpha_t)$ has positive area, and we may choose
$\hat c\in D_t\setminus \mathrm{Cap}(c,\alpha_t)$.

Moreover, $c$ and $\hat c$ both lie in the same cell $D_t$, which is the intersection of the strict hemispheres selected by the constructed history.
Therefore they give the same nonzero sign for every comparison along that history.
Since the seed $\xi$ is fixed, identical replies up to round $i-1$ force the algorithm to choose the same comparison pair at round $i$.
By induction over $i=1,\dots,t$, the full comparison history is identical under $\preccurlyeq_c$ and $\preccurlyeq_{\hat c}$.

It remains to lower bound $\alpha_t$.
By the spherical-cap formula,
\[
\frac{\operatorname{area}(\mathrm{Cap}(v,\alpha_t))}{\operatorname{area}(\mathbb S^{d-1})}
=
\frac{\int_0^{\alpha_t}\sin^{d-2}\theta\,d\theta}
{\int_0^\pi\sin^{d-2}\theta\,d\theta}
=
2^{-(t+1)}.
\]
Using $\sin\theta\le \theta$ on $[0,\alpha_t]$, we get
\begin{equation}
\label{eq:angle-lb-alpha}
\frac{\alpha_t^{d-1}}{d-1}
\ \ge\
2^{-(t+1)}\int_0^\pi \sin^{d-2}\theta\,d\theta .
\end{equation}
Moreover, by the trigonometric representation of the Beta function,
\[
\int_0^\pi \sin^{d-2}\theta\,d\theta
=
\sqrt{\pi}\,\frac{\Gamma\big(\tfrac{d-1}{2}\big)}{\Gamma\big(\tfrac{d}{2}\big)}
=
\sqrt{\pi}\,\frac{\Gamma(x)}{\Gamma(x+\tfrac12)}
\]
for $x:=\frac{d-1}{2}$.
By Wendel's inequality, for every $x>0$ and every $s\in(0,1)$, $\frac{\Gamma(x+s)}{\Gamma(x)} \le x^{s}.$
Taking $s=\tfrac12$ yields $\Gamma(x+\tfrac12)/\Gamma(x)\le \sqrt{x}$ and hence  
\[
\int_0^\pi \sin^{d-2}\theta\,d\theta
\;\ge\;
\sqrt{\pi}\,\frac{1}{\sqrt{x}}
=
\sqrt{\frac{2\pi}{d-1}}
\;\ge\;
\sqrt{\frac{2\pi}{d}}.
\]
Substituting this into \eqref{eq:angle-lb-alpha} yields
$
\alpha_t
\ge
\left(\sqrt{2\pi/d}\,(d-1)\,2^{-(t+1)}\right)^{1/(d-1)}.
$
Since $\big(\sqrt{2\pi/d}\,(d-1)\big)^{1/(d-1)}\ge 1$ for all $d\ge2$, we obtain
$\alpha_t\ge 2^{-(t+1)/(d-1)} $.
This proves \eqref{eq:minimax-angle}.
\end{proof}

With this theorem, we can immediately obtain the following lower bound for normal direction estimation under the general class of preferences considered in this section.

\begin{corollary}
\label{cor:minimax-general}
Fix $d\ge 2$ and $\epsilon\in(0,1/(2\pi))$.
Consider any algorithm with \(t\) comparisons in the sense of Definition~\ref{def:adaptive-comparison-algorithm}.
Suppose that, for every realization of $\xi$ and every $c\in\mathbb S^{d-1}$, when the preference relation is $\preccurlyeq_c$, the algorithm outputs $\widehat n\in\mathbb S^{d-1}$ satisfying
$\|\widehat n-c\|\le\epsilon$.
Then
\begin{equation}
\label{eq:minimax-N-bound}
t
\ > \
(d-1)\log_2\!\left(\frac{1}{\pi\epsilon}\right) - 1.
\end{equation}
Consequently, any comparison based algorithm that guarantees worst-case normal-direction error at most $\epsilon$ over the class of plateau-free and regular preference relations must use
$\Omega(d\log(1/\epsilon))$ comparisons.
\end{corollary}

\begin{proof}
The family $\{\preccurlyeq_c:c\in\mathbb S^{d-1}\}$ is a subclass of the preference relations considered here.
Fix any realization of the private random seed.
By Theorem~\ref{thm:minimax}, there exist $c,\hat c\in\mathbb S^{d-1}$ with identical comparison transcripts under this realization and
$\angle(c,\hat c)> 2^{-(t+1)/(d-1)}$.
Since the algorithm receives the same transcript, it produces the same output $\widehat n$ on both instances.
The assumed guarantee gives $\|\widehat n-c\|\le\epsilon$ and $\|\widehat n-\hat c\|\le\epsilon$.
Therefore,
$\angle(c,\hat c)
\le
\angle(c,\widehat n)+\angle(\widehat n,\hat c)
\le
\pi\epsilon$,
where we used $\|a-b\|=2\sin(\angle(a,b)/2)$ and $\sin z\ge 2z/\pi$ for $z\in[0,\pi/2]$.
Thus $2^{-(t+1)/(d-1)} < \pi\epsilon$, which rearranges to \eqref{eq:minimax-N-bound}.
\end{proof}

Corollary  \ref{cor:minimax-general} shows that to guarantee worst-case error smaller than $\epsilon$ in normal direction estimation, any algorithm must use at least on the order of $d\cdot\log(1/\epsilon)$ comparisons, which matches the scaling achieved by the proposed methods up to constants and logarithmic factors in $d$ (Corollary \ref{cor:OSNE-epsilon}; or additionally in $h_0,1/r_x,$ and $1/\delta$ in Theorem \ref{thm:OSNE-convex-oracle}). 

 During the preparation of the first version of this manuscript, the concurrent work~\cite{tao2026gradient} proves a lower bound $\Omega(d\,\log(1/\epsilon))$ of comparison-based method for estimating a differentiable function's gradient direction up to $\epsilon$. Directly using their results also implies a lower bound for normal direction estimation. They use a different approach by first constructing a maximal $\epsilon$-net of the unit sphere and then showing that no method can distinguish an unknown point in the net in a certain number of comparisons.

\subsection{From normal directions to cutting-plane methods}
\label{subsec:cutting-plane}
 
Recall that our goal is to solve \eqref{eq:main-optimization-problem}, where we assume $\mathcal{C}$ is convex and having access to efficient Euclidean projections onto $\mathcal{C}$.
When the preference relation is convex, access to normal directions and Euclidean projections onto $\mathcal{C}$ immediately yield valid separating halfspaces that contain the optimal set $\X^\star$.
Consequently, one may directly use  cutting-plane methods, including the ellipsoid method and the analytic- or volumetric-center cutting-plane methods, to obtain logarithmic dependence on the target accuracy in its complexity.
The main limitation is dimension: even for the best-known cutting-plane methods, the number of separating halfspaces scales linearly with $d$, and the per-iteration arithmetic and memory costs typically scale more steeply with $d$.

\medskip
\noindent
\textbf{Separating halfspaces from normal directions.}
Assume throughout this subsection that the preference relation is plateau-free and convex, so every sublevel set $\calS_x$ is convex. 
A valid cut at a query point $x$ is a closed halfspace $H$ such that $\X^\star \subseteq H$ and $x\notin \operatorname{int}(H)$.
When $x\in \mathcal{C}\setminus \X^\star$ and the preference relation is $r_x$-regular at $x$ with $r_x>0$, the outward normal direction $n_x$ defines the supporting halfspace
\begin{equation}
\label{eq:H-optimality}
H_x:= \{y\in\X:\ \ip{n_x}{y-x}\le 0\}.
\end{equation}
By Lemma~\ref{lem:convex-supporting-halfspace}, $\calS_x \subseteq H_x$.
Moreover, every $x^\star\in\X^\star$ satisfies $x^\star\preccurlyeq x$, hence $x^\star\in\calS_x$. Therefore $\X^\star \subseteq \calS_x \subseteq H_x$.
Thus $H_x$ is a valid cut.

If instead $x\notin \mathcal{C}$, we can use a feasibility cut.
Let $\hat x=\Pi_{\mathcal{C}}(x)$ be the Euclidean projection of $x$ onto $\mathcal{C}$ and set $v_x:=\frac{x-\hat x}{\|x-\hat x\|}$, then let
\begin{equation}
\label{eq:H-feasibility}
H_x
:=
\{y\in\X:\ \ip{v_x}{y-\hat x}\le 0\}.
\end{equation}
By convexity of $\mathcal{C}$, $\mathcal{C}\subseteq H_x$ and hence $\X^\star\subseteq \mathcal{C}\subseteq H_x$.
Therefore, whether the query point violates optimality or feasibility, one can construct a separating halfspace guaranteed to contain $\X^\star$.

\medskip
\noindent
\textbf{Complexity of cutting plane methods.}
Algorithm~\ref{alg:generic-cutting-plane} is written in a generic form. Representatives include the ellipsoid method, analytic-center cutting-plane method, and volumetric-center cutting-plane method (see \cite{Khachiyan1980,Nesterov2018}).
Different cutting-plane methods share the same separation step and mainly differ in two places: how the query point $x_t$ is chosen in line \ref{line:find-center}, and how the new localization region $\Omega_{t+1}$ is represented after the new cut is added in line \ref{line:update-local-set}. It is often $\Omega_{t+1}=\Omega_t\cap H_t$, while ellipsoid methods update $\Omega_{t+1}$ to a smaller outer ellipsoid of $\Omega_t\cap H_t$. 
As for the center selection (line \ref{line:find-center}),  ellipsoid methods simply choose the center of the ellipsoid as it keeps $\Omega_t$ as ellipsoids. Other cutting-plane methods keep $\Omega_t$ as a polyhedron, so they may spend more efforts in finding the center. For example, the analytic-center cutting-plane method recenters to the analytic center of $\Omega_t$ by solving a log-barrier problem, and the volumetric-center cutting-plane method recenters to the volumetric center of $\Omega_t$ by solving a volumetric barrier problem.

\begin{algorithm}[htbp]
\caption{Generic cutting-plane methods for solving \eqref{eq:main-optimization-problem}}
\label{alg:generic-cutting-plane}
\begin{algorithmic}[1]
\State \textbf{Input:} a bounded localization set $\Omega_0$ satisfying $\X^\star\subseteq \Omega_0$.
\For{$t=0,1,2,\dots$}
    \State \textit{Center selection.} Choose $x_t\in\Omega_t$.\label{line:find-center}
    \State \textbf{If} $\Omega_{t+1}$ is sufficiently small \textbf{return} $x_t$.
   \State \textit{Separation-oracle query.} Obtain a valid cut $H_{x_t}$ via either \eqref{eq:H-optimality} or \eqref{eq:H-feasibility}. \label{line:cutting-plane-separation}
    \State \textit{Update localization set.} Use $H_t$ to update $\Omega_{t+1}$ so that $\mathcal{X}^\star\subseteq\Omega_{t+1} $.   \label{line:update-local-set}
\EndFor 
\end{algorithmic}
\end{algorithm} 
 
For a bounded initial region $\Omega_0$ of diameter $R$ and target accuracy $\epsilon$ measured by the distance to $\X^\star$,  with a restatement of the complexity results, classical ellipsoid-type methods require $O(d^2\log(dR/\epsilon))$ separating-halfspace oracle calls, whereas volumetric-center based methods achieve sharper $O(d\log(dR/\epsilon))$ dependence \cite{Khachiyan1980,Nesterov2018,LeeSidfordWong2015,jiang2020improved}.
In particular, \cite{LeeSidfordWong2015} gives a cutting-plane method using expected $O\!\left(d\log(dR/\epsilon)\right)$
separating-halfspace oracle calls and $O\!\left(d^3\log^{O(1)}(dR/\epsilon)\right)$ additional arithmetic time overall.
Subsequently, \cite{jiang2020improved} refines the additional overhead to essentially $O(d^2)$ time per separating-halfspace oracle call while preserving the $O(d\log(dR/\epsilon))$ separating-halfspace oracle complexity.

\medskip
\noindent
\textbf{Advantages and disadvantages.}
Using cutting-plane methods as a direct solver based on normal direction estimation is straightforward.  
Normal direction estimation becomes a black-box separation halfspace
oracle, and the optimization algorithm layer can directly leverage cutting-plane methods. Furthermore, cutting-plane methods usually exhibit polylogarithmic dependence on the target accuracy, especially suitable for high-accuracy requirements.
However, there are two main drawbacks. 
First, even with the best cutting-plane method, the number of iterations scales linearly with $d$. In our setting, each iteration may require estimating a normal direction from comparisons, so the total comparison cost multiplies. Simply saying, if obtaining one reliable normal direction costs $\widetilde{O}(d)$ comparisons (here $\widetilde{O}$ omits logarithmic factors in accuracy and dimension), then the overall number of comparisons may be on the order of $\widetilde{O}(d^2\, \log(dR/\epsilon))$, again hiding only logarithmic factors in $d$, $R$, and $\epsilon^{-1}$.
Second, cutting-plane methods have to maintain information of the full-dimensional geometry via either polytopes or ellipsoids, so memory occupancy and arithmetic overhead costs scale at least quadratically in $d$ per iteration, and can dominate in moderate-to-large dimensional problems.  

A good algorithm should balance the trade-off between the dimension-dependent costs and the accuracy-dependent costs. Cutting-plane methods are ideal choice for getting high-accuracy solutions but incur high dimension-dependent costs. Once $d$ becomes large, the costs in maintaining the full-dimensional geometry model and solving the centers can be prohibitive. Therefore, we still seek methods that do not maintain a full-dimensional geometry model and have smaller per-iteration overhead. In the next section, we will study two normal direction span-based methods, which are more scalable in high-dimensional settings, although they may exhibit worse dependence on the target accuracy.

\section{Two normal direction span-based methods}
\label{sec:ngd-fixed-radius}

In this section, we develop optimization methods based on the normal direction estimators from Section~\ref{sec:normal-estimation}. Since normal directions point toward less-preferred regions, moving in their negatives is expected to produce more-preferred points. This leads to methods that update within the span of previously estimated normal directions together with the projection corrections induced by the constraint set. We call them normal-direction span-based methods.

\begin{definition}[Normal-direction span-based methods]
\label{def:normal-direction-span-method}
Fix $K\in\mathbb{N}$. A normal-direction span-based method maintains internal points
$z_1,\dots,z_{K+1}\in\X$ and the corresponding feasible points
$x_t:=\Pi_{\mathcal{C}}(z_t)$ for $t=1,\dots,K+1$.
At each iteration $t=1,\dots,K$, the method receives the normal direction at $x_t$
and receives $n_t:=n_{x_t}$.
It then chooses the next internal point so that
\begin{equation}\label{eq:general-update}
z_{t+1}\in z_1+\mathrm{span}\{n_1,\dots,n_t,\ x_1-z_1,\dots,x_t-z_t\}.
\end{equation} 
After $K$ iterations, it may output any $\hat x\in\{x_1,\dots,x_{K+1}\}$.
\end{definition}

We study methods that use normal directions estimated by methods introduced in Section~\ref{sec:normal-estimation}.
Our analysis is entirely function-free: rather than relying on objective values or derivatives, we measure progress using the geometric quantities introduced in Section~\ref{subsec:optimality-stationarity}: the level-set optimality gap $\DeltaLS(\cdot)$ and the regularity radius $r_x$.

We begin in Section~\ref{subsec:ngd-fixed-radius-alg} with a simple fixed-step normal direction descent method (NDD). Its guarantee is two-case: either the returned point has small level-set optimality gap, or it is preferred to an iterate with small regularity radius. Section~\ref{subsec:regularity-vs-distance} then introduces a local growth condition under which the second alternative can be converted into the first. Section~\ref{subsec:pf-ngd-selfcal} introduces an adaptive version, adaNDD, which removes the need to tune the comparison radius and step-size in advance. Finally, Section~\ref{subsec:ngd-exact-lower-bound} shows a lower bound on the worst-case convergence rate of normal-direction span-based methods, which implies that NDD and adaNDD are nearly optimal in the number of normal directions.

\subsection{Normal direction descent method (NDD)}
\label{subsec:ngd-fixed-radius-alg}

We now consider a basic normal direction descent scheme that repeatedly moves in the direction
opposite to an estimated outward normal direction.
At iteration $k$, given a feasible iterate $x_k\in\mathcal{C}$, we estimate the outward normal
direction $n_{x_k}$ using comparisons, and then take a fixed-length step in the negative
estimated normal direction, followed by Euclidean projection onto $\mathcal{C}$:
\begin{equation}
\label{eq:ngd-update}
x_{k+1}\;=\;\Pi_{\mathcal{C}}\!\bigl(x_k-\eta\,\widehat n_k\bigr),
\end{equation}
where $\widehat n_k$ is the output of Algorithm~\ref{alg:OSNE} run at the base point $x_k$ with a
fixed comparison radius $h$ and $T$ planar bisection steps.

When the preference relation is realized by a differentiable function $f$, then
$n_{x_k}=\tfrac{\nabla f(x_k)}{\|\nabla f(x_k)\|}$. This makes \eqref{eq:ngd-update} precisely the normalized gradient descent (using an inexact normalized gradient estimator $\widehat{n}_k$).
Normalized gradient descent methods have been studied in a range settings, including smooth convex and strongly convex objectives as well as smooth nonconvex problems; see, e.g., \cite{nesterov2004introductory,levy2017online,orabona2023normalized,mei2021leveraging}.

However, we do not assume the existence of the objective function, and we focus on function-free optimization using \eqref{eq:ngd-update} and measure progress using the level-set optimality gap and regularity radius introduced in
Section~\ref{sec:co-foundations}.
Algorithm~\ref{alg:fixed-ngd-osne} gives the full procedure. In addition to the parameters $(h,T)$
for Algorithm~\ref{alg:OSNE}, the method takes a stepsize $\eta>0$ and a horizon $K$.
The algorithm maintains a best-so-far point using one additional comparison per iteration.

\begin{algorithm}[htbp]
\caption{Normal direction descent (NDD) method }
\label{alg:fixed-ngd-osne}
\begin{algorithmic}[1]
\State \textbf{Input:} initial point $x_1\in\mathcal{C}$, stepsize $\eta>0$, comparison radius $h>0$,
planar bisection steps $T\in\mathbb{N}_{++}$, horizon $K\in\mathbb{N}$.
\State Initialize $\hat x\gets x_1$.
\For{$k=1,2,\dots,K$}
    \State Compute normal direction estimate $\widehat n_k$ of $x_k$ by Algorithm~\ref{alg:OSNE} with fixed parameters $(h,T)$.
    \State $x_{k+1}\gets \Pi_{\mathcal{C}}\!\bigl(x_k-\eta\,\widehat n_k\bigr)$.
    \State \textbf{If} $x_{k+1}\prec \hat x$ \textbf{then} set $\hat x\gets x_{k+1}$.
\EndFor
\State \textbf{return} $\hat x$.
\end{algorithmic}
\end{algorithm}

If the normal direction estimates $\widehat n_k$ are exact, Algorithm~\ref{alg:fixed-ngd-osne} is a normal-direction span-based method.
Indeed, its internal update is $z_{t+1}=x_t-\eta n_t$, where $x_t=\Pi_{\mathcal C}(z_t)$.
Since $x_t=z_t+(x_t-z_t)$,  
$z_{t+1}\in z_1+\operatorname{span}\{n_1,\dots,n_t,\ x_1-z_1,\dots,x_t-z_t\}$.
The update satisfies the span restriction in Definition~\ref{def:normal-direction-span-method}, and the best-so-far output is one of the visited feasible points.

Before showing the main guarantees of Algorithm~\ref{alg:fixed-ngd-osne}, we present two lemmas that are key to the analysis. The first lemma relates the level-set optimality gap $\DeltaLS(x)$ to a supporting-hyperplane gap at $x$, and the second lemma shows that $\DeltaLS$ is monotone with respect to the preference relation.

\begin{lemma}
\label{lem:DeltaLS-by-gap}
Assume the preference relation $\preccurlyeq$ is convex and plateau-free. 
Fix $x\in\mathcal{C}$ with $x\notin\X^\star$, and let $n_x$ be the outward unit normal at $x\in\partial \calS_x$. Then for any $x^\star\in\X^\star$,
\begin{equation}
\label{eq:DeltaLS-by-gap-lemma}
\DeltaLS(x)\ \le\ \langle n_x,\,x-x^\star\rangle \ .
\end{equation} 
\end{lemma}

\begin{proof}
By convexity and plateau-freeness of the preference relation, $\calS_x$ is convex and closed.
Since $\calS_x$ is convex and admits an outward unit normal $n_x$ at $x\in\partial \calS_x$, it is contained in
the supporting halfspace
\begin{equation}
\label{eq:support-halfspace-at-x-lemma}
\calS_x\ \subseteq\ \{y\in\X:\ \langle n_x,\,y-x\rangle\le 0\}.
\end{equation}
Fix $x^\star\in\X^\star$. By definition of $\X^\star$, we have $x^\star\preccurlyeq x$, hence $x^\star\in \calS_x$.
Let $s_0:=\sup\{s\ge 0:\ x^\star+s\,n_x\in \calS_x\}$ and set $v:=x^\star+s_0\,n_x$. Then $v\in\partial \calS_x=\calL_x$.
Due to \eqref{eq:support-halfspace-at-x-lemma}, we have 
$\langle n_x,\, v-x \rangle = \langle n_x,\, x^\star+s_0 n_x - x\rangle = \langle n_x,\, x^\star - x\rangle + s_0 \le 0$, which implies
$s_0\le \langle n_x,\,x-x^\star\rangle$.
Therefore,
$$
\DeltaLS(x)=\dist(\X^\star,\calL_x)\le \dist(x^\star,\calL_x)\le \|v-x^\star\|=s_0\le \langle n_x,\,x-x^\star\rangle \ .
$$This proves \eqref{eq:DeltaLS-by-gap-lemma}.
\end{proof}

\begin{lemma}
\label{lem:DeltaLS-monotone-best}
Assume the preference relation $\preccurlyeq$ is plateau-free.
If $y\preccurlyeq x$, then
\begin{equation}
\label{eq:DeltaLS-monotone-lemma}
\DeltaLS(y)\ \le\ \DeltaLS(x).
\end{equation} 
\end{lemma}

\begin{proof}
    If any of $x$ and $y$ is optimal, then the inequality holds trivially. We therefore assume $x,y\notin\X^\star$.
If $y\preccurlyeq x$ then $\calS_y\subseteq \calS_x$, hence
$\X\setminus \calS_x\subseteq \X\setminus \calS_y$.
Since $\X^\star\subseteq \calS_y$ and plateau-freeness implies $\partial\calS_z=\calL_z$ and that
$\calS_z$ is closed for all $z$, we have
$\dist(\X^\star,\calL_z)=\dist(\X^\star,\X\setminus\calS_z)$ for $z\in\{x,y\}$.
Thus, enlarging the complement cannot increase its distance to $\X^\star$, which yields
\eqref{eq:DeltaLS-monotone-lemma}. 
\end{proof}

Now we are ready to state the main guarantee for Algorithm~\ref{alg:fixed-ngd-osne}. 
Theorem \ref{thm:ngd-fixed-radius-main} below shows that the method either returns a point with small level-set optimality gap, or reaches an iterate with sufficiently small regularity radius.  

\begin{theorem}[Normal direction descent method]
\label{thm:ngd-fixed-radius-main} 
Assume the preference relation $\preccurlyeq$ is convex and plateau-free.
Assume $\mathcal{C}\subseteq\X$ is nonempty, closed, and convex, and $d\ge 2$.
Assume furthermore that the preference relation is $r_x$-regular at every $x\in\mathcal{C}\setminus\X^\star$, with regularity radius $r_x>0$.
Assume $x_1\notin\X^\star$ and write $D_1:=\dist(x_1,\X^\star)$.
Run Algorithm~\ref{alg:fixed-ngd-osne} for $K\in\mathbb N_{++}$ iterations with stepsize $\eta = D_1/\sqrt K$, comparison radius $h$, and $T$ planar bisection steps, starting from $x_1\in\mathcal{C}$. 
Then Algorithm~\ref{alg:fixed-ngd-osne} uses at most
\[
N\ \le\ K\Big((d-1)(T+3)+2\Big)
\]
pairwise comparisons.
Let $\widetilde D:=\min\left\{D_1+\frac{\eta}{2}(K-1),\ \operatorname{diam}(\mathcal{C})\right\}$.
Then at least one of the following holds:
\begin{enumerate}[nosep,label=(\roman*)]
\item
\begin{equation}
\label{eq:ngd-fixed-radius-bound}
\DeltaLS(\hat x)
\ \le \ \hat{\epsilon} \ := \  
\frac{2D_1}{\sqrt K}
\ +\ \frac{\pi\sqrt{d-1}\,\widetilde D}{2^{T-1}} \ ;
\end{equation}
\item there exists $k\in\{1,\dots,K\}$ such that
\begin{equation}
\label{eq:ngd-fixed-radius-small-r}
r_{x_k}
\ <\ \frac{4\sqrt{d-1}\,\widetilde D\,h}{\hat{\epsilon}} \ .
\end{equation}
In this case, $\hat{x}\preccurlyeq x_k$.
\end{enumerate}
\end{theorem}

\begin{proof}
If there exists some $k\in\{1,\dots,K\}$ such that $x_k\in\X^\star$, then the best-so-far rule implies $\hat x\in\X^\star$, and hence alternative~(i) holds trivially.
We therefore assume $x_k\notin\X^\star$ for all $k\in\{1,\dots,K\}$.

Fix a run of Algorithm~\ref{alg:fixed-ngd-osne}. For each $k$, write
$D_k:=\dist(x_k,\X^\star)$, choose
$x_k^\star\in\arg\min_{z\in\X^\star}\|x_k-z\|$, and let $n_k:=n_{x_k}$.
Since $\hat x\preccurlyeq x_k$ for every $k$ by the best-so-far update, Lemmas~\ref{lem:DeltaLS-monotone-best} and \ref{lem:DeltaLS-by-gap} imply
\begin{equation}
\label{eq:ngd-main-gap}
\DeltaLS(\hat x)\ \le\ \DeltaLS(x_k)\ \le\ \langle n_k,\,x_k-x_k^\star\rangle,
\qquad k=1,\dots,K.
\end{equation}

Using the update \eqref{eq:ngd-update} with $\eta=D_1/\sqrt K$, the nonexpansiveness of Euclidean projection, and the fact that $x_k^\star\in\X^\star\subseteq\mathcal{C}$, we obtain
\begin{equation}
\label{eq:ngd-main-recursion-0}
\begin{aligned}
D_{k+1}^2
&=\dist(x_{k+1},\X^\star)^2
\le \|x_{k+1}-x_k^\star\|^2
\le \left\|x_k-\frac{D_1}{\sqrt K}\widehat n_k-x_k^\star\right\|^2 \\
&= D_k^2-\frac{2D_1}{\sqrt K}\,\langle \widehat n_k,\,x_k-x_k^\star\rangle+\frac{D_1^2}{K},
\end{aligned}
\end{equation}
Rearranging and using $\|x_k-x_k^\star\|=D_k$, we obtain
\begin{equation} \label{eq:ngd-main-recursion}
\langle n_k,\,x_k-x_k^\star\rangle
\le
\frac{\sqrt K}{2D_1}\left(D_k^2-D_{k+1}^2+\frac{D_1^2}{K}\right)
+D_k\,\|\widehat n_k-n_k\| .
\end{equation}
Summing \eqref{eq:ngd-main-recursion} over $k=1,\dots,K$, telescoping the $D_k^2$ terms, and using $D_{K+1}^2\ge 0$, we get
\begin{equation}
\label{eq:ngd-main-sum}
\sum_{k=1}^K \langle n_k,\,x_k-x_k^\star\rangle
\ \le\
D_1\sqrt K\ +\ \sum_{k=1}^K D_k\,\|\widehat n_k-n_k\| .
\end{equation}
Averaging \eqref{eq:ngd-main-gap} over $k=1,\dots,K$ and combining the result with \eqref{eq:ngd-main-sum} yields
\begin{equation}
\label{eq:ngd-main-master}
\DeltaLS(\hat x)
\le
\frac{D_1}{\sqrt K}
+\frac{1}{K}\sum_{k=1}^K D_k\,\|\widehat n_k-n_k\| .
\end{equation}

Next we bound the error term on the right-hand side of \eqref{eq:ngd-main-master}.
Since $\|x_{k+1}-x_k\|\le \eta=D_1/\sqrt K$, the triangle inequality gives $D_{k+1}\le D_k+\eta$, hence $D_k\le D_1+(k-1)\eta$.
Also, $D_k\le \operatorname{diam}(\mathcal{C})$ because $x_k,x_k^\star\in\mathcal{C}$.
Therefore
\begin{equation}
\label{eq:ngd-main-Davg}
\frac{1}{K}\sum_{k=1}^K D_k
\le
\min\left\{D_1+\frac{\eta}{2}(K-1),\ \operatorname{diam}(\mathcal{C})\right\}
=
\widetilde D .
\end{equation}

Assume first that \eqref{eq:ngd-fixed-radius-small-r} fails for all $k\in\{1,\dots,K\}$.
Then Theorem~\ref{thm:OSNE} implies
\begin{equation}
\label{eq:ngd-main-error-0} 
\|\widehat n_k-n_k\|\le
2\sqrt{d-1}\,\frac{h}{r_{x_k}}+\frac{\pi\sqrt{d-1}}{2^T}
\le
\frac{\hat\epsilon}{2\widetilde D}+\frac{\pi\sqrt{d-1}}{2^T},
\end{equation}
 Multiplying \eqref{eq:ngd-main-error-0}  by $D_k$, averaging over $k$, and using \eqref{eq:ngd-main-Davg}, we have
\begin{equation}
\label{eq:ngd-main-error}
\frac{1}{K}\sum_{k=1}^K D_k\,\|\widehat n_k-n_k\|
\le
\frac{\hat\epsilon}{2}+\frac{\pi\sqrt{d-1}\,\widetilde D}{2^T} \ .
\end{equation} 
Substituting \eqref{eq:ngd-main-error} into \eqref{eq:ngd-main-master} and using the definition of $\hat\epsilon$ in \eqref{eq:ngd-fixed-radius-bound}, we obtain
\begin{equation}
\label{eq:ngd-main-final}
\DeltaLS(\hat x)
\le
\frac{D_1}{\sqrt K}
+\frac{\hat\epsilon}{2}
+\frac{\pi\sqrt{d-1}\,\widetilde D}{2^T}
=
\hat\epsilon .
\end{equation}
Thus alternative~(i) follows from \eqref{eq:ngd-main-final}.

If instead \eqref{eq:ngd-fixed-radius-small-r} holds for some $k\in\{1,\dots,K\}$, then $\hat x\preccurlyeq x_k$ by the best-so-far rule, so alternative~(ii) holds.

Finally, Theorem~\ref{thm:OSNE} shows that each call to Algorithm~\ref{alg:OSNE} uses at most $(d-1)(T+3)+1$ pairwise comparisons, and each iteration uses one additional comparison to update the best-so-far point. Therefore the claimed comparison bound follows.
\end{proof}

\begin{remark}[Comparison complexity]
\label{rem:ngd-fixed-radius-complexity}
Fix $\epsilon\in(0,D_1)$. If one only asks that the bound in alternative~(i) of Theorem~\ref{thm:ngd-fixed-radius-main} be at most $\epsilon$,  one may take $K=\lceil 16D_1^2/\epsilon^2\rceil$ and $T=\lceil \log_2(4\pi\sqrt d\,\widetilde D/\epsilon)\rceil$. For this choice of $K$, we also have  $\widetilde D\le   D_1+2D_1^2/\epsilon$. Then the whole algorithm uses at most $\widetilde{O}\bigl(\tfrac{dD_1^2}{\epsilon^{2}}\bigr)$ comparisons, where $\widetilde O$ hides only logarithmic factors of $d$, $D_1$ and $\epsilon^{-1}$.  By contrast, making the threshold in alternative~(ii) small only requires taking $h$ sufficiently small, and this does not change the complexity bound.
\end{remark}

Here we focus only on the case $d\ge 2$ because when $d=1$, a cutting-plane method as described in Section~\ref{subsec:cutting-plane} is already efficient. Theorem~\ref{thm:ngd-fixed-radius-main} gives a two-case guarantee, but the two alternatives play different roles. The first alternative is the main near-optimality guarantee, stated directly in terms of the level-set optimality gap $\DeltaLS(\hat x)$. The second alternative shows that the output $\hat x$ is preferred to an iterate with small regularity radius. 
As discussed in Section~\ref{sec:co-foundations} (and later in the next subsection), small regularity radius is itself a meaningful geometric certificate and is often associated with near-stationarity. 
It is not the part that drives the comparison complexity, because its threshold is proportional to the comparison radius $h$, so it can be made arbitrarily small by taking $h$ sufficiently small. In this sense, the dominant part is alternative~(i).

\begin{remark}[Normal direction complexity]
\label{rem:ngd-fixed-radius-complexity-oracle}
If the normal direction $\hat{n}_k$ used in the algorithm is equal to the real normal direction $n_k$, \eqref{eq:ngd-main-master} implies that $\DeltaLS(\hat{x})\le \frac{D_1}{\sqrt{K}}$. This guarantee remains only if the preference is plateau-free and convex, and does not require any regularity condition.  
\end{remark}

When the preference relation is realized by a differentiable function $f$, we have $n_{x_k}=\tfrac{\nabla f(x_k)}{\|\nabla f(x_k)\|}$, so Algorithm~\ref{alg:fixed-ngd-osne} is precisely projected normalized gradient descent with an inexact normalized gradient oracle. 
If assuming $\hat{n}_k$ is equal to $n_k$, our proof up until \eqref{eq:ngd-main-master} follows the classic proof of normalized subgradient method by \cite{shor1985minimization}. 
The new ingredient here is that the normal direction is obtained only through comparisons, so the analysis additionally controls the accumulated error.  See also \cite{orabona2023normalized,hazan2015beyond} for analyses of normalized gradient methods.

Furthermore, because the normalized gradient descent method also works for nonconvex optimization problems and NDD has the same update rule, we believe that NDD may also apply to preference relations without convexity assumption. We do not yet have a function-free analysis but the existing guarantees for normalized gradient descent methods on the underlying objective function may still apply to NDD; see, e.g., \cite{levy2016power,murray2019revisiting}.

\subsection{A local growth condition of regularity radius}\label{subsec:regularity-vs-distance}

The second alternative in Theorem~\ref{thm:ngd-fixed-radius-main} is geometric.
To interpret it, we now discuss a relation that arises in many common preference models: a small regularity radius often forces the current level-set optimality gap  to be small.
A convenient way to summarize this relation is the following local growth condition.
 
\begin{condition}[Local growth condition of regularity radius]
\label{cond:radius-growth}
For the problem \eqref{eq:main-optimization-problem}, there exist constants
$\gamma_1>0$ and $\gamma_2\in(0,+\infty]$ such that for every
$x\in\mathcal C\setminus\X^\star$,
\begin{equation}
\label{eq:radius-growth}
r_x \ge \min\left\{\gamma_1 \cdot \DeltaLS(x),\ \gamma_2\right\}.
\end{equation}
\end{condition}

In words, as $x$ approaches $\X^\star$ in terms of $\DeltaLS(x)$, the regularity radius $r_x$ can shrink at most linearly in $\DeltaLS(x)$, and once $x$ is sufficiently far from $\X^\star$, the regularity radius is uniformly lower bounded by $\gamma_2$.
Since $\DeltaLS(x)\le \dist(x,\X^\star)$ for every $x\in\X$, any stronger distance-based condition of the form $r_x \ge \min\{\gamma_1\dist(x,\X^\star),\,\gamma_2\}$ automatically implies Condition~\ref{cond:radius-growth}.
This condition summarizes a relation that holds for many common preference relations in optimization.
Below we present several examples of functions that are often seen as objective functions of optimization problems and realize preference relations satisfying Condition~\ref{cond:radius-growth}.
These examples are stated for the unconstrained case $\mathcal C=\X$ and the constrained case $\mathcal C\subset\X$ will be discussed later.
  
\begin{example}[Convex quadratic functions]
\label{ex:curv-quadratic}
Let $f(x):=\frac12(x-\hat{x})^\top Q(x-\hat{x})$ be a convex quadratic function,  for symmetric matrix $Q\in\R^{d\times d}$ and vector $\hat{x}\in\R^d$. Then the preference relation realized by $f$ satisfies Condition~\ref{cond:radius-growth} with
$(\gamma_1,\gamma_2) = \big(\sqrt{\lambda_{\min}^+(Q)/\lambda_{\max}^+(Q)},\,+\infty\big)$, where
$\lambda_{\max}^+(Q)$ and $\lambda_{\min}^+(Q)$ are the largest and smallest positive eigenvalues of $Q$,
respectively.
\end{example}

\begin{proof}
Let $\mathcal N:=\ker(Q)$. Then $\X^\star=\hat x+\mathcal N$. Fix any $x\notin \X^\star$ and write $c:=f(x)>0$. For any $y\in \calL_x$, write $y-\hat x=v+w$ with $v\in\mathcal N^\perp$ and $w\in\mathcal N$. Then $c=f(y)=\frac12 v^\top Qv\le \frac{\lambda_{\max}^+(Q)}{2}\|v\|_2^2=\frac{\lambda_{\max}^+(Q)}{2}\dist(y,\X^\star)^2$, so $\dist(y,\X^\star)\ge \sqrt{2c/\lambda_{\max}^+(Q)}$. Equality is attained by taking $y-\hat x$ along an eigenvector corresponding to $\lambda_{\max}^+(Q)$, and hence $\DeltaLS(x)=\sqrt{2c/\lambda_{\max}^+(Q)}$. On the other hand, writing $x-\hat x=v+w$ with $v\in\mathcal N^\perp$ and $w\in\mathcal N$, we have $\nabla f(x)=Qv$ and therefore $\|\nabla f(x)\|_2^2=v^\top Q^2v\ge \lambda_{\min}^+(Q)\,v^\top Qv=2\lambda_{\min}^+(Q)c=\lambda_{\min}^+(Q)\lambda_{\max}^+(Q)\DeltaLS(x)^2$. Since $f$ is $\lambda_{\max}^+(Q)$-smooth, Lemma~\ref{lem:regularityradius-hess-grad} gives $r_x\ge \|\nabla f(x)\|_2/\lambda_{\max}^+(Q)\ge \sqrt{\lambda_{\min}^+(Q)/\lambda_{\max}^+(Q)}\,\DeltaLS(x)$. This proves the claim.
\end{proof}

\begin{example}[Distance to a closed convex set]
\label{ex:curv-dist-to-set}
Let $C_0\subset\X$ be nonempty, closed, and convex, and let the preference relation be realized by
$f(x):=\dist(x,C_0)$.
Then Condition~\ref{cond:radius-growth} holds with $(\gamma_1,\gamma_2)=(1,+\infty)$.
\end{example}

\begin{proof}
Fix $x\notin C_0$ and let $\hat{x}:=\Pi_{C_0}(x)$ be the Euclidean projection. Let
$t:=\|x-\hat{x}\|=\dist(x,C_0)=\dist(x,\X^\star)>0$.
The sublevel set $\calS_x$ is $\{u+tv:u\in C_0, v\in B(0,1)\}$.
Since $\hat{x}\in C_0$, we have $B(\hat{x},t)\subseteq \calS_x$ and $x\in\partial B(\hat{x},t)$, so
$B(\hat{x},t)$ is an interior tangent ball to $\calS_x$ at $x$.
Moreover, since $\calS_x$ is convex, there exists an exterior tangent ball of radius $t$ at $x$.
Therefore, $r_x\ge t$.  
Since  $\DeltaLS(x)=\dist(x,C_0)=t$, Condition~\ref{cond:radius-growth} holds with $(\gamma_1,\gamma_2)=(1,+\infty)$.
\end{proof}

\begin{example}[Convex $L$-smooth with local strong convexity]
\label{ex:curv-nondegenerate}
Assume the preference relation is realized by a convex differentiable function $f:\X\to\R$ such that
\begin{enumerate}[nosep]
\item (\emph{$L$-smoothness}) $\|\nabla f(x)-\nabla f(y)\| \le L\|x-y\|$ for all $x,y\in\X$;
\item (\emph{local strong convexity}) there exist a minimizer $x^\star\in\X^\star$ and constants $\mu>0$ and $\rho>0$ such that $f$ is
$\mu$-strongly convex on the ball $B(x^\star,\rho)$, i.e., for all $x,y\in B(x^\star,\rho)$,
$f(y) \ge f(x)+\langle \nabla f(x),y-x\rangle+\frac{\mu}{2}\|y-x\|^2.$
\end{enumerate}
Then Condition~\ref{cond:radius-growth} holds with
$(\gamma_1,\gamma_2)=\left(\tfrac{\mu}{L},\,\frac{\mu\rho}{L}\right)$. Here $\rho$ could be $+\infty$ and then $\gamma_2=+\infty$.
\end{example}

\begin{proof}
For $x\notin \X^\star$, we give a lower bound on $\|\nabla f(x)\|$.
If $\|x-x^\star\|\le \rho$, then strong convexity implies
$\langle \nabla f(x),x-x^\star\rangle\ge \mu\|x-x^\star\|^2$ and hence
$\|\nabla f(x)\|\ge \mu\|x-x^\star\|=\mu\,\dist(x,\X^\star)$.

If $\|x-x^\star\|\ge \rho$, write $u:=(x-x^\star)/\|x-x^\star\|$ and define the one-dimensional
restriction $g(t):=f(x^\star+t u)$.
Convexity of $f$ implies $g$ is convex, hence $g'$ is nondecreasing.
Let $x_\rho:=x^\star+\rho u\in B(x^\star,\rho)$. Local strong convexity with $(x,y)=(x_\rho,x^\star)$
yields $\langle \nabla f(x_\rho),x_\rho-x^\star\rangle\ge \mu\rho^2$, i.e., $g'(\rho)\ge \mu\rho$.
Therefore $g'(t)\ge g'(\rho)\ge \mu\rho$ for all $t\ge \rho$, and in particular
$\|\nabla f(x)\|\ge \langle \nabla f(x),u\rangle=g'(\|x-x^\star\|)\ge \mu\rho$.

Applying Lemma~\ref{lem:regularityradius-hess-grad} in both cases gives $r_x \ge \frac{\|\nabla f(x)\|}{L}
\ge \min\big\{\tfrac{\mu\,\dist(x,\X^\star)}{L},\ \tfrac{\mu\rho}{L}\big\}
\ge
\min\big\{\tfrac{\mu\,\DeltaLS(x)}{L},\ \tfrac{\mu\rho}{L}\big\}$.
This completes the proof.
\end{proof}

\begin{example}[$L$-smooth with PL inequality and quadratic growth]
\label{ex:curv-pl-qg}
Assume the preference relation is realized by a differentiable function $f:\X\to\R$.
Let $f^\star:=\min_{z\in\X} f(z)$ be the optimal value.
Assume further that $f$ satisfies:
\begin{enumerate}[nosep]
\item (\emph{$L$-smoothness}) $\|\nabla f(x)-\nabla f(y)\| \le L\|x-y\|$ for all $x,y\in\X$;
\item (\emph{PL inequality}) there exists $\mu_{\mathrm{PL}}>0$ such that for all $x\in\X$,
$\frac12\|\nabla f(x)\|^2\ \ge\ \mu_{\mathrm{PL}}\bigl(f(x)-f^\star\bigr)$.
\item (\emph{Quadratic growth}) there exists $\mu_{\mathrm{QG}}>0$ such that for all $x\in\X$,
$f(x)-f^\star\ \ge\ \frac{\mu_{\mathrm{QG}}}{2}\dist(x,\X^\star)^2$.
\end{enumerate}
This function does not have to be convex. Then Condition~\ref{cond:radius-growth} holds with
$(\gamma_1,\gamma_2) = \left(\frac{\sqrt{\mu_{\mathrm{PL}}\mu_{\mathrm{QG}}}}{L},\,+\infty\right)$.
\end{example}

\begin{proof}
Combining PL and quadratic growth gives
$\|\nabla f(x)\|^2 \ge \mu_{\mathrm{PL}}\mu_{\mathrm{QG}}\dist(x,\X^\star)^2$, hence
$\|\nabla f(x)\| \ge \sqrt{\mu_{\mathrm{PL}}\mu_{\mathrm{QG}}}\,\dist(x,\X^\star)$.
Lemma~\ref{lem:regularityradius-hess-grad} then yields
$r_x \ge \|\nabla f(x)\|/L
\ge \sqrt{\mu_{\mathrm{PL}}\mu_{\mathrm{QG}}}\,\dist(x,\X^\star)/L
\ge \bigl(\sqrt{\mu_{\mathrm{PL}}\mu_{\mathrm{QG}}}/L\bigr)\cdot \DeltaLS(x)$.
\end{proof}

Nelder--Mead method is a classic derivative-free optimization method that relies solely on function comparisons. It is widely used in many real-world applications and for a while people believed for convex optimization it always converged to the optimal solution. However, McKinnon \cite{mckinnon1998convergence} showed that the method may converge to a nonstationary point and constructs a function example. We show that the induced preference relation of the example of Figure 2.2 in \cite{mckinnon1998convergence} satisfies Condition~\ref{cond:radius-growth}.

\begin{example}[McKinnon's Nelder--Mead example]
\label{ex:curv-mckinnon}
Consider the example function of   \cite{mckinnon1998convergence}: \(f(x,y)=360x^2+y+y^2\) for \(x\le 0\), and \(f(x,y)=6x^2+y+y^2\) for \(x\ge 0\).
Then the preference relation realized by \(f\) satisfies Condition~\ref{cond:radius-growth} with \((\gamma_1,\gamma_2)=(1/360,+\infty)\).
\end{example}

\begin{proof}
The unique minimizer is\((0,-1/2)\).
Also, \(f\) is the sum of the globally \(2\)-strongly convex function \(6x^2+y+y^2\) and the convex function \(354(\max\{-x,0\})^2\), so \(f\) is globally \(2\)-strongly convex.
Moreover, \(\nabla f(x,y)=(720x,1+2y)\) for \(x\le 0\) and \(\nabla f(x,y)=(12x,1+2y)\) for \(x\ge 0\), so \(f\) is globally \(720\)-smooth.
Therefore Example~\ref{ex:curv-nondegenerate} applies with \(L=720\), \(\mu=2\), and \(\rho=+\infty\), which gives \((\gamma_1,\gamma_2)=(\mu/L,+\infty)=(1/360,+\infty)\).
\end{proof} 
 \noindent
There are other function examples in Section 3 of \cite{mckinnon1998convergence}. Among other examples that have smooth function level sets, it is still true that the regularity radius approaches $0$ only when a point approaches the optimal solution.

The above examples show that Condition~\ref{cond:radius-growth} holds for a wide range of preference relations in the unconstrained case $\mathcal C=\X$. They also apply to constrained problems $\mathcal C\subset\X$ whenever the optimal solutions coincide with the optimal solutions of the corresponding unconstrained problem, because the regularity radius and the level-set optimality gap  are not changing if enforcing the constraint does not change the optimal solution set.

The next example covers directly a constrained optimization situation.

\begin{example}[Bounded feasible set with nonvanishing gradient]
\label{ex:curv-bounded-feasible-nonvanishing-gradient}
Consider the problem \eqref{eq:main-optimization-problem}. 
Assume that the preference relation is realized by an $L$-smooth differentiable function
$f:\X\to\R$ with $L>0$ and suppose that the gradient of
$f$ is bounded away from zero on  $\mathcal{C}$, i.e.,
$
g_{\mathcal C}:=\inf_{x\in\mathcal C}\|\nabla f(x)\|>0.
$
Then Condition~\ref{cond:radius-growth} holds for any $\gamma_1>0$ and $\gamma_2=\frac{g_{\mathcal C}}{L}.$
\end{example}

\begin{proof}
Fix any $x\in\mathcal C\setminus\X^\star$. 
By Lemma~\ref{lem:regularityradius-hess-grad},
$
r_x\ge \frac{\|\nabla f(x)\|}{L}\ge \frac{g_{\mathcal C}}{L}.
$,
which proves Condition~\ref{cond:radius-growth} with any $\gamma_1 >0$ and $\gamma_2 =  
\frac{g_{\mathcal C}}{L} $.
\end{proof}

This example is especially relevant when the optimal solution is attained on the boundary of $\mathcal{C}$. In such cases, the gradient of a smooth realization need not vanish at the optimum. If the gradient norm is uniformly bounded away from zero on $\mathcal C$, then the regularity radius is uniformly positive there, and Condition~\ref{cond:radius-growth} follows immediately.

It should also be noted that a given preference relation can be realized by multiple functions that are monotone reformulations of each other (Lemma \ref{lem:phi-global}), and a preference relation satisfying Condition~\ref{cond:radius-growth} need not be realized by any of the classes above.

We now return to Theorem~\ref{thm:ngd-fixed-radius-main}.
Under Condition~\ref{cond:radius-growth}, the small regularity radius alternative in Theorem~\ref{thm:ngd-fixed-radius-main} can be converted into a bound on the level-set optimality gap $\DeltaLS(x)$.
This yields the following corollary, which recovers a pure $\epsilon$-level-set guarantee for a concrete choice of parameters.

\begin{corollary}
\label{cor:ngd-fixed-radius-eps-short}
Under the assumptions of Theorem~\ref{thm:ngd-fixed-radius-main}, suppose that Condition~\ref{cond:radius-growth} holds with parameters $(\gamma_1,\gamma_2)$.
Fix $\epsilon\in(0,D_1)$, where $D_1:=\dist(x_1,\X^\star)$.
Choose
\[
K\,:=\,\left\lceil \frac{6D_1^2}{\epsilon^2}\right\rceil,
\quad
T\,:=\,\left\lceil \log_2\!\Big(14\pi\sqrt{d}\Big(1+\frac{D_1}{\epsilon}\Big)^2\Big)\right\rceil,
\quad
0<h\,\le\,\frac{\min\{\epsilon\gamma_1,\gamma_2\}}{7\sqrt d\bigl(1+\frac{D_1}{\epsilon}\bigr)^2}.
\]
Then Algorithm~\ref{alg:fixed-ngd-osne} returns $\hat x$ such that $\DeltaLS(\hat x)\le \epsilon$.
Moreover, the total number of comparisons satisfies
\begin{equation}
\label{eq:cor-eps-N}
N
\ \le\
7d\Big(1+\frac{D_1^2}{\epsilon^2}\Big)\,
\log_2\!\Big(224\pi\cdot d\big(1+\tfrac{D_1}{\epsilon}\big)^2\Big).
\end{equation}
\end{corollary}
 
\begin{proof}
Let $z:=D_1/\epsilon>1$ and $\eta=D_1/\sqrt K$.
Since $K=\lceil 6z^2\rceil$, we have $K-1<6z^2$, $\sqrt K\ge \sqrt6\,z$, and $K<7z^2$.
Consequently,
\begin{equation}
\label{eq:cor-fixed-Dtilde}
\widetilde D
\le
D_1+\frac{\eta}{2}(K-1)
<
\epsilon\left(z+\frac{\sqrt6}{2}z^2\right)
\le
\frac{\sqrt6}{2}\,\epsilon(1+z)^2.
\end{equation}
Also, by the choice of $T$, we have $2^{T-1}\ge 7\pi\sqrt d(1+z)^2$.
Combining this with \eqref{eq:cor-fixed-Dtilde} and $\sqrt{d-1}\le \sqrt d$ yields
\begin{equation}
\label{eq:cor-fixed-epshat}
\hat\epsilon
=
\frac{2D_1}{\sqrt K}
+
\frac{\pi\sqrt{d-1}\,\widetilde D}{2^{T-1}}
\le
\frac{2}{\sqrt6}\epsilon+\frac{\sqrt6}{14}\epsilon
< \epsilon.
\end{equation}
Moreover, since $\hat\epsilon\ge 2D_1/\sqrt K$, \eqref{eq:cor-fixed-Dtilde} and $K<7z^2$ imply $\frac{4\sqrt{d-1}\,\widetilde D\,h}{\hat\epsilon}
\le
\frac{2\sqrt d\,\widetilde D\,h\sqrt K}{D_1}
<
\sqrt{42}\sqrt d(1+z)^2h.$
Using the choice of $h$ and the fact that $\sqrt{42}<7$, we obtain
\begin{equation}
\label{eq:cor-fixed-radius}
\frac{4\sqrt{d-1}\,\widetilde D\,h}{\hat\epsilon}
<
\min\{\epsilon\gamma_1,\gamma_2\}.
\end{equation}

We now apply Theorem~\ref{thm:ngd-fixed-radius-main}.
If alternative~(i) holds, then \eqref{eq:cor-fixed-epshat} gives $\DeltaLS(\hat x)\le \epsilon$.
Otherwise, alternative~(ii) and \eqref{eq:cor-fixed-radius} yield some $k\in\{1,\dots,K\}$ such that $\hat x\preccurlyeq x_k$ and
$r_{x_k}<\min\{\epsilon\gamma_1,\gamma_2\}$. Since $x_k\in\mathcal C\setminus\X^\star$, Condition~\ref{cond:radius-growth} then implies $\DeltaLS(x_k)<\epsilon$.
By Lemma~\ref{lem:DeltaLS-monotone-best}, $\DeltaLS(\hat x)\le \DeltaLS(x_k)<\epsilon$.
Thus $\DeltaLS(\hat x)\le \epsilon$ in either case.

Finally, Theorem~\ref{thm:ngd-fixed-radius-main} gives
$N\le K\big((d-1)(T+3)+2\big)\le Kd(T+3)$.
Since $K\le 7(1+D_1^2/\epsilon^2)$ and $T+3
\le
\log_2\!(14\pi\sqrt d(1+\frac{D_1}{\epsilon})^2)+4
\le
\log_2\!(224\pi d(1+\frac{D_1}{\epsilon})^2)$,
we obtain \eqref{eq:cor-eps-N}.
\end{proof}

Corollary~\ref{cor:ngd-fixed-radius-eps-short} shows that, under the local growth condition of the regularity radius, the two-case guarantee in Theorem~\ref{thm:ngd-fixed-radius-main} can be converted into a pure level-set guarantee.
In particular, to obtain $\DeltaLS(\hat x)\le \epsilon$, it suffices to take 
$\widetilde O(dD_1^2/\epsilon^2)$ comparisons, where $\widetilde O(\cdot)$ hides only logarithmic factors of $d$, $D_1$, and $\epsilon^{-1}$.

\subsection{Adaptive normal direction descent method (adaNDD)}
\label{subsec:pf-ngd-selfcal}

Although we have already provided the NDD method in
Section~\ref{subsec:ngd-fixed-radius-alg}, it requires tuning the stepsize $\eta$ and the
comparison radius $h$. In particular, the parameter choices must depend on problem-related
quantities such as $\dist(x_1,\X^\star)$ and the local growth parameters $(\gamma_1,\gamma_2)$ from
Condition~\ref{cond:radius-growth}.
In this subsection, we present an adaptive alternative that is called adaNDD. It avoids explicit tuning by
combining: (i) Algorithm~\ref{alg:OSNE} implemented with the adaptive normal direction estimation routine from
Section~\ref{subsec:convex-lineseach}, and (ii) an adaptive normal direction descent method based on coin-betting.

The method maintains an internal sequence $\{z_k\}_{k\ge 1}\subseteq \X $ and the corresponding feasible projected iterates $x_k := \Pi_{\mathcal{C}}(z_k)$.
At each iteration $k$, we call Algorithm~\ref{alg:OSNE} with adaptive comparison radius at the feasible point $x_k$ and request a target normal direction accuracy $\epsilon_k$.
To achieve this accuracy, we set a particular planar bisection depth $T_k$ so that the returned direction $\widehat n_k$ satisfies
$\|\widehat n_k-n_{x_k}\|\le \epsilon_k$.
 
The number of comparisons used by the adaptive normal direction estimator can vary across iterations. To control the overall comparison complexity, we use a budgeted implementation in which the $k$-th call to Algorithm~\ref{alg:OSNE} with line search is allowed at most 
\begin{equation}
\label{eq:pf-budgeted-per-iter}
B_k
:=
2(d-1)\left\lceil \log_2\!\Big(\frac{7\sqrt{d-1}}{\epsilon_k}\Big)\right\rceil
+
2\left\lceil \log_2\!\Big(\frac{ h_0 d^3}{r^\star\, \epsilon_k}\Big)\right\rceil_{+}
+
16
+
4\left\lceil \log_2\!\Big(\frac{\pi^2 k^2}{6\delta}\Big)\right\rceil 
\end{equation}
comparisons, for target regularity radius $r^\star>0$ (as an alternative optimality criterion rather than the level-set optimality gap) and initial radius $h_0>0$.
If the budget is exhausted before the normal-estimation call terminates, the method stops
immediately and returns the current best-so-far point.

Overall, the adaNDD method is presented in Algorithm~\ref{alg:pf-ngd-osne}.
It does not require advance knowledge of $\dist(x_1,\X^\star)$ or   Condition~\ref{cond:radius-growth} parameters to choose a stepsize or a fixed comparison radius.
The user specifies only the initial comparison radius $h_0$, the target regularity radius $r^\star$, the confidence level $\delta$, and the maximal iteration budget $K$.

\begin{algorithm}[t]
\caption{Adaptive normal direction descent method (adaNDD)}
\label{alg:pf-ngd-osne}
\begin{algorithmic}[1] 
\State \textbf{Input:} initial point $x_1\in\mathcal{C}$, initial comparison radius $h_0>0$,
target regularity radius $r^\star>0$,  confidence level $\delta\in(0,1)$, horizon $K\in\mathbb{N}$.
\State Initialize shared radius $h\gets h_0$, $S_0\gets 0$, $A_0\gets 0$, $\hat x\gets x_1$, and $z_1 \gets x_1$.
\For{$k=1,2,\dots,K$}
    \State Set $\epsilon_{k}\gets \min\!\Big\{\frac12,\, \frac{1}{\sqrt{k}\,(1+\|x_k-x_1\|)}\Big\}$,
    $T_{k}\gets \left\lceil \log_2\!\Big(\frac{\pi\sqrt{d-1}}{2\epsilon_{k}}\Big)\right\rceil$ and $B_k$ as in \eqref{eq:pf-budgeted-per-iter}. 
    \State Compute normal direction estimate $\widehat n_k$ of $ x_k$ by  Algorithm~\ref{alg:OSNE} with adaptive comparison radius initialized at radius $h$, planar bisection depth $T_k$, and a budget $B_k$ of comparisons.\label{line:estimate-normal}
    \State \textbf{If} the budget is exhausted before line \ref{line:estimate-normal} terminates \textbf{then return} $\hat x$.
    \State \textbf{If} $z_k\in\mathcal{C}$ \textbf{then} $g_k\gets \widehat n_k$;  \textbf{else}  $s_k\gets \frac{z_k-x_k}{\|z_k- x_k\|}$ and $g_k\gets \widehat n_k+\max\{0,\,-\langle \widehat n_k, s_k\rangle\} \, s_k$.
    \State $S_k \gets S_{k-1} + g_k$.
    \State $A_k \gets A_{k-1} + \langle g_k,\,z_k-x_1\rangle$.
    \State $z_{k+1} \gets x_1 - \dfrac{1 - A_k}{k+1}\,S_k$.
    \State $x_{k+1} \gets \Pi_{\mathcal{C}}(z_{k+1})$.
    \State \textbf{If} $x_{k+1}\prec \hat x$ \textbf{then} set $\hat x\gets x_{k+1}$.
\EndFor
\State \textbf{return} $\hat x$.
\end{algorithmic}
\end{algorithm}

Algorithm~\ref{alg:pf-ngd-osne} is also a normal direction span-based method if the normal direction estimates are exact.
Indeed, if $z_t\in\mathcal C$ then $g_t=n_t$, while if $z_t\notin\mathcal C$ then
$g_t=n_t+\beta_t s_t$ with $s_t$ proportional to $z_t-x_t$.
Thus $g_t\in\operatorname{span}\{n_t,\ x_t-z_t\}$.
Since
$z_{t+1}=x_1-\frac{1-A_t}{t+1}\sum_{i=1}^t g_i$,
it follows that
$z_{t+1}\in x_1+\operatorname{span}\{n_1,\dots,n_t,\ x_1-z_1,\dots,x_t-z_t\}$,
which is precisely the span restriction in Definition~\ref{def:normal-direction-span-method}.
The scalar factor $\frac{1-A_k}{k+1}$ is updated online through the cumulative inner products $\langle g_i,\,z_i-x_1\rangle$.
This is exactly the KT coin-betting update (Algorithm~\ref{alg:kt-olo}) for online linear optimization written in anchored form; see, e.g., \cite{orabona2016coin}.

\begin{algorithm}[htbp]
\caption{KT coin-betting for OLO over a Hilbert space $H$ (\cite[Algorithm~1]{orabona2016coin})}
\label{alg:kt-olo}
\begin{algorithmic}[1]
\State \textbf{Input:} parameter $d_0>0$.
\For{$t=1,2,\dots$}
    \State Predict with
    $
    w_t \gets \frac{1}{t}\Big(d_0 + \sum_{i=1}^{t-1}\langle g_i, w_i\rangle\Big)\sum_{i=1}^{t-1} g_i
    $
    \Comment{empty sums for $t=1$ give $w_1=0$}
    \State Receive $g_t\in H$ such that $\|g_t\|\le 1$.
\EndFor
\end{algorithmic}
\end{algorithm}

The above KT coin-betting algorithm enjoys a sublinear regret guarantee in online linear optimization, which we restate below. 

\begin{lemma}[KT coin-betting regret bound {\cite[Corollary~5]{orabona2016coin}}]
\label{lem:kt-regret-orabona-pal}
Let $H$ be a Hilbert space and let $(g_t)_{t=1}^K\subseteq H$ satisfy $\|g_t\|\le 1$.
In online linear optimization, on each round $t$ the learner chooses $w_t\in H$, then observes $g_t$.
Then Algorithm~\ref{alg:kt-olo} guarantees that for all $u\in H$ and all $K\ge 1$,
\begin{equation}
\label{eq:kt-regret-cor5}
\sum_{t=1}^K \langle g_t,\,u-w_t\rangle
\ \le\
\|u\|\,
\sqrt{K\,
\ln\!\Bigl(1+\frac{24\,K^2\|u\|^2}{d_0^2}\Bigr)}
\;+\;
d_0\Bigl(1-\frac{1}{e\sqrt{\pi K}}\Bigr),
\qquad \forall d_0>0.
\end{equation}
\end{lemma}

With Lemma~\ref{lem:kt-regret-orabona-pal}, we can now derive an upper bound on the cumulative inner products
$\sum_{t=1}^K \langle g_t,\,z_t-x^\star\rangle$
for Algorithm~\ref{alg:pf-ngd-osne}. This idea has been also used by  \cite{orabona2023normalized} to design a parameter-free normalized gradient descent method for unconstrained convex optimization.

\begin{lemma}
\label{lem:kt-anchored-bound} 
Consider a run of Algorithm~\ref{alg:pf-ngd-osne} that completes $K$ iterations.
For every $x^\star\in\X$,  it holds that $\|g_t\| \le 1$ for all $t$ and
\begin{equation}
\label{eq:kt-anchored}
\sum_{t=1}^K \langle g_t,\,z_t-x^\star\rangle
\ \le\
\|x^\star-x_1\|\,
\sqrt{K\,
\ln\!\Bigl(1+24\,K^2\|x^\star-x_1\|^2\Bigr)}
\;+\; 1.
\end{equation}
\end{lemma}

\begin{proof}
First, $\|g_k\|\le 1$ for every $k$.
If $z_k\in\mathcal C$, then $g_k=\widehat n_k$, so this follows from $\|\widehat n_k\|=1$.
If $z_k\notin\mathcal C$, then $g_k=\widehat n_k+\beta_k s_k$ with
$\beta_k=\max\{0,-\langle \widehat n_k,s_k\rangle\}$.
If $\langle \widehat n_k,s_k\rangle\ge0$, then $g_k=\widehat n_k$; otherwise
$g_k=\widehat n_k-\langle \widehat n_k,s_k\rangle s_k$, and hence
$\|g_k\|^2=\|\widehat n_k\|^2-\langle \widehat n_k,s_k\rangle^2\le1$.

Define $w_t:=z_t-x_1\in\X$ so that $w_1=0$, and set $r_t:=-g_t$.
Then the recursion on $(w_t)$ coincides with the KT update in
Algorithm~\ref{alg:kt-olo} run with parameter $d_0=1$ and feedback $r_t$.
Apply Lemma~\ref{lem:kt-regret-orabona-pal} in $H=\X$ with comparator $u:=x^\star-x_1$ to obtain
\[
\sum_{t=1}^K \langle r_t,\,u-w_t\rangle
\ \le\
\|u\|\,
\sqrt{K\,
\ln\!\Bigl(1+24\,K^2\|u\|^2\Bigr)}
\;+\;
\Bigl(1-\frac{1}{e\sqrt{\pi K}}\Bigr)
\ \le\
\|u\|\,
\sqrt{K\,
\ln\!\Bigl(1+24\,K^2\|u\|^2\Bigr)}
\;+\;1.
\]
Finally,
\[
\sum_{t=1}^K \langle r_t,\,u-w_t\rangle
=
\sum_{t=1}^K \langle -g_t,\,(x^\star-x_1)-(z_t-x_1)\rangle
=
\sum_{t=1}^K \langle g_t,\,z_t-x^\star\rangle,
\]
which yields \eqref{eq:kt-anchored}.
\end{proof}
 
Furthermore, Lemma~\ref{lem:pf-proj-bridge} below shows that each
$\langle g_k,\,z_k-x^\star\rangle$
provides an upper bound for
$\langle \widehat n_k,\,x_k-x^\star\rangle$. This step is critical in bridging the internal iterates $z_k$ and the projected iterates $x_k$.

\begin{lemma} 
\label{lem:pf-proj-bridge}
Assume $\mathcal{C}\subseteq\X$ is nonempty, closed, and convex.
Consider any iteration $k$ of Algorithm~\ref{alg:pf-ngd-osne} at which $g_k$ is defined.
For every $x^\star\in\mathcal{C}$, it holds that
\begin{equation}
\label{eq:pf-proj-bridge}
\langle \widehat n_k,\,x_k-x^\star\rangle
\ \le\
\langle g_k,\,z_k-x^\star\rangle.
\end{equation}
\end{lemma}

\begin{proof} 
If $z_k\in\mathcal C$, then $x_k=z_k$ and $g_k=\widehat n_k$ by Algorithm~\ref{alg:pf-ngd-osne}, so \eqref{eq:pf-proj-bridge} holds with equality.
Assume now that $z_k\notin\mathcal C$. By Algorithm~\ref{alg:pf-ngd-osne}, $s_k=(z_k-x_k)/\|z_k-x_k\|$ and
$g_k=\widehat n_k+\beta_k s_k$, where $\beta_k:=\max\{0,-\langle \widehat n_k,s_k\rangle\}$.
Write $z_k=x_k+\alpha_k s_k$, where $\alpha_k:=\|z_k-x_k\|>0$.
Projection optimality gives $\langle s_k,z-x_k\rangle\le0$ for all $z\in\mathcal C$, and hence
$\langle s_k,x_k-x^\star\rangle\ge0$ for all $x^\star\in\mathcal C$.
Moreover, $\langle g_k,s_k\rangle=\langle \widehat n_k,s_k\rangle+\beta_k=\max\{\langle \widehat n_k,s_k\rangle,0\}\ge0$.
Therefore,
\[
\langle g_k,\,z_k-x^\star\rangle
=
\langle g_k,\,x_k-x^\star\rangle+\alpha_k\langle g_k,s_k\rangle
\ge
\langle \widehat n_k,\,x_k-x^\star\rangle,
\]
where the last inequality uses $g_k-\widehat n_k=\beta_k s_k$, $\beta_k\ge0$, and
$\langle s_k,x_k-x^\star\rangle\ge0$. 
This completes the proof.
\end{proof}

Additionally, it can be proven that the internal iterates $z_k$ (and $x_k=\Pi_{\mathcal C}(z_k)$) generated by Algorithm~\ref{alg:pf-ngd-osne} remain within a bounded distance from the initial point $x_1$. The boundedness of the iterates is used in analyzing the effect of the error of normal direction estimates.

\begin{lemma}
\label{lem:pf-iterate-radius}
Consider a run of Algorithm~\ref{alg:pf-ngd-osne} that completes $K$ iterations, and assume $x_k\notin\X^\star$ for all
$k\le K$ (otherwise the best-so-far output is already optimal).
Let $D_1:=\dist(x_1,\X^\star)$. Then for all $k\in\{1,\dots,K+1\}$,
\begin{equation}
\label{eq:pf-iterate-radius}
\|x_k-x_1\|\le \|z_k-x_1\|
\ \le\
1+(k-1)D_1+2\sqrt{k-1}.
\end{equation} 
\end{lemma}

\begin{proof}
Fix $x^\star\in\X^\star$ such that $\|x^\star-x_1\|=D_1$ and write $w_k:=z_k-x_1$.
Recall from Algorithm~\ref{alg:pf-ngd-osne} that $w_{k+1}=-\frac{1-A_k}{k+1}S_k$, where $S_k=\sum_{i=1}^k g_i$ and $A_k=\sum_{i=1}^k \langle g_i,\,w_i\rangle$.
By Lemma~\ref{lem:kt-anchored-bound} with comparator $x^\star=x_1$ and horizon $k$, we have $A_k\le 1$, and hence $(1-A_k)/(k+1)\ge 0$.
Moreover, $\|g_i\|\le 1$ implies $\|S_k\|\le k$. Therefore
\begin{equation}
\label{eq:pf-wk-by-scalar}
\|w_{k+1}\|
=
\frac{1-A_k}{k+1}\,\|S_k\|
\le
\frac{1-A_k}{k+1}\,k
\le
1-A_k.
\end{equation}

We now upper bound $1-A_k$.
For each $i\le k$, let $n_i:=n_{x_i}$.

First consider the case $z_i\in\mathcal{C}$.
Then $x_i=z_i$, $g_i=\widehat n_i$, and
\begin{equation}
\label{eq:pf-inner-prod-decomp}
\langle g_i,\,w_i\rangle
=
\langle n_i,\,z_i-x_1\rangle
+
\langle \widehat n_i-n_i,\,w_i\rangle.
\end{equation}
Since $x^\star\in\X^\star\subseteq \calS_{x_i}$ and $n_i$ is an outward normal at $x_i\in\partial\calS_{x_i}$, the supporting-halfspace property yields $\langle n_i,\,x_i-x^\star\rangle\ge 0$. Because $z_i=x_i$ in this case,
\begin{equation}
\label{eq:pf-inner-prod-decomp-2}
\langle n_i,\,z_i-x_1\rangle
=
\langle n_i,\,x_i-x^\star\rangle+\langle n_i,\,x^\star-x_1\rangle
\ge
-D_1.
\end{equation}
Also, Theorem~\ref{thm:OSNE-convex-oracle} and the definition of $\epsilon_i$ in Algorithm~\ref{alg:pf-ngd-osne} give $\|\widehat n_i-n_i\|\le \epsilon_i\le \frac{1}{\sqrt{i}(1+\|x_i-x_1\|)}=\frac{1}{\sqrt{i}(1+\|w_i\|)}$, so by Cauchy--Schwarz,
\begin{equation}
\label{eq:pf-inner-prod-decomp-3}
\langle \widehat n_i-n_i,\,w_i\rangle
\ge
-\|\widehat n_i-n_i\|\,\|w_i\|
\ge
-\frac{1}{\sqrt{i}}.
\end{equation}
Combining \eqref{eq:pf-inner-prod-decomp}, \eqref{eq:pf-inner-prod-decomp-2}, and \eqref{eq:pf-inner-prod-decomp-3} yields $\langle g_i,\,w_i\rangle \ge -D_1-\frac{1}{\sqrt{i}}$.

Now consider the case $z_i\notin\mathcal{C}$.
Let $s_i=(z_i-x_i)/\|z_i-x_i\|$, set $\beta_i=\max\{0,\,-\langle \widehat n_i,s_i\rangle\}$, and define $g_i=\widehat n_i+\beta_i s_i$.
By projection optimality and $x_1\in\mathcal{C}$, we have $\langle s_i,\,x_i-x_1\rangle\ge 0$ and $\langle g_i,s_i\rangle=\max\{\langle \widehat n_i,s_i\rangle,0\}\ge 0$.
Since $w_i=z_i-x_1=(x_i-x_1)+\|z_i-x_i\| s_i$, it follows that
\begin{equation}
\label{eq:pf-inner-prod-proj-case}
\langle g_i,\,w_i\rangle
=
\langle g_i,\,x_i-x_1\rangle+\|z_i-x_i\|\,\langle g_i,s_i\rangle
\ge
\langle \widehat n_i,\,x_i-x_1\rangle.
\end{equation}
Using $\langle \widehat n_i,\,x_i-x_1\rangle=\langle n_i,\,x_i-x_1\rangle+\langle \widehat n_i-n_i,\,x_i-x_1\rangle$, together with \eqref{eq:pf-inner-prod-decomp-2} and the same Cauchy--Schwarz bound as above, we obtain $\langle \widehat n_i,\,x_i-x_1\rangle\ge -D_1-\frac{1}{\sqrt{i}}$. Combining this with \eqref{eq:pf-inner-prod-proj-case} yields $\langle g_i,\,w_i\rangle \ge -D_1-\frac{1}{\sqrt{i}}$.

Hence $\langle g_i,\,w_i\rangle\ge -D_1-\frac{1}{\sqrt{i}}$ for every $i\le k$. Therefore
$
1-A_k
=
1-\sum_{i=1}^k \langle g_i,\,w_i\rangle
\le
1+kD_1+\sum_{i=1}^k \frac{1}{\sqrt{i}}
\le
1+kD_1+2\sqrt{k}.
$
Plugging this into \eqref{eq:pf-wk-by-scalar} gives $\|w_{k+1}\| \le 1+kD_1+2\sqrt{k}$.
Reindexing yields the second inequality of \eqref{eq:pf-iterate-radius}.
Finally, by nonexpansiveness of projection,
$
\|x_k-x_1\|
=
\|\Pi_{\mathcal C}(z_k)-\Pi_{\mathcal C}(x_1)\|
\le
\|z_k-x_1\|,
$
which gives the first inequality of \eqref{eq:pf-iterate-radius}.
\end{proof}

With the above lemmas in place, we are ready to state the main computational guarantees of the adaNDD method.

\begin{theorem}[Adaptive normal direction descent method]
\label{thm:pf-oracle-two-case}
Assume the preference relation $\preccurlyeq$ is convex and plateau-free.
Assume $\mathcal{C}\subseteq\X$ is nonempty, closed, and convex, and $d\ge 2$.
Assume furthermore that the preference relation is $r_x$-regular for every
$x\in\mathcal{C}\setminus\X^\star$ with $r_x>0$.
Run Algorithm~\ref{alg:pf-ngd-osne} for $K\ge 1$ iterations, and write $D_1:=\dist(x_1,\X^\star)$. The total number of comparisons has the following upper bound:
\begin{equation}
\label{eq:pf-total-comparisons-budgeted}
\begin{aligned}
N
\ \le\
2K d \, \Big\lceil  \log_2 \!\Big( 7 d^{1/2}  K^{3/2} (D_1+3)\Big)\Big\rceil
\;+\; 2K\left( \left\lceil \log_2\Big(\tfrac{ h_0 d^{5/2}}{r^\star}\Big) \right\rceil_{+}
\;+\; 2\left\lceil  \log_2\!\left(\tfrac{2K^2}{\delta}\right)\right\rceil + 9 
\right).
\end{aligned}
\end{equation} 
If the method completes $K$ iterations, then its output $\hat x$
satisfies
\begin{equation}
\label{eq:pf-ngd-osne-bound}
\DeltaLS(\hat x)
\ \le\
\frac{D_1\sqrt{\ln\!\bigl(1+24\,K^2 D_1^2\bigr)}+2D_1+3}{\sqrt{K}}.
\end{equation}
If the method stops early because the budget is exhausted at some iteration $\tau$ during the
normal-estimation call at $x_\tau$, then for every $\delta\in(0,1)$, with probability at
least $1-\delta$ (over the internal randomness of Algorithm~\ref{alg:OSNE}),
\begin{equation}
\label{eq:pf-expensive-implies-curv}
r_{x_\tau}\ <\ r^\star\,.
\end{equation}  
If, in addition, Condition~\ref{cond:radius-growth} holds with parameters $(\gamma_1,\gamma_2)$
and $r^\star \le \gamma_2$, then on the same event, 
\begin{equation}
\label{eq:pf-expensive-implies-close}
\DeltaLS(\hat x)
\ \le\
\frac{r^\star}{\gamma_1} .
\end{equation}
\end{theorem}

\begin{proof} 
We first show the proof of \eqref{eq:pf-ngd-osne-bound} and \eqref{eq:pf-total-comparisons-budgeted}, the case the method completes $K$ iterations.
 
If there exists some $k\le K+1$ such that $ x_k\in\X^\star$, then the best-so-far rule implies
$\hat x\in\X^\star$ and hence $\DeltaLS(\hat x)=0$, so \eqref{eq:pf-ngd-osne-bound} holds trivially.
We therefore assume $ x_k\notin\X^\star$ for all $k\le K+1$, so that $n_k:=n_{ x_k}$ exists and is unique.

Fix $x^\star\in\X^\star$ such that $\|x^\star-x_1\|=D_1$.
By Lemma~\ref{lem:DeltaLS-by-gap} and
Lemma~\ref{lem:DeltaLS-monotone-best},
\begin{equation}
\label{eq:pf-decomp}
\begin{aligned}
\DeltaLS(\hat x)
\ & \le\
\min_{1\le k\le K}\DeltaLS( x_k)
\ \le\
\frac{1}{K}\sum_{k=1}^K \DeltaLS( x_k)
\ \le\
\frac{1}{K}\sum_{k=1}^K \langle n_k,\, x_k-x^\star\rangle \\
\ & = \ \frac{1}{K}\underbrace{ \sum_{k=1}^K \langle \widehat n_k,\, x_k-x^\star\rangle}_{\text{denoted by $\mathrm{I}$}}
\;-\;\frac{1}{K}
\underbrace{\sum_{k=1}^K \langle \widehat n_k-n_k,\, x_k-x^\star\rangle}_{\text{denoted by $\mathrm{II}$}}
\end{aligned}
\end{equation} 
Below we derive an upper bound on $\mathrm{I}$ and a lower bound on $\mathrm{II}$ separately, and then combine them to obtain \eqref{eq:pf-ngd-osne-bound}.
 
By Lemma~\ref{lem:pf-proj-bridge}, we have
$\langle \widehat n_k,\, x_k-x^\star\rangle\le \langle g_k,\,z_k-x^\star\rangle$.
Summing and applying Lemma~\ref{lem:kt-anchored-bound} yields
\begin{equation}\label{eq:pf-kt-term}
\mathrm{I} = \sum_{k=1}^K \langle \widehat n_k,\, x_k-x^\star\rangle
\ \le\
\sum_{k=1}^K \langle g_k,\,z_k-x^\star\rangle
\ \le\
D_1\sqrt{K\ln\!\bigl(1+24\,K^2D_1^2\bigr)}+1.
\end{equation} 
 
By Theorem~\ref{thm:OSNE-convex-oracle} and the choice of $(\epsilon_k,T_k)$ in Algorithm~\ref{alg:pf-ngd-osne}, 
$\|\widehat n_k-n_k\| \le \epsilon_k \le \frac{1}{\sqrt{k}\,(1+\|x_k-x_1\|)}$.
Note that $\| x_k-x^\star\|\le \| x_k-x_1\|+\|x_1-x^\star\| \le \|x_k-x_1\|+D_1$, we obtain
\[
\big|\langle \widehat n_k-n_k,\, x_k-x^\star\rangle\big|
\le
\|\widehat n_k-n_k\|\,\| x_k-x^\star\|
\le
\frac{\|x_k-x_1\|+D_1}{\sqrt{k}\,(1+\|x_k-x_1\|)}
\le
\frac{1+D_1}{\sqrt{k}}.
\]
Summing over $k$ and using $\sum_{k=1}^K k^{-1/2}\le 2\sqrt{K}$ yields $|\mathrm{II}|\le 2\sqrt{K}(1+D_1)$.
Combining \eqref{eq:pf-decomp}, \eqref{eq:pf-kt-term}, and  $|\mathrm{II}|\le 2\sqrt{K}(1+D_1)$ gives  \eqref{eq:pf-ngd-osne-bound}.

Now we count the number of comparisons used by the method.
Each completed iteration uses one additional comparison to update the best-so-far point, and the
$k$-th normal-estimation call is allocated budget $B_k$. Therefore, $N \le K+\sum_{k=1}^K B_k$. For each $k\le K$, Lemma~\ref{lem:pf-iterate-radius} yields
$1+\|x_k-x_1\| \le K\big(\dist(x_1,\X^\star)+3\big)$.
Hence $\epsilon_k^{-1} \le K^{3/2}\big(\dist(x_1,\X^\star)+3\big)$ and $\frac{7\sqrt{d-1}}{\epsilon_k} \le 7d^{1/2}K^{3/2}\big(\dist(x_1,\X^\star)+3\big)$
and also
\[
\frac{ h_0 d^3}{r^\star\,\epsilon_k}
\ \le\
\frac{ h_0 d^3}{r^\star}\,K^{3/2}\big(\dist(x_1,\X^\star)+3\big)
\ \le\
(h_0 d^{5/2}/r^\star)\cdot \big( d^{1/2} K^{3/2}\big(\dist(x_1,\X^\star)+3\big) \big) .
\]
Moreover, since $k\le K$ and $\pi^2/6\le 2$,
$
\log_2\!\Big(\frac{\pi^2 k^2}{6\delta}\Big)
 \le
\log_2\!\Big(\frac{2K^2}{\delta}\Big).
$
Substituting these bounds into \eqref{eq:pf-budgeted-per-iter} and using
$\lceil \log_2(ab)\rceil_{+}\le \lceil \log_2(a)\rceil_{+}+\lceil \log_2(b)\rceil$ (for $b\ge 1$)
yields, for all $k\le K$,
\[
B_k
\ \le\
2d\Big\lceil \log_2\!\Big(7d^{1/2} K^{3/2}\big(\dist(x_1,\X^\star)+3\big)\Big)\Big\rceil
+
2\Big\lceil \log_2(\tfrac{h_0 d^{5/2}}{r^\star})\Big\rceil_{+}
+
16
+
4\Big\lceil \log_2\!\Big(\frac{2K^2}{\delta}\Big)\Big\rceil.
\]
Therefore, we obtain \eqref{eq:pf-total-comparisons-budgeted}.
If the method stops early at iteration $\tau\le K$, then it uses at most
$\tau+\sum_{k=1}^{\tau}B_k\le K+\sum_{k=1}^{K}B_k$ comparisons, so \eqref{eq:pf-total-comparisons-budgeted} also holds.

\medskip
Next we show the proof of   \eqref{eq:pf-expensive-implies-curv} and \eqref{eq:pf-expensive-implies-close}. Now we consider the case if the method stops early at iteration $\tau\le K$.
Fix $\delta\in(0,1)$ and define $\delta_k:=\frac{6\delta}{\pi^2k^2}$ so that
$\sum_{k=1}^{\infty}\delta_k=\delta$.
Condition on the history up to iteration $k$ (so $x_k$ and $z_k$ are fixed),
and let $N_k$ denote the number of pairwise comparisons that the unbudgeted call of
Algorithm~\ref{alg:OSNE} at $x_k$ would require before termination.
By Theorem~\ref{thm:OSNE-convex-oracle} (applied at $ x_k$ with target accuracy $\epsilon_k$ and confidence $\delta_k$),
whenever $ x_k\notin\X^\star$,
\begin{equation}
\label{eq:pf-one-step-tail}
\mathbb{P}\!\left(
N_k>\bar N_k\ \Big|\ x_k
\right)
\ \le\ \delta_k,
\end{equation}
where
\[
\bar N_k
:=
2(d-1)\left\lceil \log_2\!\Big(\frac{7\sqrt{d-1}}{\epsilon_k}\Big)\right\rceil
+
2\left\lceil \log_2\!\Big(\frac{h_0 d^3}{r_{ x_k}\,\epsilon_k}\Big)\right\rceil_{+}
+
16
+
4\left\lceil \log_2\!\Big(\frac{1}{\delta_k}\Big)\right\rceil.
\]
Now assume $r_{x_k}\ge r^\star$, then we have $\frac{h_0 d^3}{r_{x_k}\,\epsilon_k}  \le \frac{h_0 d^3}{r^\star \epsilon_k}$. 
By monotonicity of $z\mapsto \lceil\log_2(z)\rceil_{+}$, this implies $\bar N_k\le B_k$,
with $B_k$ defined in \eqref{eq:pf-budgeted-per-iter}. Therefore, $\mathbb{P}\!\left(N_k>B_k | x_k\right)  \le \delta_k$  whenever $r_{ x_k}> r^\star$.

A union bound over $k\le K$ shows that, except on an event of probability at most $\delta$,
for every $k\le K$ we have the implication: if $N_k>B_k$ then $r_{ x_k}< r^\star$.
In particular, if the algorithm stops early at iteration $\tau\le K$, then necessarily $N_\tau>B_\tau$,
and hence \eqref{eq:pf-expensive-implies-curv} holds.

Finally, assume Condition~\ref{cond:radius-growth} holds with parameters $(\gamma_1,\gamma_2)$ and
$r^\star\le \gamma_2$. On the same event we have $r_{ x_\tau}<r^\star \le\ \gamma_2$, so \eqref{eq:radius-growth} implies $\DeltaLS(x_\tau) \le  \frac{r_{x_\tau}}{\gamma_1} < \frac{r^\star}{\gamma_1}$.
Since $\hat x$ is best-so-far and $\DeltaLS(\cdot)$ is monotone under preference improvement
(Lemma~\ref{lem:DeltaLS-monotone-best}), we conclude
$\DeltaLS(\hat x)\le \DeltaLS( x_\tau)\le \frac{r^\star}{\gamma_1}$, which is \eqref{eq:pf-expensive-implies-close}.
\end{proof}

The theorem remains a two-case guarantee.
If the method completes $K$ iterations, then \eqref{eq:pf-ngd-osne-bound} gives a direct level-set optimality gap guarantee.
If it stops early, then \eqref{eq:pf-expensive-implies-curv} certifies, with probability at least $1-\delta$, that the current iterate has regularity radius smaller than $r^\star$.
This high-probability statement is inherited from the adaptive normal direction estimator.

Unlike NDD (Algorithm \ref{alg:fixed-ngd-osne}), adaNDD (Algorithm \ref{alg:pf-ngd-osne}) does not require prior knowledge of problem-dependent quantities such as the distance to the optimal solutions $\dist(x_1,\X^\star)$, the diameter of $\mathcal{C}$, or the local growth condition parameters. 
Users only need to choose the initial comparison radius $h_0$, the target regularity radius $r^\star$, and the confidence level $\delta$. 
In particular, no stepsize or comparison radius needs to be tuned.
The level-set guarantee \eqref{eq:pf-ngd-osne-bound} is unaffected by the choice of $(h_0,r^\star,\delta)$, while the total number of comparison bound \eqref{eq:pf-total-comparisons-budgeted} depends on $(h_0,r^\star,\delta)$ only through the logarithmic factors, which do not scale with $d$, and thus are often negligible compared with the leading term $O(Kd\log(Kd(D_1+3)))$ of \eqref{eq:pf-total-comparisons-budgeted}.

The guarantee \eqref{eq:pf-ngd-osne-bound} of completing $K$ iterations holds for all $K\ge 1$. The bound automatically improves as $K$ grows, without any need to preselect a target accuracy $\epsilon$. Compared with NDD with exact normal directions (see Remark \ref{rem:ngd-fixed-radius-complexity-oracle}), adaNDD achieves a similar dependence of \eqref{eq:pf-ngd-osne-bound} in $K$ with only an additional logarithmic factor.

Condition~\ref{cond:radius-growth} is used only after this point, to convert the small regularity radius certificate into the level-set optimality gap guarantee \eqref{eq:pf-expensive-implies-close}.
Choosing the parameter $r^\star$ smaller affects the comparison complexity \eqref{eq:pf-total-comparisons-budgeted} by a logarithmic factor. If Condition~\ref{cond:radius-growth} holds with parameters $(\gamma_1,\gamma_2)$, then choosing 
any $r^\star  \le \min\!\left\{ \epsilon \gamma_1,\, \gamma_2\right\}$ is enough to ensure $\DeltaLS(\hat{x})\le \epsilon$ with probability at least $1-\delta$ on the event of stopping early.

\subsection{Lower bound of normal direction span-based methods}
\label{subsec:ngd-exact-lower-bound}

The convergence rate of the level-set optimality gap in NDD and adaNDD scale as $O(\dist(x_1,\X^\star)/\sqrt{K})$ (or with a logarithmic factor) in the number of normal direction estimates $K$ (see Remark \ref{rem:ngd-fixed-radius-complexity-oracle} and Theorem \ref{thm:pf-oracle-two-case}). We now show that this dependence on $K$ is unimprovable for the broad class of normal direction span-based methods.  We will show that there exists a family of problems of the form \eqref{eq:main-optimization-problem} in which $\mathcal{C}$ is nonempty, closed and convex, and the preference relation $\preccurlyeq$ is convex, plateau-free and regular at every $x\in \mathcal{C}\setminus\X^\star$, such that no normal direction span-based method can guarantee $\DeltaLS(\hat x) <  \frac{3 \dist(x_1,\X^\star)}{4\sqrt{K+1}}$ in $K$ iterations for all these problems.

The preference relation is realized by a function $f$ defined below. We first let
\begin{equation}\label{defh}
 h(x):=\max_{1\le i\le d} x_i.
\end{equation} 
And then let $\rho:\R^d\to\R$ be a nonnegative $C^{\infty}$ function satisfying
\begin{equation}\label{defrho}
\int_{\R^d}\rho(z)\,dz=1
\quad\text{and}\quad
\{z\in\R^d:\rho(z)>0\}
\subseteq
\bigl\{z\in\R^d:\ 0<z_d<z_{d-1}<\cdots<z_1<1\bigr\}.
\end{equation} 
Set $\chi:=D_1/(4\sqrt{d})$ and define $\rho_\chi(z):=\chi^{-d}\rho(z/\chi)$, so that $\int_{\R^d}\rho_\chi(z)\,dz=1$. 
Then
$\{z:\rho_\chi(z)>0\}\subseteq
\{z\in\R^d:0<z_d<z_{d-1}<\cdots<z_1<\chi\}$.
With $h$ and $\rho_\chi$ defined, let $f$ be the convolution of $h$ with $\rho_\chi$, namely,
\begin{equation}\label{deff}
 f(x):=\int_{\R^d} h(x+z)\rho_\chi(z)\,dz.
\end{equation}
For a normal direction span-based method with $K$ iterations, we consider the problem below.
\begin{problem}\label{example:hard-instance}
    Let the problem be of the form \eqref{eq:main-optimization-problem}, in which $d = K+1$ and $\mathcal{C} = B(0,D_1):=\{x\in\R^d:\|x\|_2\le D_1\}$, and the preference relation $\preccurlyeq$ is realized by $f$ in \eqref{deff}, namely, $x\preccurlyeq y$ if and only if $f(x)\le f(y)$. Let the initial iterate $x_1$ be the origin $0$ of $\R^d$.
\end{problem}

There are some basic regularity properties of $f$.
\begin{lemma}\label{lm:property-of-f}
    Consider $f$ defined in \eqref{deff}. The function $f$ is convex and there exists $L>0$ such that $f$ is $L$-smooth on $\mathcal{C}$.
    For every point $x\in\mathcal{C}$, $\frac{1}{\sqrt{d}}\le \|\nabla f(x)\|_2\le 1$.
\end{lemma}
\begin{proof}
Since $h$ is convex and $1$-Lipschitz, and since $\rho_\chi$ is a nonnegative kernel with total mass $1$, standard properties of convolution imply that $f$ is also convex and $1$-Lipschitz. Moreover, note that the convolution can be rewritten as $f(x)=\int_{\R^d} h(y)\rho_\chi(y-x)\,dy$. From standard properties of convolution, $f$ is $C^\infty$. In addition, since $h$ is globally Lipschitz and $\rho_\chi$ is smooth and compactly supported, the second-order derivatives of $f$ are bounded on $\R^d$. Hence, $f$ is $L$-smooth for some $L>0$. 

At every point where $h$ is differentiable, its gradient is one of the basis vectors $\{e_1,\dots,e_d\}$ that correspond to a maximizing coordinate. Since $h$ is Lipschitz, it is differentiable almost everywhere, and differentiating \eqref{deff} under the integral sign yields $\nabla f(x)\in \operatorname{conv}\{e_1,\dots,e_d\}$ for all $x\in\R^d$. 
Hence $\nabla f(x)$ is a probability vector. Therefore,  for all $x\in\R^d$
\begin{equation}\label{eq:lowerboundgrad}    
\frac{1}{\sqrt d}\le \|\nabla f(x)\|_2\le 1 \ ,
\end{equation} 
and in particular $\nabla f(x)\neq 0$ everywhere in $\mathcal{C}$.
\end{proof}

Lemma \ref{lemma:welldefine} below shows that Problem \ref{example:hard-instance} belongs to the class of problems we have considered in this section.
\begin{lemma}\label{lemma:welldefine}
    Consider Problem \ref{example:hard-instance}.
    The constraint set $\mathcal{C}$ is nonempty, closed, and convex, and the preference relation $\preccurlyeq$ is convex, plateau-free, and regular at every $x\in \mathcal{C}$. The optimal solution set is on the boundary of $\mathcal{C}$ and $\dist(x_1,\X^\star) = D_1$.
\end{lemma}
\begin{proof}
    Recall the properties of $f$ established in Lemma \ref{lm:property-of-f}. Convexity of $\preccurlyeq$ follows from convexity of $f$. Since $\nabla f(x)\neq 0$ at every point, each nonoptimal function level set \(\{u:f(u)=f(x)\}\), equivalently each induced preference level set \(\calL_x\),  is a smooth hypersurface. In particular it has empty interior, so the preference is plateau-free. Also, because $f$ is $L$-smooth for some $L>0$ and $\nabla f(x)\neq 0$ for every
    $x\in\mathcal{C}$, Lemma~\ref{lem:regularityradius-hess-grad} implies that $\preccurlyeq$ is regular at every $x\in\mathcal{C}$.
    Finally, $f$ has no interior minimizer on $\mathcal{C}=B(0,D_1)$, since any interior minimizer of
     the differentiable convex function $f$ would satisfy $\nabla f(x)=0$, which is impossible. Hence every optimal point lies on $\partial B(0,D_1)$, and therefore $\dist(x_1,\X^\star)=\dist(0,\X^\star)=D_1.$
\end{proof}

Furthermore, Lemma \ref{lemma:hard-instance-properties} below shows a certain pattern of zero components of normal directions that will be key in the lower bound analysis.

\begin{lemma}\label{lemma:hard-instance-properties}
Consider Problem \ref{example:hard-instance}.
   For all $\bar{x}\in\mathcal{C}$ such that $\bar{x}_r=\bar{x}_{r+1}=\cdots=\bar{x}_d=0$ for some $r<d$, the corresponding normal directions satisfy $(n_{\bar{x}})_j=0$ for all $j>r$.
\end{lemma}

\begin{proof}
For every $j>r$ and every $y$ satisfying $\rho_\chi(y)>0$ we have
$(\bar{x}+y)_r=\bar{x}_r+y_r=y_r>y_j=\bar{x}_j+y_j=(\bar{x}+y)_j$. Here the inequality is due to the support of $\rho_\chi(\cdot)$ and the condition \eqref{defrho}. 
Therefore, for almost every such $y$, the 
$j$th component of  $\nabla h(\bar{x}+y)$ is zero. Computing the derivative of \eqref{deff}  yields for all $j > r$,  $(\nabla f(\bar{x}))_j=0$.
By Lemma~\ref{lm:property-of-f}, we have $\|\nabla f(x)\|_2>0$ for all $x\in\mathcal{C}$.
Since the preference relation is realized by the differentiable function $f$, the outward normal satisfies
$ n_x=\frac{\nabla f(x)}{\|\nabla f(x)\|_2}$ for all $x\in\mathcal{C}$.
Combining this with $(\nabla f(\bar{x}))_j=0$ for $j > r$ gives $(n_{\bar{x}})_j=0$ for all $j>r$, which completes the proof.
\end{proof}

Below we show the lower bound for normal direction span-based methods on Problem \ref{example:hard-instance}.

\begin{theorem}[Lower bound of normal direction span-based methods]
\label{thm:ngd-exact-lower-bound}
For Problem \ref{example:hard-instance}, every $K$-step normal direction span-based method (Definition \ref{def:normal-direction-span-method}) starting from $x_1 = z_1 = 0$ satisfies
\begin{equation}
\label{eq:exact-ngd-lb-main}
\DeltaLS(\hat x)
 \ge
\frac{3\,\dist(x_1,\X^\star)}{4\sqrt{K+1}} \ .
\end{equation} 
\end{theorem}

\begin{proof}

We first prove that for the iterates generated by any normal direction span-based method on this instance: for every $t=1,\dots,K+1$,
\begin{equation}
\label{eq:exact-ngd-lb-zero-chain}
 x_t,\ z_t\in\mathrm{span}\{e_1,\dots,e_{t-1}\} \ . 
\end{equation}
The base case $t=1$ is immediate since $z_1= x_1=0$. Suppose
\eqref{eq:exact-ngd-lb-zero-chain} holds for every $i=1,\dots,t$. 
Then $x_t\in\mathrm{span}\{e_1,\dots,e_{t-1}\}$, so its coordinates $t,\dots,d$ are zero. By Lemma~\ref{lemma:hard-instance-properties}, we have $n_t\in\mathrm{span}\{e_1,\dots,e_t\}$.
More generally, for every $i\le t$, the induction hypothesis gives
$x_i\in\mathrm{span}\{e_1,\dots,e_{i-1}\}$, so Lemma~\ref{lemma:hard-instance-properties} implies
$n_i\in\mathrm{span}\{e_1,\dots,e_i\}\subseteq \mathrm{span}\{e_1,\dots,e_t\}$.
Also, for every $i\le t$, the induction hypothesis implies that both $z_i$ and $x_i$ belong to
$\mathrm{span}\{e_1,\dots,e_{t-1}\}$, hence so does the correction direction $x_i-z_i$.
The update rule \eqref{eq:general-update} therefore yields
$z_{t+1}\in\mathrm{span}\{e_1,\dots,e_t\}$. Since $x_{t+1}=\Pi_{\mathcal{C}}(z_{t+1})$ is a radial rescaling of $z_{t+1}$, we also have $x_{t+1}\in\mathrm{span}\{e_1,\dots,e_t\}$.  This completes the induction and proves \eqref{eq:exact-ngd-lb-zero-chain}.

In particular, $(x_t)_d=0$ for every $t$. For any $z$ satisfying $\rho_\chi(z)>0$, due to \eqref{defrho} we have
$z_d>0$, and therefore $ h(x_t+z)\ge (x_t+z)_d=z_d>0.$
Averaging over $z$ gives
\begin{equation}
\label{eq:exact-ngd-lb-positive-values}
 f(x_t)>0
 \qquad\text{for every }t=1,\dots,K+1.
\end{equation}
Hence any output $\hat x\in\{ x_1,\dots,x_{K+1}\}$ satisfies $f(\hat x)>0$.

Now set   $y:=(-\tfrac{D_1}{\sqrt d},\dots,-\tfrac{D_1}{\sqrt d})$. Then $\|y\|_2=D_1$, so $y\in\mathcal{C}$. If
$\rho_\chi(z)>0$, then $0<z_i<\chi=\tfrac{D_1}{4\sqrt d}$ for every $i$, and therefore
\[
 h(y+z)=\max_{1\le i\le d}(y_i+z_i)\le -\tfrac{D_1}{\sqrt d}+\chi=-\frac{3D_1}{4\sqrt d}.
\]
By \eqref{deff}, we obtain $ f(y)\le -\frac{3D_1}{4\sqrt d}.$
If $f^\star:=\min_{u\in\mathcal{C}} f(u)$, then $f^\star\le f(y)$, and together with
\eqref{eq:exact-ngd-lb-positive-values} this yields
\begin{equation}
\label{eq:exact-ngd-lb-value-gap}
 f(\hat x)-f^\star\ge \frac{3D_1}{4\sqrt d}.
\end{equation}

Finally, because $f$ is globally $1$-Lipschitz, for every $u\in\calL_{\hat x}$ and every
$x^\star\in\X^\star$ we have $ f(\hat x)-f^\star=f(u)-f(x^\star)\le \|u-x^\star\|_2.$ Taking the infimum over $u\in\calL_{\hat x}$ and $x^\star\in\X^\star$ gives $ \DeltaLS(\hat x)\ge f(\hat x)-f^\star.$ Combining this with \eqref{eq:exact-ngd-lb-value-gap} and using $d=K+1$ and $\dist(x_1,\X^\star) = D_1$ proves $ \DeltaLS(\hat x)\ge \frac{3D_1}{4\sqrt{K+1}},$ as claimed.
\end{proof}

Theorem~\ref{thm:ngd-exact-lower-bound} shows that the $O(\dist(x_1,\X^\star) /\sqrt{K})$ dependence in Theorem~\ref{thm:ngd-fixed-radius-main} is sharp for normal direction span-based methods, and that Theorem~\ref{thm:pf-oracle-two-case} is sharp up to the logarithmic factor in
\eqref{eq:pf-ngd-osne-bound}. Given that this lower bound remains for exact normal direction oracles, one may conclude that if the normal directions are estimated with some error (for example, using Algorithm~\ref{alg:OSNE}), then the $O(\dist(x_1,\X^\star) /\sqrt{K})$ dependence is still unimprovable.

This hard instance is inspired by the classic lower-bound constructions for first-order methods on convex Lipschitz functions; see, e.g.,
\cite{nesterov2004introductory}. In both cases, the key mechanism is a zero chain: each query of normal direction reveals at most one new coordinate.
Our results differ in two ways. First, the guarantee here is formulated in terms of the level-set optimality gap $\DeltaLS$ rather than
in terms of a function-value gap. Second, the classical nonsmooth hard instance does not automatically fit our regularity assumption of preference level sets. The convolution \eqref{deff} using kernel $\rho_\chi$ preserves the zero-chain structure while producing a preference relation whose level sets are regular.

This instance also satisfies the local growth condition. 
Let \(L\) be a smoothness constant of \(f\) in $\mathcal{C}$. By Lemma~\ref{lm:property-of-f}, \(\|\nabla f(x)\|_2\ge 1/\sqrt d\) for all \(x \in \mathcal{C}\). This is an instance of Example~\ref{ex:curv-bounded-feasible-nonvanishing-gradient}, so Condition~\ref{cond:radius-growth} holds with, for example,
$
(\gamma_1,\gamma_2)=\left(1,\frac{1}{L\sqrt d}\right).
$

\section{Summary, remarks and future directions}
\label{sec:conclusion}
 
This paper proposes a function-free optimization framework for optimizing under preference relations. Instead of assuming an application-given objective function, we study the preference relation and its level-set geometry directly. We introduced the level-set optimality gap as a measure of optimality and the regularity radius as a measure of stationarity. Under a regularity condition on the preference relation, we showed how to estimate normal directions from pairwise comparisons with near-optimal comparison complexity. Under convexity and a local growth condition on the regularity radius, these normal direction estimates lead to NDD and adaNDD, whose dependence on the number of normal direction steps nearly matches lower bounds for normal-direction span-based methods up to logarithmic factors. The resulting comparison complexity is $\widetilde O(dD^2/\epsilon^2)$.
 
Although the paper is established under Euclidean space $\X$, we expect the same theory holds for any finite-dimensional real Hilbert space. The main objects used in the paper, including preference level sets, distances, normal directions, angles, separating halfspaces and orthonormal bases, depend only on the inner product and its induced norm, and hold also under real  Hilbert space.   
 
Several questions remain open for future research. One direction is to study NDD and adaNDD for preference relations without convexity assumptions. More broadly, it would be interesting to develop a nonconvex function-free optimization theory that does not rely on a particular objective function. For example, one may ask under what comparison-level assumptions and within how many comparisons an algorithm can guarantee a point that is preferred to some feasible point with sufficiently small regularity radius. 
 
Another direction is to study comparison complexity lower bounds for function-free optimization. The lower bounds established in this paper are for normal direction estimation and for the number of normal direction steps among normal-direction span-based methods. It remains open to determine the minimax number of pairwise comparisons required for function-free optimization, and develop a method that achieves this comparison complexity.


\section*{Statements and Declarations}
 This work was partially supported by ONR Award \#N00014-25-1-2088. 
The authors have no competing interests to declare that are relevant to the content of this article. 

\begin{small}
\bibliographystyle{plain}
\bibliography{references}
\end{small}

\end{document}